\documentclass[smallextended]{svjour3}      
\smartqed  % flush right qed marks, e.g. at end of proof
\usepackage{graphicx}
\usepackage{geometry}
\usepackage[T1]{fontenc}    % use 8-bit T1 fonts
\usepackage{booktabs}       % professional-quality tables

\usepackage{amsfonts,amsmath,amssymb,bbm}
\usepackage{enumitem}
\usepackage{graphicx}
\usepackage[table]{xcolor}
    \usepackage{adjustbox}
\usepackage[UKenglish]{isodate}
\usepackage{graphicx}
\usepackage[font=small,labelfont=it]{caption}
\usepackage{subcaption}
\usepackage{hyperref} 
    \hypersetup{
    	colorlinks = true,
    	linkcolor = blue,
    	anchorcolor = blue,
    	citecolor = blue,
    	filecolor = blue,
    	urlcolor = blue
    }
\usepackage{float}
\usepackage{enumitem}
\usepackage{multirow}
\usepackage{fancyvrb}
\usepackage{rotating}
    \usepackage{tikz}
    \usepackage{tikz-cd} 
    \pgfdeclarelayer{edgelayer}
    \pgfdeclarelayer{nodelayer}
    \pgfsetlayers{edgelayer,nodelayer,main}
    \tikzstyle{new style 0}=[fill={rgb,255: red,255; green,94; blue,247}, draw=black, shape=circle]
    \tikzstyle{pointy}=[fill=white, draw=black, shape=circle]
    \tikzstyle{pointy}=[->]

\usepackage{algpseudocode}
\usepackage[linesnumbered,lined,boxed,commentsnumbered,ruled,longend]{algorithm2e}

\SetCommentSty{mycommfont}

\newcommand{\ra}[1]{\renewcommand{\arraystretch}{#1}}
\usepackage{lscape}

\usepackage{comment}

\renewcommand{\phi}{\varphi}

\newcommand{\rr}{\mathbb{R}}

\let\emptyset\varnothing

\newcommand\sbullet[1][.5]{\mathbin{\vcenter{\hbox{\scalebox{#1}{$\bullet$}}}}}
\newcommand{\N}{\mathbb{N}}

\newcommand{\R}{\mathbb{R}}

\newcommand{\Z}{\mathbb{Z}}

\newcommand{\eqdef}{\ensuremath{\stackrel{\mbox{\upshape\tiny def.}}{=}}}

\NewDocumentCommand{\anastasis}{mo}{
    \IfValueF{#2}{
                        {{\scriptsize
                            \textcolor{violet}{ 
                            \textbf{A:}
                            \textit{{#1}}
                            }
                        }}
        }
    \IfValueT{#2}{
                        \marginnote{{\scriptsize
                            \textcolor{violet}{ 
                            \textbf{A:}
                            \textit{{#1}}
                            }
                        }}
        }
                    }
\definecolor{lightgray}{rgb}{0.83, 0.83, 0.83}

\NewDocumentCommand{\luca}{mo}{
    \IfValueF{#2}{
                        {{\scriptsize
                            \textcolor{green}{ 
                            \textbf{L:}
                            \textit{{#1}}
                            }
                        }}
        }
    \IfValueT{#2}{
                        \marginnote{{\scriptsize
                            \textcolor{green}{ 
                            \textbf{L:}
                            \textit{{#1}}
                            }
                        }}
        }
                    }
\NewDocumentCommand{\giulia}{mo}{
    \IfValueF{#2}{
                        {{\scriptsize
                            \textcolor{red}{ 
                            \textbf{GL:}
                            \textit{{#1}}
                            }
                        }}
        }
    \IfValueT{#2}{
                        \marginnote{{\scriptsize
                            \textcolor{red}{ 
                            \textbf{GL:}
                            \textit{{#1}}
                            }
                        }}
        }
}
\definecolor{shamrockgreen}{rgb}{0.0, 0.62, 0.38}
\usepackage{soul}

\definecolor{blue-green}{rgb}{0.0, 0.87, 0.87}

\begin{document}

\markboth{Luca Galimerti, Anastasis Kratsios, Giulia Livieri}{Designing Universal Causal Deep Learning Models: Infinite-Dimensional Dynamical Systems}

% \catchline{}{}{}{}{}
%
%%%%%%%%%%%%%%%%%%%%%%%%%%%%%%%%%%%%%%%%%%%%%%%%%%%%%%%%%%%%%%%%%%%%
\title{Designing Universal Causal Deep Learning Models: The Case of Infinite-Dimensional Dynamical Systems from Stochastic Analysis}
\titlerunning{Designing Universal Causal Deep Learning Models}
% \title{Designing Universal Causal Deep Learning Models: The Case of Infinite-Dimensional Dynamical Systems from Stochastic Analysis}

% \author{Luca Galimberti\footnote{King's College London, Mathematics Department, UK}}

% \address{Norwegian University of Science and Technology (NTNU),\\
% Department of Mathematics\\
% H{\o}gskoleringen 1, 7034 Trondheim, Norway \\
% luca.galimberti@ntnu.no}

% \author{Anastasis Kratsios\footnote{McMaster University, Mathematics Department, Canada}}

% \address{McMaster University,\\
% Department of Mathematics\\
% 1280 Main St W, Hamilton, ON L8S 4L8\\
% kratsioa@mcmaster.ca}

% \author{Giulia Livieri\footnote{London School of Economics, Statistics Department, UK}}

% \address{London School of Economics (LSE),
% \\
% Department of Statistics\\
% Columbia House, Houghton Street, London, WC2A 2AE\\
% g.livieri@lse.ac.uk}

	% % Finance and Stochastics
	\author{Luca Galimberti\and Anastasis Kratsios \and Giulia Livieri}
	\institute{L. Galimberti \at King's College London\\
    	Department of Mathematics\\
    	 Strand Building, Strand, London, WC2R 2LS \\
    	\email{luca.galimberti@kcl.ac.uk}
    	\and 
    	A. Kratsios 
    	\at McMaster University and The Vector Institute\\
    	Department of Mathematics\\
        1280 Main Street West, Hamilton, Ontario, L8S 4K1, Canada\\
        \email{kratsioa@mcmaster.ca}
        \and 
    	G. Livieri
    	\at London School of Economics (LSE)\\
        Department of Statistics\\
        Columbia House, Houghton Street, London, WC2A 2AE\\
        \email{g.livieri@lse.ac.uk}
    }
    
    % \date{October 22$^{th}$, 2022}

\maketitle

\begin{abstract}
Several non-linear operators in stochastic analysis, such as solution maps to stochastic differential equations, depend on a temporal structure which is not leveraged by contemporary neural operators designed to approximate general maps between Banach space.  
This paper therefore proposes an operator learning solution to this open problem by introducing a deep learning model-design framework that takes suitable infinite-dimensional linear metric spaces, e.g. Banach spaces, as inputs and returns a universal \textit{sequential} deep learning model adapted to these linear geometries specialized for the approximation of operators encoding a temporal structure.  
We call these models \textit{Causal Neural Operators}.  Our main result states that the models produced by our framework can uniformly approximate on compact sets and across arbitrarily finite-time horizons H\"older or smooth trace class operators, which causally map sequences between given linear metric spaces.  Our analysis uncovers new quantitative relationships on the latent state-space dimension of Causal Neural Operators, which even have new implications for (classical) finite-dimensional Recurrent Neural Networks. In addition, our guarantees for recurrent neural networks are tighter than the available results inherited from feedforward neural networks when approximating dynamical systems between finite-dimensional spaces.
\hfill\\
\end{abstract}

\section{Introduction}
\label{s:Introduction}

Infinite-dimensional (non-linear) dynamical systems play a central role in several sciences, especially for disciplines driven by stochastic analytic modeling. However, despite this fact, the causal neural network approximation theory for most relevant dynamical systems in stochastic analysis {\color{black} is lacking}. Indeed, we currently only comprehend neural network approximations of stochastic differential equations (SDEs) with deterministic coefficients (e.g., \cite{G2021FS}) and time-invariant random dynamical systems with the fading memory and echo state property/unique solution property (e.g., \cite{JaegerFMP,Gonon2022NNs}).  A significant problem is causal neural network approximation of \textit{solution operators} to non-Markovian SDEs. \\
\indent Moreover, the understanding of how sequential DL models work is still not fully developed, even in the classical finite-dimensional setting. For instance, the seemingly elementary empirical fact that a sequential DL model's expressiveness increases when one utilizes a high-dimensional latent state space is understood qualitatively for general dynamical systems on Euclidean spaces (as in the reservoir computing literature (e.g., \cite{Lukas_StochInputReservoir})). 

However, the \textit{quantitative} understanding of the relationship between a sequential learning model's state and its expressiveness remains an \textit{open problem}. One notable exception to this fact is the approximation of linear state-space dynamical systems by a stylized class of \textit{Recurrent Neural Networks} (RNNs, henceforth); see \cite{helmut2022metric,LiJMLR2022}.

\paragraph{\textbf{Our contribution.}} Our paper provides a simple quantitative solution to a far reaching generalization of the above problem of constructing neural network approximation of infinite-dimensional (generalized) dynamical systems on ``good'' linear metric spaces. 
More precisely, we construct a neural network approximation of any function $f$ that ``causally'' and ``regularly'' maps sequences $(x_{t_n})_{n=-\infty}^{\infty}$ to sequences $(y_{t_n})_{n=-\infty}^{\infty}$, where each $x_{t_n}$ and every $y_{t_n}$ lives in a suitable linear metric space.  In particular, we construct our causal neural network approximation framework on the following \textit{desiderata}:
\begin{enumerate}[label=(D\arabic*)]
    \item\label{itm:D1} Predictions are causal, i.e., each $y_{t_n}$ is predicted independently of $(x_{t_m})_{m>n}$.
    \item\label{itm:D2} Each $y_{t_n}$ is predicted with a small neural network specialized at time $t_n$.
    \item\label{itm:D3} Only one of these specialized networks is stored in working memory at a time.
\end{enumerate}

We first begin by describing our causal neural network model's design.  Subsequently, we will discuss our approximation theory's implications in computational stochastic analysis. 

\begin{figure}[ht!]%[H]
    \centering
    \includegraphics[width=.75\linewidth]{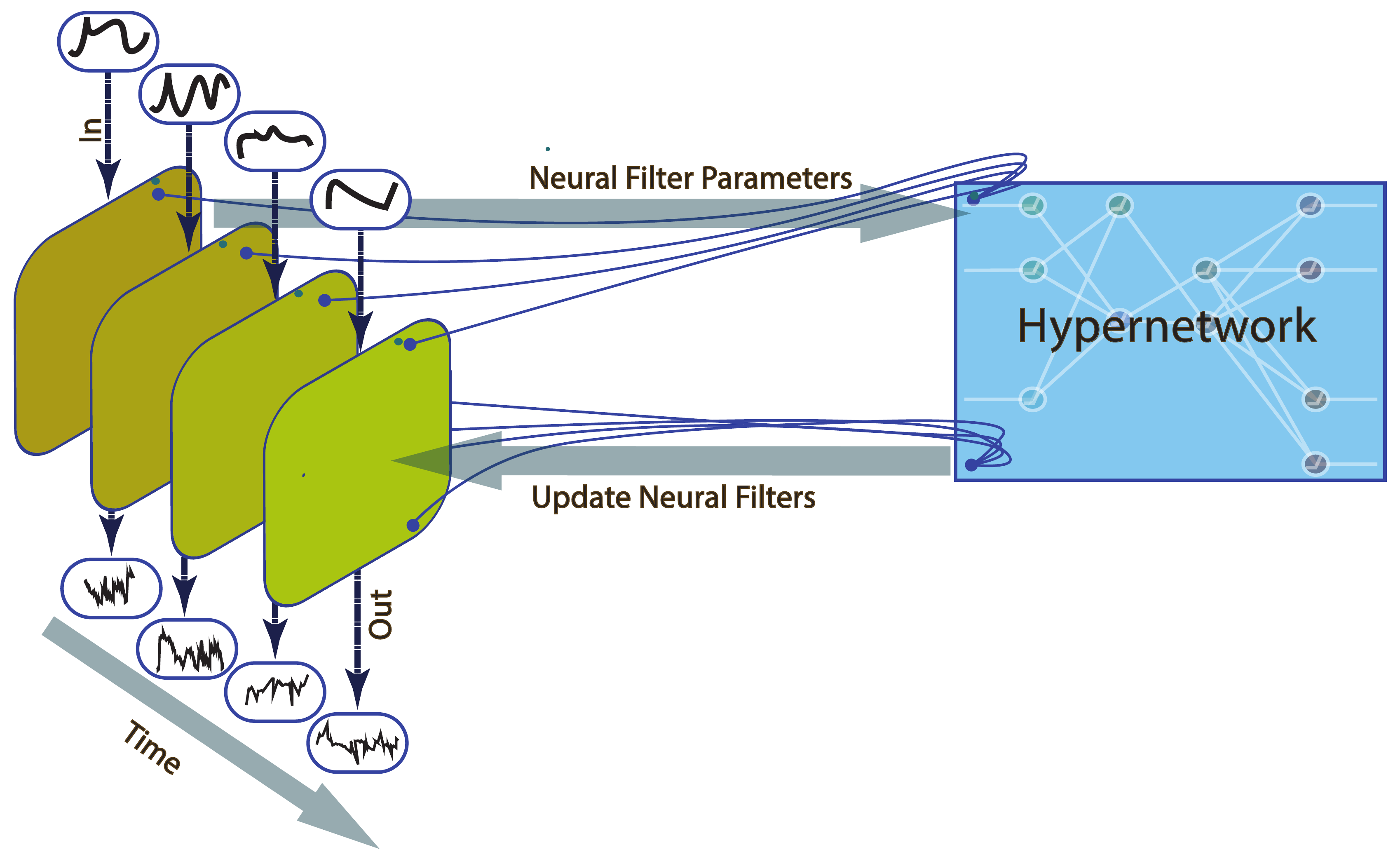}
    \caption{The Causal Neural Operator Model:\hfill\\
    \textbf{Summary:} An universal approximator of {\color{black} regular} causal sequences of operators between well-behaved Fr\'{e}chet spaces.\hfill\\
    \textbf{Overview:} The model successively applies a ``universal'' \textit{neural filter} (see Figure~\ref{fig:model_staticcase}) on consecutive time-windows; the \textit{internal parameters} of this neural filter evolve according to a latent \textit{dynamical system} on the neural filter's parameter space; implemented by a deep ReLU network called a \textit{hypernetwork}.}
    \label{fig:model_dynamiccase}
\end{figure}

\noindent Our neural network model, which we call the \textit{Causal Neural Operator} (CNO, henceforth) is illustrated in Figure~\ref{fig:model_dynamiccase} and works in the following way.  At any given time $t_n$, it predicts an instance of the output time-series at that time $t_n$ using an immediate time-window from the input time-series (e.g., it predicts each $y_{t_n}$ using only $(x_{t_i})_{i=n-10}^{n}$).  At each time $t_n$, this prediction is generated by a non-linear operator defined by a finitely parameterized neural network model, called a \textit{neural filter} (the vertical black arrows in Figure~\ref{fig:model_dynamiccase}). Our neural network model stores only one neural filter's parameters in working memory at the current time by using an auxiliary deep ReLU neural network, called a \textit{hypernetwork} in the machine learning literature (e.g., \cite{HDQICLR2017,VOHSG2020ICLR}), to generate the next neural filter specialized at $t_{n+1}$ using only the parameters of the current ``active'' neural filter specialized at time $t_n$ (the blue box in Figure~\ref{fig:model_dynamiccase}).  Thus, a dynamical system (i.e., the hypernetwork) on the neural filter's parameter space interpolating between each neural filter's parameters encodes our entire model.

\noindent The principal approximation-theoretic advantage of this approach lies in the fact that the hypernetwork is not designed to approximate anything, but rather, it only needs to \textit{memorize/interpolate} a finite number of finite-dimensional (parameter) vectors.  Since memorization (e.g., \cite{VershinynMemorization,KDDWP2022,Ruiyang}) requires only a polynomial number of parameters {\color{black} to achieve zero approximation error on a finite set}, while approximation (e.g., \cite{yarotsky2017error,KP2022JMLR,ZHS2022JMPA,BehnooshAnastasisTMLR}) requires an exponential number {\color{black} of parameters to achieve a possibly non-zero error over a large set containing the finite set of interest}, then, {\color{black} leveraging memorization yields both lighter (fewer parameters) and more accurate deep learning models}; that is, the constructed neural network model is exponentially more efficient.  {\color{black} In particular, using a neural network for memorization allows the trained DL model to generalize beyond the data it is interpolating, a capability that a simple list does not possess. When both the input and output spaces are finite-dimensional, our models effectively reduce to RNNs, which are known for their ability to generalize beyond their training data~\cite{WangGenRNN}. This generalization is attributed to factors such as having a finite VC (Vapnik-Chervonenkis) dimension~\cite{KoiSontRNN,RNNGeneralize} or finite Rademacher complexity~\cite{JRadComplexRNN}.} 
Thus, this neural network design allows us to successfully encode all the parameters required to approximate long stretches of time $\{t_0,\dots,t_N\}$ (for large $N$) with far fewer parameters (i.e., at the cost of $O(\log(N))$ additional layers in the hypernetwork).  Thus, we successfully achieve desiderata \ref{itm:D1}--\ref{itm:D3} provided that each neural filter relies on only a small number of parameters. We show that this is the case whenever $f$ is ``sufficiently smooth''; the rigorous formulation of all these outlined ideas are expressed in Lemma \ref{lemma:deterministic_weaving} and Theorem \ref{thm:theorem_universality_causal}.

\begin{figure}[ht!]%[H]
    \centering
    \includegraphics[width=.65\linewidth]{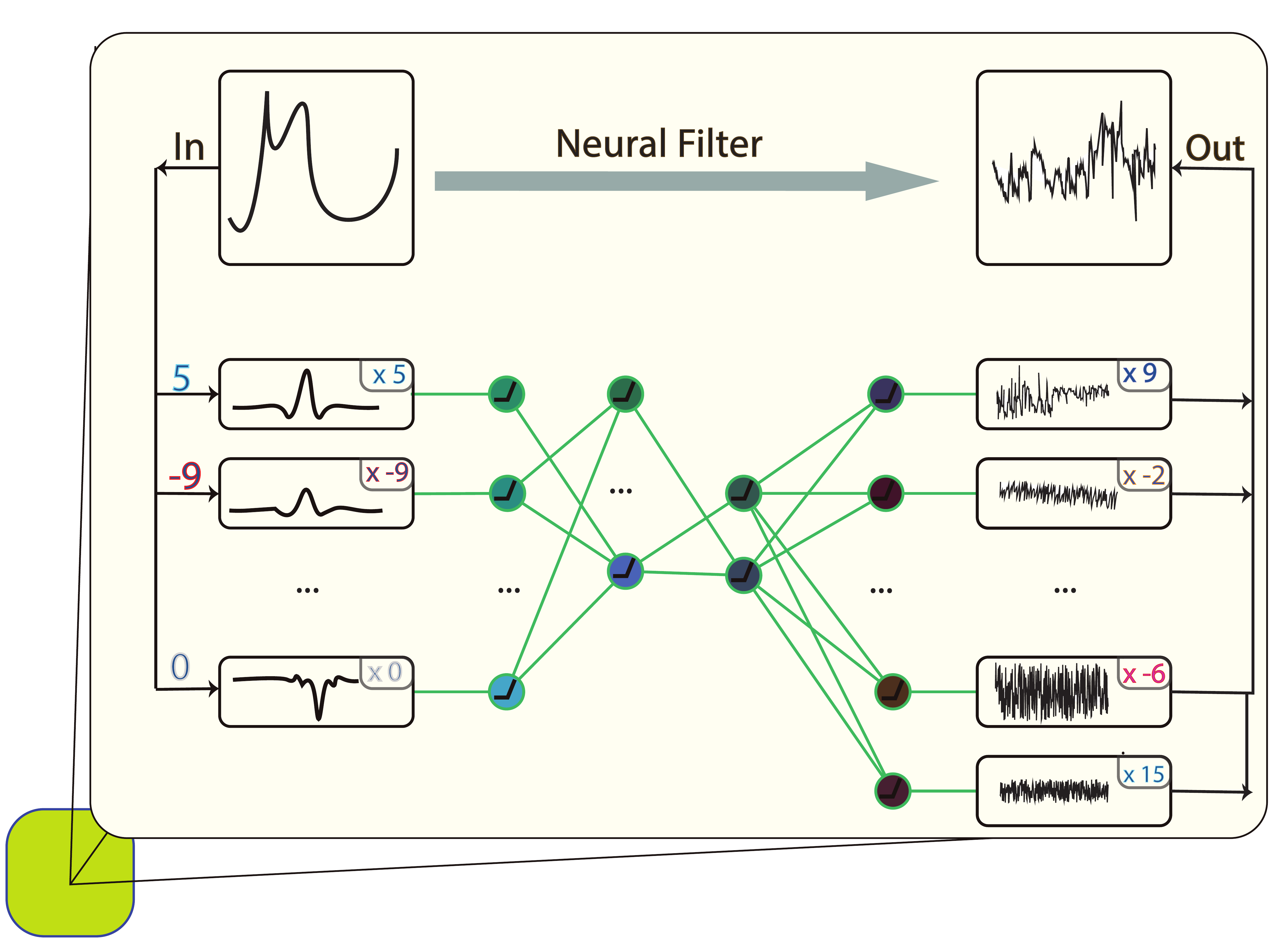}
    \caption{The Neural Filter \hfill\\
    \textbf{Summary:} An universal approximator {\color{black} of regular maps} between any well-behaved Fr\'{e}chet spaces.\hfill\\
    \textbf{Overview:} The neural filter first \textit{encodes} inputs from a (possibly infinite-dimensional) linear space by approximately representing the input as coefficients of a sparse (Schauder) basis.   These basis coefficients are then \textit{transformed} by a deep ReLU network and the network's outputs are \textit{decoded} by the coefficients of a sparse basis representation of an element of the output linear space. Assembling the basis using the outputted coefficients produces the neural filter's output.}
    \label{fig:model_staticcase}
\end{figure}

Though we are focused on the approximation theoretic properties of our modeling framework, we have designed our CNO by {\color{black} accounting for} practical considerations. Namely, we intentionally designed the CNO model so that, like transformer networks \cite{VawaniTransformers2017}, it can be trained non-recursively (via our federated training algorithm, see Algorithm~\ref{alg:train_CNO} below).  This design choice is motivated by the main reasons why the \textit{transformer network} model (e.g., \cite{VawaniTransformers2017}) has replaced residual (e.g., \cite{H2020CVPR}) and RNN (especially Long Short-Term Memory (LSTMs, henceforth) \cite{HochNeurCom1997}) counterparts in practice (e.g., \cite{Hop1982NA,Wi1986NATURE}); namely, not back-propagating through time during training. The reason is that omitting any recurrence relation between a model's prediction in sequential prediction tasks, at-least during the model's construction, has been empirically confirmed to yield more reliable and accurate models trained faster and without vanishing or exploding gradient problems; see, e.g., \cite{HochVanishing,Pa2013ICML}.  Nevertheless, our model does ultimately reap the benefits of recursive models even if we construct it non-recursively, using our parallelizable training procedure.  

The neural filter, illustrated in Figure~\ref{fig:model_staticcase}, is a \textit{neural operator} with quantitative universal approximation guarantees far beyond the Hilbert space setting. It works by first encoding infinite-dimensional problems into finite-dimensions problems.  It then predicts outputs by passing the truncated basis coefficients through a feed-forward neural network with trainable (P)ReLU activation function.  Finally, it reassembles them in the output space by interpreting the network's outputs as the coefficients of a pre-specified Schauder basis or if both spaces are reproducing kernel Hilbert spaces then the first few basis functions can learned from data using principal component analysis\footnote{Or a robust version thereof, e.g.\ \cite{RobustPCAinfinite2017} and then normalizing and orthogonalizing via Gram-Schmidt.}, e.g.\ as with PCA--Net \cite{PCANET2023}.  {\color{black} A similar encoding-MLP-decoding scheme was also used in \cite{C2023AA} for approximately solving nonlinear Kolmogorov equations on Hilbert spaces.
We also note that some infinite-dimensional deep learning models between function spaces on Euclidean domains, such as the DeepONet architecture of~\cite{DeepONet2021Nature}, replace the basis vectors with trainable deep neural networks; however, this technique does not readily apply to general Fr\'{e}chet spaces.}\\
\indent Our ``static'' approximation theorems provides quantitative approximation guarantees for several ``neural operators'' used in practice, especially in the numerical Partial Differential Equations (PDEs), e.g., \cite{Kar2021PR}, and in the inverse-problem literature, e.g., \cite{An2020IOP,Bub2021IOP,Al2021Neurips,Bub2021ISIAMIS,deHoop2020MSL}.  
In the static case, the same argument is valid also for the general qualitative (rate-free) approximation theorems of \cite{Sithcompact1996,LucaUATPaperI,Korolev2022SMA}.  
\\
\indent We now describe more in detail the different areas in which the present paper contributes.

\paragraph{Our contribution in the Approximation Theory of Neural Operators.}
Our results provide the first set of quantitative approximation guarantees for generalized dynamical systems evolving on general infinite-dimensional spaces. By refining the memorizing hypernetwork argument of \cite{AKP2022WP}, together with our general solution to the static universal approximation problem, in the class of H\"{o}lder functions\footnote{By universality here, we mean that every $\alpha$-H\"{o}lder function can be approximated by our ``static model'', for any $0<\alpha \le 1$.  NB, when all spaces are finite-dimensional then this implies the classical notion of universal approximation, formulated in \cite{hornik1989}, since compactly supported smooth functions are $1$-H\"{o}lder (i.e.\ Lipschitz) and these are dense in the space of continuous functions between two Euclidean spaces equipped with the topology of uniform convergence on compact sets.}, we are able to confirm a well-known folklore approximation of dynamical systems literature. Namely, that increasing a sequential neural operator's latent space's dimension by a positive integer $Q$ and our neural network's depth\footnote{We use $\tilde{O}$ to omit terms depending logarithmically on $Q$ and $T$.} by $\tilde{\mathcal{O}}(T^{-Q}\log(T^{-Q}))$ and width by $\tilde{\mathcal{O}}(Q T^{-Q})$ implies that we may approximate $\mathcal{O}(T)$ more time-steps in the future with the same prescribed approximation error. 

To the best of our knowledge, our dynamic result is the only quantitative universal approximation theorem guaranteeing that a recurrent neural network model can approximate any suitably regular infinite-dimensional non-linear dynamical systems.  Likewise, our static result is to the best of our knowledge the only general infinite-dimensional guarantee showing that a neural operator enjoys favourable approximation rates when the target map is smooth enough.

\paragraph{Our contribution in the Approximation Theory of RNNs}
In the finite-dimensional context, CNOs become strict sub-structures of full RNNs, where the internal parameters are updated/generated via an auxiliary hypernetwork.  Noticing this structural inclusion, our results rigorously {\color{black} support} the folklore that RNNs {\color{black} may be more suitable when approximating} causal maps, than \textit{feedforward neural network} (FFNN, henceforth), see Section~\ref{sec:Comparisons}.  This is because our theory yields expression rates for RNN approximations of causal maps between finite-dimensional spaces, which are more efficient than currently available comparable rates for FFNNs.
% \noindent \textcolor{red}{G.L.: Sanity Check: Are the previous contributions consistent with the last modifications?} \luca{What do you mean here?}

\paragraph{Technical contributions:}
Our results apply to sequences of non-linear operators between any ``good linear'' metric spaces.  By ``good linear'' metric space we mean any Fr\'{e}chet space admitting Schauder basis.  This includes many natural examples (e.g., the sequence space $\mathbb{R}^{\mathbb{N}}$ with its usual metric) outside the scope of the Banach, Hilbert\footnote{Note every separable Hilbert space carries an \textit{orthonormal} Schauder basis, so for the reader interested in Hilbert input and output spaces, we note that these conditions are automatically satisfied in that setting.} %
spaces carrying Schauder basis and Euclidean settings; which are completely subsumed by our assumptions.  In other words, we treat the most general tractable \textit{linear} setting where one can hope to obtain \textit{quantitative} universal approximation theorems. 

% \noindent \textcolor{red}{G.L.: Sanity Check: Check at the end, before submission how the paper is organized.}
\paragraph{\textbf{Organization of our paper}}
This research project answers theoretical deep learning questions by combining tools from approximation theory, functional analysis, and stochastic analysis.  Therefore, we provide a concise exposition of each of the relevant tools from these areas in our ``preliminaries'' Section~\ref{sec:Preliminaries}.\\
Section~\ref{sec:main_results} contains our quantitative universal approximation theorems.  In the static case, we derive expression rates for the static component of our model, namely the neural filters, which depend on the regularity of the target operator being approximated; from H\"{o}lder trace-class to smooth trace-class and on the usual quantities\footnote{Such as the compact set's diameter.}. Our main approximation theorem in the dynamic case additionally encodes the target causal map's memory decay rate.  

Section~\ref{sec:applications_StochAnalMathFin} applies our main results to derive approximation guarantees for the solution operators of a broad range of SDEs with stochastic coefficients, possibly having jumps (``stochastic discontinuities'') at times on a pre-specified time-grid and with initial random noise.  Section~\ref{sec:Comparisons}, examines the implication of our approximation rates for RNNs, in the finite-dimensional setting, where we find that RNNs are strictly more efficient than FFNN when approximating causal maps. Section \ref{s:Conclusion} concludes. Finally, Appendix \ref{sec:Background_for_proofs} contains any background material required in the derivations of our main results whose derivations are relegated to Appendix \ref{app:proofs} {\color{black} and Appendix~\ref{A:Additional_Backgound} contains auxiliary background material on Fr\'{e}chet spaces and generalized inverses.}

\subsection{Notation}\label{subsec:Notations}
For the sake of the reader, we collect and define here the notations we will use in the rest of the paper, or we indicate the exact point where the first appearance of a symbol occurs:
\begin{enumerate}
    \item $\mathbb{N}_{+}$\,:\,it is the set of natural numbers strictly greater than zero, i.e. $1, 2, 3, \cdots$. On the other hand, we use $\mathbb{N}$ to denote the positive integers, and $\mathbb{Z}$ to denote the integers.
    \item $[[N]]$\,:\,it denotes the set of natural numbers between $1$ and $N$, $N \in \mathbb{N}_+$, i.e. $[[N]]=\{1,\ldots, N\}$.
    \item Given a topological vector space $(F,\tau)$, $F'$ will denote its topological dual, namely the space of continuous linear forms on $F$.
    \item Given two topological vector spaces $(E,\sigma)$ and $(F,\tau)$, $L(E,F)$ denotes the space of continuous linear operators from $E$ into $F$; if $E=F$, then we will write $L(E)=L(E,E)$. 
    \item Given a Fr\'{e}chet space $F$, we use $\langle \cdot,\cdot\rangle$ to denote the canonical pairing of $F$ with its topological dual $F^{\prime}$,
    \item We denote the open ball of radius $r>0$ about a point $x$ in a metric space $(X,d)$ by $\operatorname{Ball}_{(X,d)}(x,r) \eqdef \{u\in X:\,d(x,u)<r\}$,
    \item We denote the closure of a set $A$ in a metric space $(X,d)$ by $\overline{A}$.  
    \item $\mathcal{P}, p_k$: \ref{subsec:Frechet}
    \item $\Phi$: \eqref{eq:distance_Frechet2}
    \item $\beta_k^F$ with $F$= Fr\'echet space: \eqref{eq: coordinates function}
    \item $d_{F:n}$ with $F$= Fr\'echet space: \eqref{eq:metric_on_F}
    \item $[d], P([d])$: \ref{subsec:feed_forward}
    \item $P_{F:n}, I_{F:n}$ where $F$ is a Fr\'echet space: \eqref{eq:function_P_E_n} and \eqref{eq:function_I_E_n}; furthermore, $A_{F:n}\eqdef I_{F:n}\circ P_{F:n}$
    \item $C_{tr}^{k,\lambda}(K,B)$ and $C_{\alpha, tr}^{\lambda}(K,B)$: \ref{def:Ck_traceclass} and \ref{def:holder_traceclass} 
    \item $\psi_n$ and $\phi_n$:
    \eqref{eq:setteoreticmap_psi}
    \eqref{eq:setteoreticmap_varphi}
    \item The canonical projection onto the $n^{th}$ coordinate of an $x \in \prod_{n \in \mathbb{Z}}\, \mathcal{X}_n$ is denoted by $x_n$; where each $\mathcal{X}_n$ is an arbitrary non-empty set.  \hfill\\
    In particular, if $f:A  \rightarrow \prod_{n \in \mathbb{Z}}\, \mathcal{X}_n$, with $A$ an arbitrary non-empty set, then $f(x)_n$ denotes the projection of $f(x)\in \prod_{n \in \mathbb{Z}}\, \mathcal{X}_n$ onto the $n^{th}$ coordinate,
    \item $\mathcal{NF}^{\text{(P)ReLU}}_{[n]}$: The set of neural filters from $B$ to $E$,
    \item $V$: the ``special function'', defined as the inverse of the map\footnote{The map $u\mapsto u^4\,\log_3(u+2)$ is a continuous and strictly increasing surjection of $[0,\infty)$ onto itself; whence, $V$ is well-defined.} $u\mapsto u^4\,\log_3(u+2)$ on $[0,\infty)$.  
    \item {\color{black}$f^{-}$: Generalized inverse of a real-valued increasing function $f$ on $\mathbb{R}$, see Appendix~\ref{subsec:generalized_inverse}.}
    % \textcolor{red}{G.L.: This can be delete: see Point 13.}\luca{I don't think so. But we can definitely merge it with point 13}
\end{enumerate}

\section{Preliminaries}\label{sec:Preliminaries}
In this section, we remind some preparatory material for the derivations of the main results of this paper. Finally, we remark that the notation in each of the subsequent subsections is self-contained and it is the one used on the cited paper: it will be up to the reader to contextualize it in the next sections.    
\subsection{Fr\'echet spaces}\label{subsec:Frechet}
The main references for this subsection are the following ones: \cite{H1982AMS}, Part I; \cite{C2019} Chapter IV; \cite{S1971}, Chapter III and the working paper of \cite{BONET2020WP}; all the vector spaces we will deal with will be vector spaces over $\mathbb{R}$.  
Before defining a Fr\'echet space, we remind that a \emph{locally convex topological vector space}, say $(F,\tau)$, is a topological vector space whose topology $\tau$ arises from a collection of seminorms $\mathcal{P}$. When clear from the context, we will write $F$ instead of $(F,\tau)$.  The topology is \emph{Hausdorff} if and only if for every $x \in F$ with $x \neq 0$ there exists a $p \in \mathcal{P}$ such that $p(x)>0$. On the other hand, the topology is \emph{metrizable} if and only if it may be induced by a countable collection $\mathcal P = \{p_{k}\}_{k \in \mathbb{N}_+}$ of seminorms, which we may assume to be increasing, namely $p_k(\cdot)\leq p_{k+1}(\cdot), k\in\N_+$. 
\begin{definition}[Fr\'{e}chet space]
A Fr\'{e}chet space $F$ is a complete metrizable locally convex topological vector space. 
\end{definition}
\noindent Evidently, every Banach space $(F,\|\cdot\|_F)$ is a Fr\'echet space; in this case, simply $\mathcal{P}=\{\|\cdot\|_F\}$.  
A canonical choice for the metric $d_{F}$ on a Fr\'echet space $F$ (that generates the pre-existing topology) is given by:
\begin{equation}\label{eq:distance_Frechet}
    d_{F}(x, y)  \eqdef  \sum_{k = 1}^{\infty} 2^{-k}\, \Phi(p_k(x-y)),\quad x,y\in F,
\end{equation}
where
\begin{equation}\label{eq:distance_Frechet2}
\Phi(t) \eqdef \frac{t}{1+t},\,\,\, t \geq 0.
\end{equation}

\noindent We now remind the concept of \textit{directional derivative} of a function between two Fre\'chet spaces. This notion of differentiation is significantly weaker than the concept of the derivative of a function between two Banach spaces. Nevertheless, it is the weakest notion of differentiation for which many of the familiar theorems from calculus hold. In particular, the chain rule is true (cfr.~\cite{H1982AMS}). Let $F$ and $G$ be Fr\'echet spaces, $U$ an open subset of $F$, and $P: U \subseteq F \rightarrow G$ a continuous map. 
\begin{definition}[Directional Derivative]\label{def:directional}
    The derivative of $P$ at the point $x \in U$ in the direction $h \in F$ is defined by:
    \begin{equation}\label{eq:directional_derivative}
        DP(x) h = \lim_{t \rightarrow 0} \frac{P(x + t h) - P(x)}{t}.
    \end{equation}
In particular, $P$ is said to be differentiable at $x$ in the direction $h$ if the previous limit exists. $P$ is said to be $C^1$ on $U$ if the limit in Equation\eqref{eq:directional_derivative} exists for all $x \in U$ and all $h \in F$, and $DP: (U \subseteq F) \times F \rightarrow G$ is continuous (jointly as a function on a subset of the product).  
\end{definition}
\noindent As anticipated, the Definition \ref{def:directional} of a $C^{1}$ map disagrees with the usual definition for a Banach space in the sense that the derivative will be the same map, but the continuity requirement is weaker. The previous definition can be generalized and applied to higher-order derivatives. For instance, if $P: U \subseteq F \rightarrow G$, then:
    \begin{equation}\label{eq:directional_derivative_second}
        D^2 P(x) \{h, k\} = \lim_{t \rightarrow 0} \frac{D P(x + t k)h - D P(x)h}{t}.
    \end{equation}
\noindent Analogously, $P$ is said to be $C^2$ on $U$ if $DP$ is $C^1$, which happens if and only if $D^{2}P$ exists and is continuous. If $P:U\subset F\to G$ we require $D^2P$ to be continuous jointly as a function on the product space
\[
D^2P:(U\subseteq F) \times F \times F \to G.
\]
Similarly, the $k$-th derivative $D^{k}P(x)\{h_1, h_2, \ldots, h_k\}$ will be regarded as a map
\begin{equation}\label{eq:directional_derivative_order_k}
    D^{k} P : ( U \subseteq F ) \times F \times \ldots \times F \rightarrow G.
\end{equation}
$P$ is of class $C^{k}$ on $U$ if $D^{k} P$ exists and is continuous (jointly as a function on the product space). 
\begin{remark}
We will say that $P$ is $C^{k}$-$\textrm{Dir}$ if $P$ satisfies the previous definition. 
\end{remark}
\indent Next, we introduce the concept of \emph{Schauder basis} (\cite{MV1992}).  Let $F$ be a Fr\'echet space. A sequence $(f_k)_{k \in \mathbb{N}_+} \subset F$ is called a \emph{Schauder basis} if every $x \in F$ has a unique representation
\begin{equation}\label{eq:representation_schauder}
    x = \sum_{k = 1}^{\infty} x_k f_k,
\end{equation}
where the series converges in $F$ (in the ordinary sense). It is immediate to see from the definition that the maps 
\begin{equation}\label{eq: coordinates function}
    F\ni x \stackrel{\beta_k^F}{\longmapsto}  x_k,\quad k \in \mathbb{N}_+
\end{equation}
are continuous linear functionals. We remind that if a Fr\'echet space admits a Schauder basis, it is separable. However, the converse does not hold in general; whether every separable Banach space has a basis appeared in 1931 for the first time in the Polish edition of Banach's book (\cite{B1932}) and was solved in the negative by Enflo (\cite{E1973AM}).
{\color{black} Additional background on Fr\'{e}chet spaces is included in Appendix~\ref{s:Additional_Background__Frechet}.}
\\

\subsection{Feedforward Neural Networks with ReLU and PReLU activation functions}\label{subsec:feed_forward}

We give the definition of feed-forward neural networks with ReLU activation function (ReLU FFNNs, henceforth) and with a \emph{trainable} Parametric ReLU activation function (PReLU FFNNs, henceforth). Interestingly, Proposition 1 in \cite{yarotsky2017error} shows that using a ReLU activation function is not much different from using a PReLU activation function, in the sense that it is possible to replace a ReLU FFNN with a PReLU FFNN while only increasing the number of units and weights by constant factors. However, the main advantage of using a PReLU FFNN with respect to a ReLU FFNN is that the former can \emph{synchronize the depth} of several functions realized by ReLU FFNNs, a fact that will be extremely important in the derivation of Theorem \ref{thm:theorem_universality_causal}. 
In particular, a PReLU activation function is any map $\sigma\,:\,\mathbb{R}\times\mathbb{R} \rightarrow \mathbb{R}$, $(\alpha, x) \rightarrow \sigma_{\alpha}(x) \eqdef \max\{x, \alpha x\}$; the parameter $\alpha$ is called slope. Notice that for $\alpha = 0$ one obtains the ReLU activation function. As it is customary in the literature, in what follows we will often be applying the (P)ReLU activation function component-wise. More precisely, for any $\alpha\in \R$ and an $x \in \mathbb{R}^{N}$, $N \in \mathbb{N}_{+}$, we have 
\begin{equation}\label{eq:component_wise}
    \sigma_{\alpha} \sbullet[0.77] x  \eqdef  (\sigma_{\alpha}(x_i))_{i=1}^{N}.
\end{equation}
 
\noindent Fix $J \in \mathbb{N}_{+}$ and a multi-index $[d]  \eqdef  (d_0, \ldots, d_{J})$, and let $P([d])  \eqdef  J + \sum_{j = 0}^{J-1} d_j (d_{j+1} + 1) + d_J$. Weights, biases, and slopes are identified in a unique parameter $\theta \in \mathbb{R}^{P([d])}$ with 
\begin{equation}\label{eq:identification}
    \mathbb{R}^{P([d])}\,\reflectbox{$\in$}\,\theta \iff ((A^{(j)}, b^{(j)},
        {\color{black}{\alpha^{(j)}}}
    )_{j=0}^{J-1}), c), \quad (A^{(j)}, b^{(j)},{\color{black}{\alpha^{(j)}}}) \in \mathbb{R}^{d_{j+1}\times d_j} \times \mathbb{R}^{d_{j}} {\color{black}{\times \R}},\,\,c \in \mathbb{R}^{d_{J}}.
\end{equation}
\noindent With the previous identification, the recursive representation function of a $[d]$-dimensional deep feed-forward network is given by
\begin{equation}\label{eq:feed_forward}
    \begin{split}
        \mathbb{R}^{P([d])} \times \mathbb{R}^{d_0} & \reflectbox{$\in$} (\theta, x) \rightarrow \hat{f}_{\theta}(x) \eqdef  x^{(J)} + c,\\
     x^{(j+1)}& \eqdef A^{(j)}\sigma_{{\color{black}{\alpha^{(j)}}}}\sbullet[0.75](x^{(j)} + b^{(j)})\quad\text{for\,\,}j= 0, \ldots, J-1,\\
    x^{(0)} & \eqdef  {\color{black}{A^{(0)}}}x.
    \end{split}
\end{equation}

\noindent We will refer to $J$ as $\hat{f}_{\theta}$'s \textit{depth}.  We will denote by $\mathcal{N}\mathcal{N}^{\operatorname{(P)ReLU}}_{[d]}$ a deep ReLU FFNN with \textit{complexity} $[d]$.\\

\section{Main Results}
\label{sec:main_results}

\subsection{Static Case: Universal Approximation}\label{sec:main_results__Static}
We begin by treating the ``static case'' wherein we show that CNO's \textit{neural filters}, illustrated in Figure~\ref{fig:Neural_Filter_GeneralForm}, are universal approximators of (non-linear) H\"{o}lder class operators between ``good'' linear spaces.  
We note that the application of the CNO only requires us to customize its neural filters to the relevant input and outputs' geometries.  

\begin{figure}[H]
\centering
\includegraphics[width=1\textwidth]{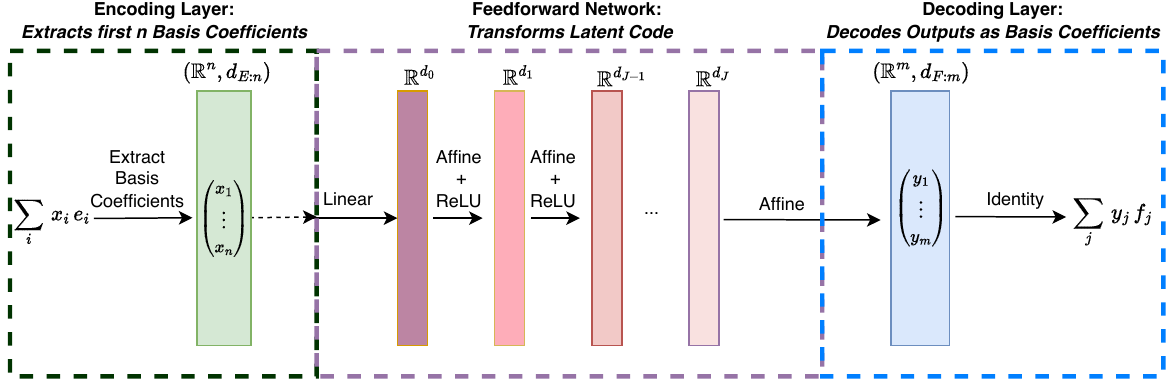}
\caption{Illustration of our ``static'' operator network in Definition~\ref{def:Neural_Filter}.  The network works in three phases.  {\color{green}{1)}} First inputs are encoded as finite-dimensional Euclidean data by mapping them to their truncated (Schauder) basis coefficients in the input space $E$.  {\color{violet}{2)}} Next these coefficients are transformed by a ReLU FFNN.  {\color{black}{3)}} The outputs of ReLU FFNN's output are interpreted as coefficients for a truncated (Schauder) basis in the output space $F$.}
\label{fig:Neural_Filter_GeneralForm}
\end{figure}

\noindent We first fix our working setting for this section
\begin{enumerate}[label=$( \mathcal{A}_{\arabic*} )$]
\item\label{itm:assumptions_one} \emph{Let $N, M \in \mathbb{N}_{+} \cup \{\infty\}$. Let $E$ and $B$ be two separable Fr\'echet spaces admitting Schauder bases $(e_h)_{h \leq N}$ and $(b_h)_{h \leq M}$. Let $E'$ and $B'$ be the topological dual of $E$ and $B$ respectively. Let $(\beta_h^{E})_{h \leq N}$ (resp. $(\beta_h^{B})_{h \leq M}$) be the unique sequence in $E'$ (resp.~$B'$) such that each $e \in E$ (resp.~each $b \in B$) has the following representation
\begin{equation*}
    e = \sum_{h = 1}^{N}\langle\beta_h^{E}, e\rangle e_h,\quad (resp.\,\,b = \sum_{h = 1}^{M}\langle\beta_h^{B}, b\rangle b_h),
\end{equation*}
where $\langle\cdot\,,\,\cdot\rangle$ is the canonical pairing between $E'$ and $E$ (resp.~between $B'$ and $B$). 
For each $n \in \mathbb{N}_{+}$, we denote by $P_{E:n}\,:\,(E, d_{E}) \rightarrow (\mathbb{R}^{n}, d_{E:n})$ the function defined as
\begin{equation}\label{eq:function_P_E_n}
    P_{E:n} : (E, d_{E}) \rightarrow (\mathbb{R}^{n}, d_{E:n}),\quad e \rightarrow (\langle \beta_{1}^{E}, e  \rangle, \langle \beta_{2}^{E}, e  \rangle, \ldots, \langle \beta_{n}^{E}, e  \rangle)^{T},
\end{equation}
where $d_{E:n}$ is the metric defined in Lemma \ref{lem:aux_lemma}. Moreover, $I_{E:n}\,:\,(\mathbb{R}^{n}, d_{E:n}) \rightarrow  (E, d_{E})$ is the function defined as  
\begin{equation}\label{eq:function_I_E_n}
    I_{E:n}\,:\,(\mathbb{R}^{n}, d_{E:n}) \rightarrow  (E, d_{E}),\quad \beta \rightarrow \sum_{h = 1}^{n} \beta_h e_h.
\end{equation}
Analogous definitions hold for $P_{B:n}$ and $I_{B:n}$.} 
\end{enumerate}

Before proceeding, we make the following trivial, yet useful remark
\begin{remark}\label{rmk:aux_remark}
    Let $F$ be a separable Fr\'echet space -- which can be either $E$ or $B$. Then, the maps $I_{F:n}$ and $P_{F:n}$ are continuous when $\mathbb{R}^{n}$ is endowed with the Euclidean topology. Therefore, they remain continuous when $\rr^{n}$ is now endowed with the metric $d_{F:n}$, because the induced topology coincides with the Euclidean one; see Lemma~\ref{lem:aux_lemma}.
\end{remark}

\indent In order to state our first approximation result, we introduce the notion of $C^{k}$-stability, $k \in \mathbb{N}$, of a non-linear operator mapping from a Fr\'echet space $E$ to a Fr\'echet space $B$. {\color{black} Notice that $C^{k}$-$\text{Dir}$ introduced in Definition \ref{def:directional} is the standard notion of directional differentiability whereas the $C^{k}$-stability formulation, although non-standard, will be useful for our approximation results. }
% \luca{Indeed: I do not understand the last sentence. It is the first notion that is related to directional differentiability, not the second one. Anastasis: why did you come up with the notion of Ck stability? Because of Teichman?}

\begin{definition}[$C^k$-Stability]\label{def:C_k_stability}
    Let $E$ and $B$ be two Fr\'echet spaces. A (non-linear) operator $f: E \rightarrow B$ is called $C^{k}$-stable if for every $m, n \in \mathbb{N}$, and every
    pair of continuous and linear maps $\tilde{I} : (\mathbb{R}^n, \|\cdot\|_2) \rightarrow (E, d_{E})$ and $\tilde{P} : (B, d_{B}) \rightarrow (\mathbb{R}^{m}, \|\cdot\|_2)$ the following composition
    \begin{equation}
        \tilde{P} \circ f \circ \tilde{I} : \mathbb{R}^{n} \rightarrow \mathbb{R}^{m},
    \end{equation}
    is of class $C^{k}$ in the usual sense.
\end{definition}
We now state and prove the following lemma.
\begin{lemma}\label{lem:lemma_Ck_stability}
    Let $E$ and $B$ be two Fr\'echet spaces. Let $f : E \rightarrow B$ be a (non-linear) operator between these two spaces which is $C^{k}$-$Dir.$ (see Subsection \ref{subsec:Frechet}, below Equation \eqref{eq:directional_derivative_order_k}).  Then, $f$ is $C^{k}$ stable as in Definition \ref{def:C_k_stability}. 
\end{lemma}
\begin{proof}
See Appendix \ref{app:proofs}, Subsection \ref{app:proof_of_lemma_Ck_stability}
\end{proof}

\noindent   The restriction of \emph{any} $C^{k}$-stable (non-linear) operator $f : E \rightarrow B$ between two Fr\'echet spaces $E$ and $B$ to \emph{any} non-empty compact subset $K \subseteq E$ extends to a $C^{k}$-stable (non-linear) operator defined on all $E$, namely the function $f$ itself. However, because our approximation theorems will hold for a \emph{pair} $(f , K )$ of a (non-linear) operator $f : E \rightarrow B$ and compact set $K$, then $f$ does not need to be smooth on $K$ but \emph{only} indistinguishable from a smooth operator on $K$. That is, our main results focus on non-linear operators belonging to the following trace class.

\begin{definition}[Trace Class {$C^{k,\lambda}_{\operatorname{tr}}( K,B)$}]\label{def:Ck_traceclass}
    Let $E$ and $B$ be two Fr\'echet spaces and let $\lambda > 0$ be a constant.  
    Let $K \subseteq E$ be a non-empty compact set. We say that a (non-linear and possibly discontinuous) operator $f :  E \rightarrow B$ belongs to the trace class $C^{k,\lambda}_{\operatorname{tr}}(K,B)$ if there exists a $\lambda$-Lipschitz\footnote{By $\lambda$-Lipschitz we mean that the optimal Lipschitz constant is $\lambda$. Notice that the case $\lambda = 0$ corresponds to the trivial case of a constant $f$ which is not treated in the present work.  } $C^k$-stable (non-linear) operator $F : E \rightarrow B$ satisfying
    \begin{equation*}
        F(x) = f(x)
    \end{equation*}
    for every $x \in K$. 
\end{definition}

\noindent The following Example \ref{example:indicator}, pictorially represented in \emph{Figure 4}, highlights our main interest in trace class maps. Precisely, these maps can be globally poorly behaved, even discontinuous, but indistinguishable from smooth functions ``locally'' (i.e.~on a particular compact subset of the input space $E$).     

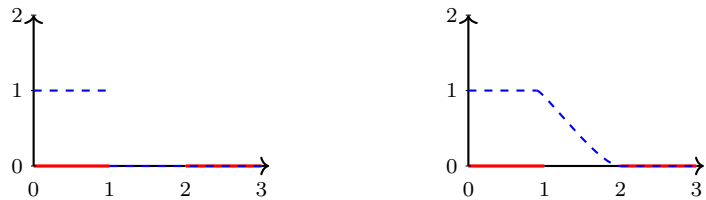
\begin{figure}[H]
\centering
\begin{subfloat}{
\centering
    \begin{tikzpicture}
        \draw[very thick,red] (0.0 , 0.0) node[black][right, below=1mm]{$0$}  -- (1.0, 0.0) node[black][right, below=1mm]{$1$};
        \draw[thick,black] (1.0 , 0.0) -- (2.0, 0.0);
        \draw[very thick,red] (2.0 , 0.0) node[black][right, below=1mm]{$2$}  -- (3.0, 0.0) node[black][right, below=1mm]{$3$};
        \draw[thick,->] (0.0 , 0.0) -- (0.0, 2.0)                   ;
        \foreach \y in {0,1,2}
        \draw (1pt,\y cm) -- (-1pt,\y cm) node[anchor=east] {$\y$}  ;
        \draw[thick,->] (3.0, 0.0) -- (3.1, 0.0)                    ;
        \draw[thick, dashed, blue] (0.0, 1.0) -- (1.0, 1.0)                 ;
        \draw[thick, dashed, blue] (1.0, 0.0) -- (3.0, 0.0)                 ;
    \end{tikzpicture}
}
\end{subfloat}
\hspace{50pt}       
\begin{subfloat}{
\centering
    \begin{tikzpicture}
        % Draw the compact in red.
        \draw[very thick,red] (0.0 , 0.0) node[black][right, below=1mm]{$0$}  -- (1.0, 0.0) node[black][right, below=1mm]{$1$};
         \draw[thick,black] (1.0 , 0.0) -- (2.0, 0.0);
        \draw[very thick,red] (2.0 , 0.0) node[black][right, below=1mm]{$2$}  -- (3.0, 0.0) node[black][right, below=1mm]{$3$};
        \draw[thick,->] (0.0 , 0.0) -- (0.0, 2.0)                   ;
        \foreach \y in {0,1,2}
        \draw (1pt,\y cm) -- (-1pt,\y cm) node[anchor=east] {$\y$}  ;
        \draw[thick,->] (3.0, 0.0) -- (3.1, 0.0)                    ;
        \draw[thick, dashed, blue] (0.0, 1.0) -- (0.9, 1.0);
        \draw[thick, dashed, blue] (0.9, 1.0) .. controls (1.0, 1.0) and (1.7, 0.0) .. (2.0, 0.0);
        \draw[thick, dashed, blue] (2.0, 0.0) -- (3.0, 0.0);
    \end{tikzpicture}
}       
\end{subfloat}
\caption{Pictorial representation of the fact that the indicator function of the interval $[0,1]$ belongs to $C^{k, \lambda}_{\operatorname{tr}}([0,1], \mathbb{R})$ for all $k \in \mathbb{N}$ and $\lambda > 0$ ; see Example \ref{example:indicator}.}
\end{figure}

\begin{example}[The indicator of the unit interval is in $C^{k,\lambda}_{\operatorname{tr}}(K,B)$]\label{example:indicator} 
    Let $E = B = (\mathbb{R},|\cdot|)$, $K = [0,1] \cup [2,3]$, and $f = I_{[0,1]}$, i.e. the indicator function of the interval $[0,1]$. Then, by means of a bump function, we immediately see that for every $k\in \mathbb{N}$ and $\lambda>0$, $f\in C^{k, \lambda}_{\operatorname{tr}}(K,B)$.  
\end{example}

\noindent At this point, some remarks are in order. In general, the problem of identifying when a map belongs to $C^{k, \lambda}_{\operatorname{tr}}( K,B)$ is a well-studied and independent area of research dating back to the beginning of the previous century (e.g., \cite{W1992}). Nonetheless, by virtue of Lemma~\ref{lem:lemma_Ck_stability} a full characterization of the pairs of functions and sets $(f, K)$ that belongs to $C^{k, \lambda}_{\operatorname{tr}}( K,B)$ in the special case that $E$ and $B$ are Euclidean spaces has been derived only (relatively) recently in a series of articles starting with \cite{F2005AOM}.  The interested reader may consult \cite{BRUE2021JFA} where the $C^{1,\lambda}_{\operatorname{tr}}( K,B)$ case is treated in the case that $B$ is Banach and $K$ is finite-dimensional (in a suitable metric-theoretic sense), for some $\lambda>0$ depending on $K$ and on $f$.  The case where $K$ is a subset of a separable Hilbert space is explicitly solved in \cite{Azagra2018JFA}.  \\

\indent Moreover, we provide results for the following trace class.
\begin{definition}[Trace Class $C_{\alpha,\operatorname{tr}}^{\lambda}(K, B)$]\label{def:holder_traceclass}
Let $E$ and $B$ be two Fr\'echet spaces, $\alpha \in (0,1]$ and $\lambda>0$ be two constants. Let $K \subseteq E$ be a non-empty compact set. We say that a (non-linear and possibly discontinuous) operator $f:E \rightarrow B$ belongs to the trace class $C_{\alpha,\text{tr}}^{\lambda}(K, B)$ if there exists an H\"older continuous (non-linear) operator $F:E \rightarrow B$ of order $\alpha$ and constant $\lambda$ satisfying
\begin{equation*}
    F(x) = f(x)
\end{equation*}
for every $x \in K$.
\end{definition} 
Functions with H\"{o}lder extensions are also actively studied.  For example, \cite[Theorem 1.12]{Lindenstrauss2000AMS} guarantees any Lipschitz function defined on a closed subset of a separable Hilbert space with values in a separable Hilbert space can be extended with the same Lipschitz constant.  However, in general, the existence of H\"{o}lder extensions between Fr\'{e}chet spaces, as well as quantitative estimates on the extension's H\"{o}lder constant, can be subtle \cite{Naor2001Mathematika}.

\noindent We state now our first main quantitative ``efficient'' approximation theorem; see Theorem \ref{thm:theorem_universality_static}. In order not to burden the statement of the theorem, we give here some definitions. First, for any $n \in \mathbb{N}_{+}$, we will use $\psi_{n}$ and $\varphi_{n}$ to denote the following two set-theoretic maps: 
\begin{equation}\label{eq:setteoreticmap_psi}
    \psi_n\,:\,(\mathbb{R}^{n}, d_{E:n})\longrightarrow (\mathbb{R}^{n},\|\cdot\|_2),\quad z\overset{\psi_n}{\longrightarrow}z,
\end{equation} 
\begin{equation}\label{eq:setteoreticmap_varphi}
\phi_n\,:\,(\mathbb{R}^{n},\|\cdot\|_2) \longrightarrow (\mathbb{R}^{n}, d_{B:n}),\quad z\overset{\varphi_n}{\longrightarrow}z.
\end{equation} 
When it is clear from the context, we suppress the index $n$ and write \textcolor{black}{$\psi$ instead of $\psi_n$ (resp.~$\varphi$ instead of $\varphi_n$)}.
Second, we introduce our first building block, which is the following neural operator, which we call a \textit{neural filter} since it filters out the part of the input not encoded in the first few Schauder basis vectors.
\begin{definition}[Neural Filters]\label{def:Neural_Filter}
Let $E$ and $B$ be two Fr\'echet spaces.  A non-linear operator $\hat{f}:E\rightarrow B$ is called a neural filter if it can be represented as
\begin{equation}
\label{eq:neuralfilter}
    \hat{f} \eqdef  I_{B:n^{out}} \circ \varphi_{n^{out}} \circ \hat{f}_{\theta} \circ \psi_{n^{in}} \circ P_{E:n^{in}}
\end{equation}
whereas: $I_{B:n^{out}}$ and $P_{E:n^{in}}$ are the functions defined in setting \ref{itm:assumptions_one}, $\psi_n$ and $\varphi_n$ are defined 
by \eqref{eq:setteoreticmap_psi} and \eqref{eq:setteoreticmap_varphi}, and $\hat{f}_{\theta} \in \mathcal{NN}_{[n]}^{\text{(P)ReLU}}$\footnote{See Subsection \ref{subsec:feed_forward}.}, with the multi-index $[n] \eqdef (d_{0}, \ldots, d_{J})$ where $d_{0} \eqdef n^{in}, d_{J} \eqdef n^{out}$ are positive integers.  The set of all neural filters with representation~\eqref{eq:neuralfilter} is denoted by $\mathcal{NF}^{\text{(P)ReLU}}_{[n]}$.  
\end{definition}

\begin{theorem}[Neural Filters {\color{black} Can} Approximate {\color{black} Regular} Non-Linear Operators]
\label{thm:theorem_universality_static}
\hfill\\
Assume setting~\ref{itm:assumptions_one}. Fix a compact subset $K \subseteq E$ with at-least two points, $k \in \mathbb{N}_{+}$, $\alpha \in (0, 1]$, $\lambda>0$ and a (non-linear) operator $f:E\rightarrow B$ belonging to either the trace-class $C^{k,\lambda}_{\operatorname{tr}}(K,B)$ or to the trace-class $C_{\alpha,\text{tr}}^{\lambda}(K, B)$. For every ``encoding error'' $\varepsilon_D>0$ and every ``approximation error'' $\varepsilon_A>0$ there exist $\hat{f} \in \mathcal{NF}^{\text{(P)ReLU}}_{[n_{\varepsilon_D}]}$ satisfying% the uniform estimate
\begin{equation}\label{eq:approximation_theorem}
    \max_{x\in K}\,
    d_{B}\big(
        f(x),
        \hat{f}(x)
        \big)
    \leq 
        \varepsilon_D 
    + 
        \varepsilon_A
,
\end{equation}
where $[n_{\varepsilon_D}]=(d_0,\dots,d_J)$ is a multi-index such with $d_0=n^{in}_{\varepsilon_D}$ and $d_J=n^{out}_{\varepsilon_D}$ defined as in Table~\ref{tab:Model_Complexity}.  {\color{black} The approximation error $\varepsilon_A$ is due to the fact that we will use approximation results for neural networks in a finite-dimensional setting\footnote{{\color{black} See Equation \eqref{eq:Proof_StaticResult_ReLUFFNN_Estimate}}}.}
\hfill\\
The ``model complexity'' of $\hat{f}_{\theta}$ is reported in Table~\ref{tab:Model_Complexity} and is a function of $f$'s regularity and the spaces $E$ and $B$.
\end{theorem}
\begin{proof}
See Appendix \ref{app:proofs}, Subsection \ref{app:proof_of_theorem_universality_static}
\end{proof}

\begin{table}[ht!]%[H]
    \centering
	\ra{1.3}
    \caption{\textbf{Optimal Approximation Rates - Neural Filter {\color{black}{with ReLU activation function}}:} The exact model complexity of the neural filter $\hat{f}$ in {Theorem~\ref{thm:theorem_universality_static}}, as a function of the target function $f$'s regularity, and the (linear) geometry of the input and output spaces $E$ and $F$. 
    \hfill\\
    \textbf{When $f$ belongs to the $C^{\lambda}_{\alpha}$-trace class:} the constants in Table~\ref{tab:Model_Complexity} are $C_1= 3^{n_{\varepsilon_D}^{in}}+3$ and $C_2 = 18 + 2\, n_{\varepsilon_D}^{in}$.  \hfill\\
    \textbf{When $f$ is belongs to the $C^{k,\lambda}$-trace class:} then $C_1 = 17 k^{n_{\varepsilon_D}^{in}+1} 3^{n_{\varepsilon_D}^{in}} n_{\varepsilon_d}^{in}$, $C_2 = 18\,k^2$, {\color{blue}{$C_3 = 85 (k+1)^{n_{\varepsilon_D}^{in}} 8^{k}$}}, and $C_{\bar{f}} = \max_{i=1,\ldots, n_{\varepsilon_D}^{in}} \|\bar{f}_{i}\|_{ C^{k}([0,1]^{n_{\varepsilon_D}^{in}})}$.
    }
    \label{tab:Model_Complexity}
    \resizebox{\columnwidth}{!}{%
    \begin{tabular}{@{}lll@{}}
    \cmidrule[0.3ex](){1-2}
	\textbf{Hyperparam.} & \textbf{Exact Quantity - High Regularity -} $\boldsymbol{C_{\operatorname{tr}}^{k,\lambda}(K,B)}$ \\
    \midrule
        $n_{\varepsilon_D}^{\text{in}}$ & 
            $
        \inf\Big\{n \in \mathbb{N}_{+}\,:\,\max_{x \in K} d_{E}(A_{E:n}(x), x) \le \frac{1}{\lambda}\omega^{\dagger}_{A,B}\left(\frac{\varepsilon_{D}}{2}\right)\Big\}$
    \\
        $n_{\varepsilon_D}^{\text{out}}$ 
        & 
        $\inf\left\{n \in \mathbb{N}_{+}\,:\,
        \underset{y \in F(K)}{\max}\,
        d_{B}( A_{B:n}(y), y)\le\frac{\varepsilon_{D}}{2}\right\}$ 
        \\
        Width & 
            $        
                n_{\varepsilon_D}^{in}(n_{\varepsilon_D}^{out}-1) 
                +
                C_1\biggl(
                        \big\lceil
                            (C_3C_{\bar{f}})^{n_{\varepsilon_D}^{in}/4k}
                                \,
                            (n_{\varepsilon_D}^{in})^{n_{\varepsilon_D}^{in}/8k}
                                \,
                            [\omega_{\phi}^{\dagger}(\varepsilon_{A})]^{-2k/n_{\varepsilon_D}^{in}}
                        \big\rceil
                    +
                        2
                \biggr)
            \cdot 
                \log_2\biggl(
                    8
                    \,
                    \big\lceil
                        (C_3C_{\bar{f}})^{n_{\varepsilon_D}^{in}/4k}
                            \,
                        (n_{\varepsilon_D}^{in})^{n_{\varepsilon_D}^{in}/8k}
                            \,
                        [\omega_{\phi}^{\dagger}(\varepsilon_{A})]^{-2k/n_{\varepsilon_D}^{in}}
                    \big\rceil
                \biggr)
            $
        \\
        Depth & 
            $
                n_{\varepsilon_D}^{out}\left(
                    1
                +
                    C_2\,
                        \biggl(
                            %L = N
                            % N - BEGIN
                            \big\lceil
                                (C_3C_{\bar{f}})^{n_{\varepsilon_D}^{in}/4k}
                                    \,
                                (n_{\varepsilon_D}^{in})^{n_{\varepsilon_D}^{in}/8k}
                                    \,
                                [\omega_{\phi}^{\dagger}(\varepsilon_{A})]^{-2k/n_{\varepsilon_D}^{in}}
                            \big\rceil
                            % N - END
                        +
                            2
                        \biggr)
                        \,
                        \log_2\biggl(
                            %L = N 
                            % N - BEGIN
                            \big\lceil
                                (C_3C_{\bar{f}})^{n_{\varepsilon_D}^{in}/4k}
                                    \,
                                (n_{\varepsilon_D}^{in})^{n_{\varepsilon_D}^{in}/8k}
                                    \,
                                [\omega_{\phi}^{\dagger}(\varepsilon_{A})]^{-2k/n_{\varepsilon_D}^{in}}
                            \big\rceil
                            % N - END
                        \biggr)
                        +
                        2
                        n_{\varepsilon_D}^{in}
                \right)
            $
    \\
    \arrayrulecolor{lightgray}
    \cmidrule[0.3ex](){1-2}
    \textbf{Hyperparam.} &  \textbf{Exact Quantity - Low Regularity -} $\boldsymbol{\operatorname{C}_{\alpha,\text{tr}}^{\lambda}(K, B)}$ 
    \\
    \arrayrulecolor{lightgray}\hline
    $n_{\varepsilon_D}^{\text{in}}$ & 
    $
    \inf\left\{n \in \mathbb{N}_{+}\,:\,\max_{x \in K} d_{E}(A_{E:n}(x), x) \le \left(\frac{1}{\lambda}\omega^{\dagger}_{A,B}\,\left(\frac{\varepsilon_{D}}{2}\right)\right)^{1/\alpha}\right\}
    $
    \\
    $n_{\varepsilon_D}^{\text{out}}$ 
            &
        $\inf\left\{n \in \mathbb{N}_{+}\,:\,
        \underset{y \in F(K)}{\max}\,
        d_{B}( A_{B:n}(y), y)\le\frac{\varepsilon_{D}}{2}\right\}$ 
    \\
        Width
        &
                $
                    n_{\varepsilon_D}^{in}\,(n_{\varepsilon_D}^{out}-1) 
                    +
                    C_1
                    \,
                    \max\biggl\{
                        n_{\varepsilon_D}^{in}
                        \,
                        \biggl\lfloor
                        \biggl(
                            [\omega_{\phi}^{\dagger}(\varepsilon_{A})]^{-n_{\varepsilon_D}^{in}/\alpha}
                            \,
                            V\big(
                                (131 \,\lambda )^{n_{\varepsilon_D}^{in}/\alpha}
                                \,
                                (n_{\varepsilon_D}^{in} n_{\varepsilon_D}^{out})^{n_{\varepsilon_D}^{in}/\alpha}
                            \big)
                        \biggr)^{1/n_{\varepsilon_D}^{in}}
                        \biggr\rfloor
                    ,
                            \biggl\lceil
                                [\omega_{\phi}^{\dagger}(\varepsilon_{A})]^{-n_{\varepsilon_D}^{in}/\alpha}
                                \,
                                V\big(
                                    (131 \,\lambda )^{n_{\varepsilon_D}^{in}/\alpha}
                                    \,
                                    (n_{\varepsilon_D}^{in} n_{\varepsilon_D}^{out})^{n_{\varepsilon_D}^{in}/\alpha}
                                \big)
                            \biggr\rceil
                        +
                            2
                    \biggr\}
                $
    \\
        Depth &  
            $
            % Cost of Deep Parallelization
            n_{\varepsilon_D}^{in}
            \left(
                    1
                + 
                % Depth of one Network
                    11\,
                    % L = N
                    \biggl\lceil
                        % N - BEGIN
                        [\omega_{\phi}^{\dagger}(\varepsilon_{A})]^{-n_{\varepsilon_D}^{in}/\alpha}
                        \,
                        V\big(
                            (131 \,\lambda )^{n_{\varepsilon_D}^{in}/\alpha}
                            \,
                            (n_{\varepsilon_D}^{in} n_{\varepsilon_D}^{out})^{n_{\varepsilon_D}^{in}/\alpha}
                        \big)
                        % N - END
                    \biggr\rceil
                    \,
                    +
                    C_2
                    \right)
            $
    \\
    \arrayrulecolor{black}
    \cmidrule[0.3ex](){1-2}
        \end{tabular}
    }% END Resize box
\end{table}

The rates in Table~\ref{thm:theorem_universality_static} are optimal for finite-dimensional Banach space input spaces and one-dimensional output space.  To see this, we only need to consider the case where $E$ is a finite-dimensional Euclidean space and $B$ is the real-line with Euclidean distance.  In this setting, neural filter model is a deep feedforward neural network with ReLU activation function.  In which case, a direct inspection of the approximation rates in Table~\ref{tab:Model_Complexity} reveal that they coincide with the approximation rates for H\"{o}lder functions derived in \cite{ZHS2022JMPA} which are optimal, as they achieve the Vapnik–Chervonenkis (VC) lower-bound on a real-valued model class' approximation rate (see \cite[Theorem 2.4]{ZHS2022JMPA}) determined by its VC-dimension\footnote{See \cite{Bartless2018JMLR} for details on the VC-dimension and near sharp computation of the VC-dimension of deep ReLU networks.}.

\begin{remark}[{Technicalities in Table~\ref{tab:Model_Complexity}}]
\label{remark:neuralFilter}
We emphasize that in the following, $\langle \cdot ,\cdot 
\rangle$ denotes the Euclidean inner product\footnote{NB, this notation coincides with our earlier use of the notation $\langle \cdot ,\cdot 
\rangle$ for the pairing of a TVS with its topological dual space by the Riesz representation theorem.}.  In particular, in the first column of Table~\ref{tab:Model_Complexity}, the functions $\bar{f}_i$ are defined by 
\[
\bar{f}_i \eqdef \langle \varphi \circ P_{B:n_{\varepsilon_D}^{out}} \circ F \circ I_{E:n_{\varepsilon_D}^{in}}\circ \psi^{-1} \circ W^{-1}, \bar{e}_i \rangle \eqdef \langle \hat{f} \circ W^{-1} ,\bar{e}_{i}\rangle
,
\]
for $i \in [[n_{\varepsilon_{D}^{in}}]]$, where the function $ W\,:\,(\mathbb{R}^{n_{\varepsilon_D}^{in}}, \|\cdot\|_2)\rightarrow (\mathbb{R}^{n_{\varepsilon_D}^{in}}, \|\cdot\|_2)$ is defined as:
\begin{equation*}
\hspace{-0.0cm}
    W\,:\,(\mathbb{R}^{n_{\varepsilon_D}^{in}}, \|\cdot\|_2)
    \rightarrow (\mathbb{R}^{n_{\varepsilon_D}^{in}}, \|\cdot\|_2) \rightarrow\mathbb{R}^{n_{\varepsilon_D}} \quad  x \rightarrow W(x) \eqdef  (2 r_{K})^{-1} (x - x_0) + \frac{1}{2}\bar{1}.
\end{equation*}
In the previous expression, we have $x_0 \in \mathbb{R}^{n_{\varepsilon_D}^{in}}$, $\bar{1} \eqdef (1, \ldots, 1) \in \mathbb{R}^{n_{\varepsilon_D}^{in}}$ and $r_K$ is a constant that depends on the compact $K$.
Moreover, in Table~\ref{tab:Model_Complexity} we use the abbreviated notation $A_{E:n} \eqdef I_{E:n_{\varepsilon_D}^{in}}\circ P_{E:n_{\varepsilon_D}^{in}}$, $A_{B:n} \eqdef I_{B:n_{\varepsilon_D}^{out}}\circ P_{B:n_{\varepsilon_D}^{out}}$, and $\omega_{A, E}$ is a modulus of continuity of the maps $(A_{E:n})_{n=1}^{\infty}$\footnote{See the proof of Theorem~\ref{thm:theorem_universality_static} for more details.} realizing the bounded approximation property on $E$ and where $\omega_{A,E}^{\dagger}$ denotes the generalized inverse\footnote{See Section \ref{subsec:generalized_inverse} for further details on generalized inverses.} of $\omega_{A,E}$.  
\end{remark}

\paragraph{Obstructions to Universal Approximation of Continuous Functions in Infinite-Dimensions}
The inability to extend higher-regularity (Lipschitz or smooth) functions while preserving their regularity, is precisely the obstruction lying at the heart of any quantitative approximation theorem between general infinite-dimensional Fr\'{e}chet spaces.  More precisely, a qualitative guarantee for continuous functions would require a version of McShane's extension theorem \cite{Beer2020McShane} for $B$-valued continuous maps but, to the best of our knowledge, such a result is only available when both $E$ and $B$ are separable Hilbert space \cite[Theorem 1.12]{Lindenstrauss2000AMS}.  However, such a result would not provide control on the target function's regularity.  Thus, without assuming that the target function belongs to a given trace-class, e.g.\ H\"{o}lder or smooth trace classes, as considered here, there is no a-priori way to clearly relate the complexity of a deep learning model, such as our neural filters, which depend on the regularity of the extension to regularity of the target function restricted to $K$.
    
Even in finite-dimensions highly-regular extensions, such as smooth extensions, see \cite{W1992} and \cite{F2005AOM}, need not exist. Moreover, it is not even clear if a uniformly continuous function can be extended to a uniformly continuous function with a proportional \textit{modulus of continuity} (see \cite{GutevJMAAExtension_2020} for details).

\subsection{Dynamic Case: {\color{black} Sequential} Universal Approximation {\color{black} Causal Operators}}\label{sec:main_results__Dynamic}

Theorem~\ref{thm:theorem_universality_static} was a static result certifying that suitable non-linear operators between infinite-dimensional linear metric spaces can be approximated by our ``neural filter'' operator network.  By training several neural filters, independently on separate time-windows, and then re-assembling then via a ``central'' \textit{hypernetwork} we can causally approximate ``any'' (generalized) dynamical system between such infinite-dimensional spaces. 

The construction of a \textit{finitely-parameterized} causal neural network approximator for these types of dynamical systems is our main result, and the main focus of this section. 
Our main result (Theorem~\ref{thm:theorem_universality_causal}) effectively certifies its ability to construct a CNO approximating any noiseless target function in this idealized approximation-theoretic framework.  
By a $\delta$-packing of a set, we mean the maximum number of points which can be placed in that set which are each at a distance of $\delta>0$ apart\footnote{See Appendix~\ref{subsec:covering_and_packing} for details.}.

We henceforth fix a \textit{non-degenerate} time grid (cfr.~Assumption 4.1 in \cite{AKP2022WP}), by which we mean a sequence $(t_{n})_{n \in \mathbb{Z}} \subseteq \mathbb{R}$, with $t_n<t_{n+1}$ for each $n\in\Z$, satisfying the following structural properties.\\ 

\noindent \textbf{Assumption 1 (Time Grid)} \textit{The time-grid $(t_n)_{n\in \mathbb{Z}}$ is assumed to satisfy 
% \luca{The referee is right: the 3rd property is implied by the others: we could write it anyway as a consequence}
\begin{enumerate}
    \item $t_0 = 0$;
    \item $0< \inf_{n \in \mathbb{Z}} \Delta t_n \leq \sup_{n \in \mathbb{Z}} \Delta t_n < \infty$;
    % \item $\inf_{n \in \mathbb{Z}} t_n = -\infty$ and $\sup_{n \in \mathbb{Z}} t_n = \infty$.\\
\end{enumerate}}
\noindent {\color{black}Note that the above assumptions imply that $\inf_{n \in \mathbb{Z}} t_n = -\infty$ and $\sup_{n \in \mathbb{Z}} t_n = \infty$.}

%\begin{assumptions}[Time Grid]
%\label{assumptions:timegrid}
%The time-grid $(t_n)_{n\in \mathbb{Z}}$ is assumed to satisfy
%\begin{enumerate}
%    \item $t_0 = 0$;
%    \item $0< \inf_{n \in \mathbb{Z}} \Delta t_n \leq \sup_{n \in \mathbb{Z}} \Delta t_n < \infty$;
%    \item $\inf_{n \in \mathbb{Z}} t_n = -\infty$ and $\sup_{n \in \mathbb{Z}} t_n = \infty$. 
%\end{enumerate}
%\end{assumptions}
In what follows, we will refer to each element $t_n$ in the non-degenerate time grid as ``time". We give now the following
\begin{definition}[Path Space]\label{def:pathspace}
    Let $(t_n)_{n \in \mathbb{Z}}$ be a fixed non-degenerate time grid. For every $n \in \mathbb{Z}$, let $E_{t_n}$ be a separable Fr\'echet space carrying a Schauder basis $(e^{(n)}_{h})_{h \in \mathbb{N}_{+}}$, and let $\mathcal{X}_{t_n}$ be a non-empty closed subset of $E_{t_n}$. The topological product $\mathcal{X} \eqdef \prod_{n \in \mathbb{Z}} \mathcal{X}_{t_n}$ is called path-space. The path space $\mathcal{X}$ is called linear if  $\mathcal{X}_{t_n} = E_{t_n},\; n\in \mathbb Z$, i.e. if $\mathcal{X} = \prod_{n \in \mathbb{Z}} E_{t_n}$.
\end{definition}
\noindent Before proceeding, we introduce the following notation. For any $n, m \in \mathbb{Z}$ with $n<m$ and $x\in \mathcal{X}$ we denote by $x_{(t_n:t_m]} \eqdef (x_{t_{n+1}}, \ldots, x_{t_{m}})$ and by $\mathcal{X}_{(t_n, t_m]} \eqdef \prod_{r=n+1}^{m}\mathcal{X}_{t_{r}}$. From Tychonoff's theorem\footnote{See Theorem 37.1 in \cite{MT2000}.} we know that an arbitrary product of compact spaces is compact in the product topology. Therefore, a path space $\mathcal{X} = \prod_{n \in \mathbb{Z}}\mathcal{X}_{t_n}$ is compact in the product topology if and only if each $\mathcal{X}_{t_n}$ is a compact subset of $E_{t_n},\;n\in\mathbb Z$. We will study \emph{causal maps} between path spaces. Briefly, what we mean with this statement is that we will analyze maps between path spaces that respect the causal forward-flow of information in time. Said differently, we will analyze maps for which, at any given time, the output must not depend on any future inputs. Because we are interested in {\color{black} quantitative} approximation results, rather than approximation guarantees via models whose number of parameters depends exponentially on the ``encoding error'' or on the ``approximation error'' (see Theorem \ref{thm:theorem_universality_static}), we will focus on the class of maps in the subsequent Definition \ref{def:causal_maps}, which are the analogue of the $C^{k,\lambda}_{\text{tr}}(K, B)$ and  $C^{\lambda}_{\alpha,\text{tr}}(K, B)$ maps introduced in Definition \ref{def:Ck_traceclass} and \ref{def:holder_traceclass}, respectively. Notice that Definition \ref{def:causal_maps} makes sense thanks to Lemma \ref{lemma: product of Schauder basis}, which states that the finite Cartesian product of Fr\'echet spaces with Schauder basis is a Fr\'echet space with a Schauder basis. 

\begin{definition}[Causal Maps of Finite Virtual Memory]
\label{def:causal_maps}
    Let $\mathcal{X} = \prod_{n \in \mathbb{Z}} \mathcal{X}_{t_n}$ be a compact path-space according to Definition \ref{def:pathspace}. Let also $\mathcal{Y}=\prod_{n \in \mathbb{Z}} B_{t_n}$ be a linear path-space; in particular, each $B_{t_n}$ is a separable Fr\'echet space with a Schauder basis. A map $f\,:\,\mathcal{X} \rightarrow \mathcal{Y}$ is called a causal map with virtual memory $r \ge 0$, if for every ``memory compression level'' $\varepsilon>0$ and each ``time-horizon'' $I\in \mathbb{N}_+$ there are $M=M(\varepsilon,I)\in\N$ with $M(\varepsilon,I)\in O(\varepsilon^{-r})$, and there are functions 
     $f_{t_i} \in C(\mathcal{X}_{(t_{i-M}, t_{i}]}, B_{t_{i}}),\, i \in [[I]]$
     satisfying
    \begin{equation}
    \label{eq:causal_maps}
        \max_{i \in [[I]]}
        \sup_{x \in \mathcal{X}} \,
            d_{B_{t_i}}(f(x)_{t_{i}}, f_{t_{i}}(x_{(t_{i-M}, t_{i}]}))<\varepsilon.
    \end{equation}
\end{definition}

Our main class of causal maps of finite virtual memory is the main deep learning model of this paper, namely, the causal neural operator outlined in Figure~\ref{fig:model_dynamiccase}.
\begin{definition}[Causal Neural Operator (CNO)]
\label{defn:CNO}
Let $\mathcal{X} = \prod_{n \in \mathbb{Z}} \mathcal{X}_{t_n}$ be a compact path-space according to Definition \ref{def:pathspace}. Let also $\mathcal{Y}=\prod_{n \in \mathbb{Z}} B_{t_n}$ be a linear path-space.  A causal map $f:\mathcal{X}\rightarrow\mathcal{Y}$ of finite virtual memory $M\ge 0$ is said to be a \textit{causal neural operator (CNO)} if: there exists a ``latent memory' $Q\in \mathbb{N}_{+}$, a multi-index $[d]$, and an ``initial latent code'' $z_0 \in \mathbb{R}^{P([d])+Q}$, a linear readout map $L\,:\,\mathbb{R}^{P([d])+Q}\rightarrow\mathbb{R}^{P([d])}$, and a (``hypernetwork'') ReLU FFNN $\hat{h}\,:\,\mathbb{R}^{P([d])+Q}\rightarrow\mathbb{R}^{P([d])+Q}$ such that the sequence of parameters $\theta_{t_i} \in \mathbb{R}^{P([d])}$, defined recursively by
\begin{equation*}
    \begin{split}
        \theta_{t_i} \eqdef & L(z_{t_i})\\
        z_{t_{i+1}} \eqdef &\begin{cases}
            \hat{h}(z_{t_i}) & \mbox{ if } t_i\ge 0 \\
            z_0 & \mbox{ if } t_i<0
        \end{cases},
    \end{split}
\end{equation*}
satisfying the representation for all $x\in \mathcal{X}$
\begin{equation*}
            f(x)_{t_n} 
        = 
            \hat{f}_{t_i}(x_{(t_{i-M}, t_i]})
\end{equation*}
where\footnote{See Definition \ref{def:Neural_Filter}.} $\hat{f}_{t_i} \in \mathcal{NF}_{[n_{\varepsilon_D}]}^{(P)ReLU}$, $
\hat{f}_{t_i} = I_{B_{t_i}:n_{\varepsilon_D^{out}}} \circ \varphi_{n_{\varepsilon_D^{out}}}\circ \hat{f}_{\theta_{t_i}} \circ \psi_{n_{\varepsilon_D^{out}}}\circ P_{E_{(t_{i-M},t_i]}:n_{\varepsilon_D}^{in}}
$ where each $\hat{f}_{\theta_{t_i}}$ is a (P)ReLU FFNN in $\mathcal{NN}_{[n_{\varepsilon_D}]}^{(P)ReLU}$ with multi-index $[n_{\varepsilon_D}]=(d_0,\dots,d_J)$ with $d_0=n^{in}_{\varepsilon_D}$ and $d_J=n^{out}_{\varepsilon_D}$.
\end{definition}

We will typically require our causal maps to possess a certain degree of regularity to deduce {\color{black} quantitative} approximation rates.  The most regular maps considered in this manuscript are causal maps of finite virtual memory which smooth trace-class maps can approximate at each instance in time. 
% \textcolor{red}{G.L.: Attention that Referee 2 pointed out that he does not understand this statement.} \luca{yes...}
\begin{definition}[Smooth Causal Maps of Finite Virtual Memory]
\label{def:causal_maps__Holder}
    Let $f:\mathcal{X}\rightarrow \mathcal{Y}$ be a causal map, in the notation of Definition~\ref{def:causal_maps}.
    If there exists a positive integer $k$ and a $\lambda>0$ such that  $f_{t_{i}} \in C_{tr}^{k,\lambda}(\mathcal{X}_{(t_{i-M}, t_{i}]}, B_{t_{i}}),\, i\in [[I]]$, then we say that the causal map $f$ is $(r, k,\lambda)$-smooth. 
    If, moreover, the functions $f_{t_i}$ belong to  $C_{tr}^{k,\lambda}
    (\mathcal{X}_{(t_{i-M}, t_{i}]}, B_{t_{i}})$ for every $k\in \mathbb{N}_+$ then we will say that $f$ is $(r,\infty,\lambda)$-smooth.
\end{definition}
We also derive approximation guarantees for the low-regularity analogue of smooth causal maps.
\begin{definition}[H\"{o}lder-Causal Maps of Finite Virtual Memory]
\label{def:causal_maps__Smooth}
    Let $f:\mathcal{X} \rightarrow \mathcal{Y}$ be a causal map, in the notation of Definition~\ref{def:causal_maps}.
    If there are an $\alpha \in (0, 1]$ and a $\lambda>0$ such that $f_{t_i} \in C_{\alpha,\text{tr}}^{\lambda}(\mathcal{X}_{(t_{i-M}, t_{i}]}, B_{t_{i}}),\, i\in [[I]]$, then we say that $f$ is $(r, \alpha, \lambda)$-H\"{o}lder.
\end{definition}

We now present the main result of the paper.  Our causal universal approximation theorem guarantees that the CNO model can approximate any causal map while ``preserving its forward flow of information through time''.  The quantitative approximation rates, describing the complexity of the CNO model implementing the approximation are recorded in Table~\ref{tab:Model_Complexity_DynamicCase} below.  

\begin{theorem}[CNOs are {\color{black} Sequential} Universal Approximators of Causal {\color{black} Operators}]
\label{thm:theorem_universality_causal}
Let $\mathcal{X} = \prod_{n \in \mathbb{Z}}\mathcal{X}_{t_n}$ be a compact path space, $\mathcal{Y} = \prod_{n \in \mathbb{Z}} B_{t_n}$ a linear path space\footnote{See Definition \ref{def:pathspace}.}, and $f : \mathcal{X}\rightarrow\mathcal{Y}$ either a $(r,k,\lambda)$-smooth or a $(r, \alpha, \lambda)$-H\"older causal map\footnote{See Definition \ref{def:causal_maps}.}.  
Fix ``hyperparameters'' $Q\in \mathbb{N}_{+}$ and $0<\delta<1$.
For every ``encoding error'' $\varepsilon_{D}>0$, every ``approximation error'' $\varepsilon_{A}>0$, and every ``time-horizon'' $I\in \mathbb{N}_+$ with $I_{\delta, Q} \eqdef \lfloor \delta^{-Q}\rfloor \ge I$ then there is an integer $M\lesssim \varepsilon_A^{-r}$, a multi-index $[d]$, a ``latent code" $z_0 \in \mathbb{R}^{P([d])+Q}$, a linear readout map $L\,:\,\mathbb{R}^{P([d])+Q}\rightarrow\mathbb{R}^{P([d])}$, and a (``hypernetwork'') ReLU FFNN $\hat{h}\,:\,\mathbb{R}^{P([d])+Q}\rightarrow\mathbb{R}^{P([d])+Q}$ such that the sequence of parameters $\theta_{t_i} \in \mathbb{R}^{P([d])}$, defined recursively by
\begin{equation*}
    \begin{split}
        \theta_{t_i} \eqdef L(z_{t_i})\\
        z_{t_{i+1}} \eqdef \hat{h}(z_{t_i}),
    \end{split}
\end{equation*}
{\color{black} $i\in \mathbb{N}$ and $z_{t_i}\eqdef z_0$ if $i<0$}% \del{with $i\in [[I]]\cup\{0\}$}
, satisfies the following uniform spatio-temporal estimate:
\begin{equation*}
    \begin{split}
    \max_{i \in 
        [[I]]
    }\,
    \sup_{x \in \mathcal{X}}\,\,
        d_{B_{t_i}}\big(
                \hat{f}_{t_i}(x_{(t_{i-M}, t_i]})
                    , 
                f(x)_{t_i}
            \big) <  \varepsilon_A + \varepsilon_D,
    \end{split}
\end{equation*}
where\footnote{See Definition \ref{def:Neural_Filter}.} $\hat{f}_{t_i} \in \mathcal{NF}_{[n_{\varepsilon_D}]}^{(P)ReLU}$, $
\hat{f}_{t_i} = I_{B_{t_i}:n_{\varepsilon_D^{out}}} \circ \varphi_{n_{\varepsilon_D^{out}}}\circ \hat{f}_{\theta_{t_i}} \circ \psi_{n_{\varepsilon_D^{out}}}\circ P_{E_{(t_{i-M},t_i]}:n_{\varepsilon_D}^{in}}
$ where each $\hat{f}_{\theta_{t_i}}$ is a (P)ReLU FFNN in $\mathcal{NN}_{[n_{\varepsilon_D}]}^{(P)ReLU}$ with multi-index $[n_{\varepsilon_D}]=(d_0,\dots,d_J)$ with $d_0=n^{in}_{\varepsilon_D}$ and $d_J=n^{out}_{\varepsilon_D}$ defined as in Table~\ref{tab:Model_Complexity}.  The model complexity of the hypernetwork $\hat{h}$ is recorded in Table~\ref{tab:Model_Complexity_DynamicCase}.  
\end{theorem}
\begin{proof}
See Appendix \ref{app:proofs}, Subsection \ref{app:proof_of_theorem_universality_causal}
\end{proof}

For brevity, we do not repeat the complexities of the neural filters approximating the target function on any time window and recall that the neural filters' approximation rates have previously been recorded in Table~\ref{tab:Model_Complexity}.  

\begin{table}[H]%[ht!]
    \centering
	\ra{1.3}
    \caption{\textbf{Causal Approximation Rates - (CNO) Causal Neural Operator:} The model complexity estimates of the hypernetwork $\hat{h}$ defining the CNO in {Theorem~\ref{thm:theorem_universality_causal}}, as a function of the target causal maps $f$'s regularity, the amount of memory allocated to the hypernetwork's latent space $Q\in \mathbb{N}_+$, and the length of the time-horizon the approximation is required to hold on $I\in \mathbb{N}_+$.}
    \label{tab:Model_Complexity_DynamicCase}
    \resizebox{\columnwidth}{!}{%
    \begin{tabular}{@{}ll@{}}
    \cmidrule[0.3ex](){1-2}
	\textbf{Hyperparam.} &  \textbf{Upper Bound}\\
    \midrule
    Width - Hyper. Net. ($\hat{h}$) & $(P([d])+Q)I_{\delta, Q} + 12$ \\
    Depth - Hyper. Net. ($\hat{h}$) & $\mathcal{O}\Biggr(
        I_{\delta, Q}\biggl(
        1+\sqrt{I_{\delta, Q}\log(I_{\delta, Q})}\,\Big(1+\frac{\log(2)}{\log(I_{\delta, Q})}\,
        \biggl[
            C + \frac{
            \Big(\log\big(I_{\delta, Q}^2\,2^{1/2}\big)-\log(\delta)\Big)
            }{\log(2)}
        \biggr]_+\Big)
        \biggr)
        \Biggr) $\\
    N. Param. - Hyper. Net. ($\hat{h}$) & 
    $
    \mathcal{O}\Biggl(I_{\delta, Q}^3(P([d])+Q)^2\,\left(1+(P([d])+Q)
    % &
    \sqrt{I_{\delta, Q}\log(I_{\delta, Q})}\,\left(1+\frac{\log(2)}{\log(I_{\delta, Q})}\,\left[C_d+
    % \\
    % &
    \frac{\Big(\log\big(I_{\delta, Q}^2\,2^{1/2}\big)-\log(\delta)\Big)}{\log(2)}\right]_+\right)\right)\Biggr),
    $
    \\
    \arrayrulecolor{lightgray}\hline
    \\
    Memory - Neural Filters ($M$) & $\mathcal{O}(\varepsilon_A^{-r})$ \\
    Complexity - Neural Filters & Table~\ref{tab:Model_Complexity}
    \\
    \arrayrulecolor{lightgray}\hline
    \\
    Constant ($C_d$) & $(P([d])+Q)I_{\delta, Q} + 12$
    \\
		\bottomrule
	\end{tabular}
	}% END Resize box
\end{table}

\subsubsection{Approximation of Smooth Causal Maps Between Hilbert Spaces on Structured Compact Sets}
\label{sec:main_results__Dynamic___HilbertCase}
The complexity bounds of the CNO model, guaranteed by Theorems~\ref{tab:Model_Complexity} and~\ref{thm:theorem_universality_causal} concern the approximation of relatively general functions between rather general Fr\'{e}chet spaces for arbitrary compact path spaces.  In particular, in this setting, the target function may be incompatible with the compact path space. {\color{black} However, in the case of Hilbert spaces, one can identify classes of compact sets over which a given smooth function can be efficiently approximated.} 
% \textcolor{red}{Check the phrasing.} We recall that, any orthonormal basis in a Hilbert space is a normalized Schauder basis.
As one may expect, all rates becomes much simpler if the all involved quantities are more structured. {\color{black} The following set of assumptions illustrates a broad class of compact sets where the CNO does not experience the curse of dimensionality, and a favorable choice of a Schauder basis becomes evident.} Furthermore, the bounds in Tables~\ref{tab:Model_Complexity} and~\ref{tab:Model_Complexity_DynamicCase} become notably simpler.
{\color{black} We now motivate our definition of these well-behaved compact sets.  We start by considering the following finite-dimensional example.}
{\color{black}
\begin{example}
\label{ex:Fin_dim}
Let $E=L^2([0,1])$ and consider the (orthonormal) Fourier basis $(e_j\eqdef e^{\mathrm{i} 2\pi j})_{j=0}^{\infty}$, where $\mathrm{i}^2=-1$.  
Fix a ``maximal frequency'' $J\in\mathbb{N}_+$, and let $\mathcal{X}\subset \operatorname{span}(\{
e^{\mathrm{i} 2\pi j}\}_{j=1}^I)$ be compact, and fix $\rho>0$.  
Set $C \eqdef e^{2\rho J}\,
\max_{x\in \mathcal{X}}\,
 \|x\| 
$.  
For any $j\in \mathbb{N}$, we have $\langle x,e_j\rangle=0$ if $j>J$ and 
$|\langle x,e_j\rangle|
\le 
\|x\|e^{2\rho j}e^{-2\rho j}\le Ce^{-2\rho j}
$ otherwise; whence, for all $j\in \mathbb{N}$
\[
    |\langle x,e_j \rangle|  
% \le 
%     \begin{cases}
%         \|x\|e^{2\rho i}e^{-2\rho i}
%         & \mbox{ if } i \le N
%     & 
%         0 & \mbox{ if } i >N
%     .
%     \end{cases}
\le 
    C e^{-2\rho j}
.
\]
\end{example}
Our well-behaved compact sets are an infinite-dimensional extension of our finite-dimensional thought experiment, in Example~\ref{ex:Fin_dim}, where we require that the coefficients of the higher-order basis vectors decay exponentially rapidly. 
Before formally defining them, let us continue the previous example
\begin{example}
\label{ex:Fin_dim__extended}
In the setting of Example~\ref{ex:Fin_dim}, let $\tilde{X}\subseteq L^2([0,1])$ consist of all $x\in L^2([0,1])$ with representation
\[
    x(t) 
    = 
        \underbrace{
            \sum_{j=0}^J\, \beta_j e^{\mathrm{i} 2\pi j}
        }_{\text{Element of }\mathcal{X}}
    +
        \underbrace{
            \sum_{j=J+1}^{\infty}\, \beta_j e^{\mathrm{i} 2\pi j}
        }_{\text{Small higher-order frequencies}}
\]
where $(\beta_j)_{j=0}^{\infty}\in \ell^2$, with $\sum_{j=0}^J\, \beta_j e^{\mathrm{i} 2\pi j}\in \mathcal{X}$ and, for each $j\ge J$, $|\beta_j|\le C e^{-2\rho j}$.  
Thus, the elements of $\tilde{X}$ are (not necessarily unique) ``extensions'' of elements of $\mathcal{X}$ by added rapidly decaying higher-order frequencies.  Moreover, for each $x\in \tilde{X}$, by construction new have
% \luca{why $t_n$? Besides, $i$ can be mistaken for the imaginary unit}
\begin{equation}
\label{eq:nat_decay_exp}
        |\langle x,e_{j}\rangle|
    \le 
        C e^{-2 \rho j}
.
\end{equation}
By the Grothendieck's compactness principle, see~\cite[Exercises 1.6]{DJSeq2012}, the set $\mathcal{X}$ is relatively compact in $L^2([0,1])$. 
% \luca{Why do we need all this construction? Can't we simply take from the beginning a set whose elements satisfy \[
%         |\langle x,e_j\rangle |
%     \le 
%         C e^{-2 \rho j }
% .
% \] }
\end{example}
We abstract Example~\ref{ex:Fin_dim__extended} into the following generally applicable condition.  An additional example of the exponential decay condition in~\eqref{eq:nat_decay_exp}, which we now generalize, will be provided in the context of mathematical finance, and will later be given in Section~\ref{s:Examples__ss:Finance___sss:PricingL2} below.
We additionally ask that our causal maps being approximated have a Markov-like property, in the sense that they only depend on the current state of the input sequence and not on the past.}\\ 

\noindent \textbf{Assumption 2 (Structured Case)} \textit{Fix constant $C
% ,\rho
>0$.  
Consider the setting of Definition~\ref{def:causal_maps} and suppose that $E_{t_n}$ and $B_{t_n}$, for each $n\in \mathbb{Z}_+$ are separable infinite-dimensional Hilbert spaces, whose inner products we denote by $\langle \cdot,\cdot\rangle_{E_{t_n}}$ (resp.~$\langle \cdot,\cdot\rangle_{B_{t_n}}$).  
For each $n\in \mathbb{Z}$% \textcolor{red}{$\mathbb{Z}_+$?} {\textcolor{green}{nein}}
, consider orthonormal basis $\{e_{n,i}\}_{i=0}^{\infty}$ of $E_{t_n}$ and $\{b_{n,i}\}_{i=0}^{\infty}$ of $B_{t_n}$.  
Fix an $(r,\infty,\lambda)$-smooth causal map $f:\prod_{n\in \mathbb{Z}}\,E_{t_n}\rightarrow\mathcal{Y}\eqdef 
% \prod_{n\in \mathbb{Z}}
% \,B_{t_n}
{\color{black} \mathbb{R}^{\mathbb{Z}}}
%\luca{Why that?}
$ with $M=M(\varepsilon,I)=0$ for each memory compression level $\varepsilon>0$ and each time-horizon $I\in \mathbb{N}_+$ in Definition~\ref{def:causal_maps__Holder}.
Consider a compact path space $\mathcal{X}\eqdef \prod_{n\in \mathbb{Z}}\,\mathcal{X}_{t_n}\subset \prod_{n\in \mathbb{Z}} E_{t_n}$ satisfying the following: there exists a constant $C>0$ such that for each $i\in \mathbb{N}$,
for all $n\in \mathbb{Z}$ and every $x\in \mathcal{X}$ 
\begin{equation}
\label{eq:f_compatability_condition}
        % \sum_{i=0}^{\infty}\,
        %     e^{\rho\,i}
        %     \,
        %     \langle x_{t_n},e_{n,i}\rangle_{E_{t_n}}^2
        % \le \frac1{C}
        {\color{black}
            |\langle x_{t_n},e_{n,i}\rangle_{E_{t_n}}|
            \le 
            % \frac{\sqrt{C|\ln(\max\{1,t\}^{-2}) - 2|}}{t^t}
            C e^{-2 \rho i }\\ \\
        }
    % \quad\mbox{ and }\quad
    %     % \sum_{i=0}^{\infty}\,
    %     %     e^{\rho\,i}
    %     %     \,
    %     %     \langle f(x)_{t_n},b_{n,i}\rangle_{B_{t_n}}^2
    %     % \le \frac1{C}
    %     {
    %         \langle f(x)_{t_n},b_{n,i}\rangle_{B_{t_n}}
    %         \le
    %         % \frac{\sqrt{C|\ln(\max\{1,t\}^{-2}) - 2|}}{t^t}
    %         \sqrt{C} e^{-\rho i /2}
    %     }
\end{equation}}
\vspace{0.05cm}

\begin{corollary}[Breaking the Curse of Dimensionality in the Structured Case]
\label{cor:dynamiccase__regular_subcase}
In the setting of Theorem~\ref{thm:theorem_universality_causal}, suppose that $\mathcal{X}$, $f$, and $\mathcal{Y}$ satisfy Assumption 2.  For every ``total approximation error $0<\varepsilon<1$'', every pair of ``hyperparameters'' $Q\in \mathbb{N}_+$ and $0<\delta<1$, and {\color{black} M=0}, and {\color{black} every compact path space $\mathcal{X}$ with $C=\mathcal{O}(\varepsilon)$, there is a} multi-index $[d]$, a ``latent code" $z_0 \in \mathbb{R}^{P([d])+Q}$, a linear readout map $L\,:\,\mathbb{R}^{P([d])+Q}\rightarrow\mathbb{R}^{P([d])}$, a (``hypernetwork'') ReLU FFNN $\hat{h}\,:\,\mathbb{R}^{P([d])+Q}\rightarrow\mathbb{R}^{P([d])+Q}$, a sequence of parameters $\theta_{t_i} \in \mathbb{R}^{P([d])}$, $\hat{f}_{t_i}$, and $
% I_{\delta,Q}
I\eqdef \lfloor \delta^{-Q}\rfloor
% \ge I
$ are as in Theorem~\ref{thm:theorem_universality_causal} satisfying 
% \textcolor{red}{G.L.: Why $x_{t_i}$?}
\[
    \begin{split}
        \max_{i \in 
            [[I]]
        }\,
        \sup_{x \in \mathcal{X}}\,\,
            \big\|
                \hat{f}_{t_i}(x_{
                % (t_{i-M}, 
                {\color{black} t_i }
                % ]
                })
                    -
                f(x)_{t_i}
            \big\|_{B_{t_i}} 
    <  
        \varepsilon
    .
    \end{split}
\]
Furthermore, the following complexity estimates hold for each neural filter $\hat{f}_{t_i}$:
\begin{enumerate}
    \item[(i)] \textbf{Encoding Dimension:} $n_{\varepsilon_D}^{\text{in}} $
    {\color{black} $
    \in 
    % \mathcal{O}(A + \ln(1/\varepsilon_D))
    \tilde{\mathcal{O}}(1)
    % \ln\big(\varepsilon^{-\rho/2}\big)+\ln(\lambda^{\rho/2}C^{\rho/2})
    $}, 
    % for some $A>0$.}
    \item[(ii)] \textbf{Decoding Dimension:} $n_{\varepsilon_D}^{\text{out}}  {\color{black}= 1}$,
    % \ln\big(\varepsilon^{-\rho/2}\big)+\ln(C^{\rho/2})$,
    \item[(iii)] \textbf{N. Params:} $P([d]) \in 
    % {\mathcal{O}(\varepsilon^{-3/2}\log(1/\varepsilon)^3)}
     {\color{black}\tilde{\mathcal{O}}(\varepsilon^{-3/2})}
    $.
\end{enumerate}
% where $\bar{\lambda}\eqdef \max\{1,\lambda\}$.
 {\color{black}Moreover, the number of parameters defining the hypernetwork is at most $\tilde{\mathcal{O}}(\sqrt{\varepsilon^{-9} 
% I
\delta^{-Q}
})$.
}
\end{corollary}
The parameter $\rho$ appears only multiplicative, and up to additive polylogarithmic factors, in the total parameter estimate (iii) in Corollary~\ref{cor:dynamiccase__regular_subcase}.  Therefore, it is suppressed by the $\tilde{\mathcal{O}}$ notation.  In particular, the curse of dimensionality has been avoided in this setting.

 {\color{black}
\begin{remark}[{Variant of Corollary~\ref{cor:dynamiccase__regular_subcase} in the Low Regularity Setting}]
\label{rem:_cor:dynamiccase__regular_subcase__LowRegVersion}
Corollary~\ref{cor:dynamiccase__regular_subcase} is stated for $(r,\infty,\lambda)$-smooth causal maps, but a similar result can also be obtained for causal maps that exhibit Lipschitz regularity by appropriately adjusting the proof. The primary difference is that the constant $C$ in inequality~\eqref{eq:f_compatability_condition} would decrease at a much faster rate along with the ``total approximation error'' $\varepsilon > 0$. This adjustment is necessary for the CNO to maintain dimension-free algebraic approximation rates in the associated compact path space.  A similar technique was recently applied in the \textit{static} low-regularity setting between Sobolev spaces, as noted in~\cite[Theorem 1]{KraLassasTakashi}. Thus, while the shape of the compact path spaces $\mathcal{X}$ regarding their exponential decay coefficient remains unchanged, the dependence on the diameter—indicated by the constant $C$—is what varies.
\end{remark}
}

\subsubsection{Discussion: How the CNO could be implemented}
\label{sec:Algorithmic_Implementation}

% {The proof of Theorem~\ref{thm:theorem_universality_causal} shows that, assuming access to idealized optimizers, we} can \textit{algorithmically} construct such approximating causal neural operators in the idealized setting, familiar to universal approximation theory \cite{KidgerLyons2022ICLR,yarotsky2017error,KP2022JMLR}, where one has complete access to a target function evaluated at all points in the input space, unobscured by any noise, as as well as a perfect optimization algorithm which can always identify a minimizer to any optimization problem.  In this idealized setting, where we can distill the approximation theoretic capabilities of our DL model apart from optimization or statistical learning question, we are able to explain its construction algorithmically.  
% {This idealized algorithm, Algorithm~\ref{alg:train_CNO}, is presented following our main result.}
{\color{black} This paper mainly examines the approximation capabilities of infinite-dimensional RNN architectures, specifically our CNO. We discuss what these structures can approximate when provided with sufficient noiseless data and ideal training algorithms. However, a natural question arises regarding their practical implementation. To address this, we present an idealized training procedure in Algorithm~\ref{alg:train_CNO}, which serves as a guide for implementing the CNO.}

\begin{algorithm}[ht]%[H]
\caption{Construct CNO}\label{alg:train_CNO}
\begin{algorithmic}
\SetAlgoLined
\Require Causal map $f: \mathcal{X} \rightarrow \mathcal{Y}$, errors: encoding $\varepsilon_D>0$ and approximation $\varepsilon_A>0$, hyperparameters: latent code complexity $Q\in \mathbb{N}_+$ and depth hyperparameter $\delta>0$.
\DontPrintSemicolon
    \State \tcc{Initialize CNO's hyperparameters}
        \State Viable time-steps: $I_{\delta, Q} \eqdef \lfloor \delta^{-Q}\rfloor$ 
        \State Memory: $M = \mathcal{O}(\varepsilon_A^{-r})$ 
        \State Set $[d]$ as in {Table~\ref{tab:Model_Complexity_DynamicCase}} 
        \State Get $\delta$-packing $\{u_i\}_{i=0}^I$ of $\overline{\operatorname{Ball}_{\mathbb{R}^Q}(0,1)}$ \tcp*{Optimally initial neural filter parameters}
    \State \SetKwBlock{ForParallel}{For $1 \le i\le I_{\delta, Q}$ in parallel}{end}
        \ForParallel{\State $\hat{f}_{\theta_{t_i}} 
        \leftarrow \underset{\hat{f}\in \mathcal{N}\mathcal{N}^{\operatorname{(P)ReLU}}_{[d]}}{\operatorname{argmin}}\,
                d_{B_{t_i}}(\hat{f}_{t_i}\big(
                    x_{(t_{i-M}, t_i]}
                    \big), f(x)_{t_i}) 
                <  
                    \varepsilon_A + \varepsilon_D
        $ \tcp*{Optimize neural filters}
        \State $z_{t_i} \leftarrow (\theta_{t_i},u_{i})$
        \tcp*{Ensure separation of neural filters' parameters}
        }
    \State \tcc{Learn Recurrence via Hypernetwork}
    \State $\hat{h}
        \leftarrow 
            \underset{h \in \mathcal{N}\mathcal{N}^{\operatorname{ReLU}}_{\cdot}}{\operatorname{argmin}}\,
                \sum_{1\le i \le I_{\delta, Q}-1}\,
                \|h(z_{t_i}) - z_{t_{i+1}}\|_2
                =
                0
        $
    \State $L:\mathbb{R}^{P([d])}\times \mathbb{R}^Q \rightarrow \mathbb{R}^{P([d])}$ projection onto first component
    % \\
    \State \Return Trained CNO: $(\hat{f},z_0)$.
\end{algorithmic}
\end{algorithm}
{\color{black} In particular, we find it beneficial to share insights from a recent implementation of a mild variant of the CNO described in~\cite{RezaTime}. In that research, the objective was to learn causal maps on finite-dimensional manifolds of non-positive curvature instead of infinite-dimensional linear spaces. Instead of utilizing neural filters, a non-Euclidean readout layer, as introduced in~\cite{PaponAnnie}, was employed. This layer is compatible with the geometry of the space in which the dynamical system operates. Nevertheless, the core hypernetwork structure was preserved, which dynamically updates the model parameters over time. The training procedure for this structure was nearly identical to Algorithm~\ref{alg:train_CNO}, with only the necessary modifications.\hfill\\
That work emphasizes a strong experimental focus, aiming to demonstrate the practicality of a training procedure such as Algorithm~\ref{alg:train_CNO}. The most accurate, stable, and rapid training method involved first training the model $\hat{f}_{\theta{t_1}}$ using empirical risk minimization, as outlined in Algorithm~\ref{alg:train_CNO}, until achieving nearly zero training loss. By ensuring the network was sufficiently large, we successfully avoided overfitting due to the double-descent phenomenon, as documented in studies on overparameterized neural networks~\cite{MeiMontanariDD,ChengTSKernelChar}. Our findings confirmed this holds true in our context as well, suggesting that similar results can be expected in the future when exploring the statistical properties of the CNO in an infinite-dimensional framework.
\hfill\\
\indent After training the initial layer to achieve satisfactory predictions at time one, we discovered that the most stable and efficient training approach was to utilize transfer learning. This involved initializing the training of each model, denoted as $\hat{f}_{\theta_{t_{i+1}}}$, using the parameters obtained from the previous training round, specifically $\theta_{t_i}$. We then conducted only a few epochs of stochastic gradient descent. In general, the parameters $\theta_{t_1}, \ldots, \theta_{t_{I_{\delta,Q}}}$ showed minimal variation from one another. Moreover, using $\theta_{t_i}$ as a starting point for training $\hat{f}_{\theta_{t_{i+1}}}$ facilitated the training of the hypernetwork. This is because $\hat{f}_{\theta_{t_{i+1}}}$ has multiple parametric representations that yield the same functional representation. Thus, initializing at $\theta_{t_i}$ ensured that we were not only learning the correct function within the function space but also remaining within the same region of the joint parameter space of the neural filters. This approach had the added benefit of requiring fewer training iterations, as the difference $|\theta_{t_i} - \theta_{t_{i+1}}|$ remained small. However, it is important to note, as discussed in~\cite{PetWojLim}, that the mapping from a deep learning model's parameter space to its function space is typically only locally Lipschitz, with an extremely large local Lipschitz constant. Therefore, even if $|\theta_{t_i} - \theta_{t_{i+1}}| \approx 0$, the corresponding functions $\hat{f}_{\theta_{t_i}}$ and $\hat{f}_{\theta_{t_{i+1}}}$ may still be significantly different in the function space.\hfill\\
\indent Lastly, once we obtained each $\theta_1,\dots,\theta_{I_{\delta,Q}}$, we learned the relevant recurrence relation by training the hypernetwork to minimize the mean squared error between the sequential parameters: 
\begin{equation} \label{eq:Hypernetloss} 
    \frac1{I_{\delta,Q}-1}
    \sum_{i=1}^{I_{\delta,Q}-1}\,
        |h(\theta_{t_i})-\theta_{t_{i+1}}|^2
. 
\end{equation}
\indent In~\cite{RezaTime}, we did not empirically need the theoretically necessary augmentation from $\theta_{t_i}$ to $z_{t_i}=(\theta_{t_i},u_i)$ using some $\delta$-packing $\{u_i\}_{i=1}^{I_{\delta,Q}}$ of the Euclidean unit ball. The loss~\eqref{eq:Hypernetloss} was numerically optimized using stochastic gradient descent, and in our companion paper~\cite{RezaTime}, we found that this provided satisfactory performance.
\hfill\\
\indent The advantage of using a hypernetwork is particularly evident at this stage, as it enabled us to train a recurrent neural operator -- the CNO -- without relying on backpropagation through time, a method known for its numerical instability. Instead, minimizing the loss function in equation \eqref{eq:Hypernetloss} follows the standard approach of empirical risk minimization, which does not involve real-time components and does not present the same numerical issues. Additionally, a second benefit of the hypernetwork becomes apparent: once trained, the CNO can generate predictions for future time points that extend beyond the training data it was optimized with.}

% \begin{remark}[{Algorithm~\ref{alg:train_CNO}} Can be Federated]
% Algorithm~\ref{alg:train_CNO} is a federated training algorithm\footnote{See for example \cite{LiIEEESPMag} for further details on federated learning algorithms.}.  In it, every neural filter acts as a nodes, which is trained independently from one another.  Once optimized, these nodes send their parameters to the hypernetwork, which acts as a server synchronizing each of nodes into a central DL model.  
% \end{remark}

We now use our results to approximate solution operators arising in stochastic analysis {\color{black} and pricing functional arising naturally in mathematical finance}.

\section{Applications to Mathematical Finance and Stochastic Analysis}
\textcolor{black}{\subsection{Static Examples: Pricing Functionals}\label{sec:example}
We now provide some examples of how our static approximation theorems are naturally amenable to pricing problems in mathematical finance.  Our aim is both to showcase the need for the general Fr\'{e}chet setting, as well as the naturality of Assumption 2. 
\subsubsection{Functionals on a Fr\'{e}chet Space which is in Not Banach: Forward Rate Curves}
We now provide a concrete example where one is interested in approximating a real-valued functional $F$ defined on a Fr\'echet space which is not a Banach space, and which shows the necessity of the generality of the setting of our work. This example stems from fixed-income or commodity markets theory, and it has appeared in \cite{LucaUATPaperI}, to which we refer for the details. See also \cite{BDG2}. We recall here that in modelling the dynamics of forward rates in fixed-income markets, or forward and futures contract prices in commodity markets, one is concerned with a stochastic process taking values in a suitable space of functions, $(x(t,\cdot))_{t\geq 0}$. Here, for every $t\geq 0$, $x(t,\cdot)$ is a random variable with state space being real-valued functions on $\mathbb R_+$, i.e., each sample defines a function $\xi\mapsto x(t,\xi), \xi\geq 0$. The minimal condition on the state space of curves is that they are locally integrable functions, see Carmona and Tehranchi \cite{carmonatehranchi} and Filipovi\'c \cite{filip} for forward rates. Local intergrability allows for defining zero-coupon bond prices, and swap prices in power and gas markets. 
Following Benth, Detering and Galimberti \cite{BDG2}, the price of a typical financial derivative in the power market can be expressed by the functional
$$
F(x)=\mathbb E[\chi(x)]
$$
where $\chi$ is a random field and $x$ is a real-valued function on $\mathbb R_+$. In practice, $x$ denotes the current term structure of power forward prices. Following the discussion above, we may choose the space of such functions to be $L_{loc}^1:=L_{loc}^1(\mathbb R_+)$, endowed with its natural topology of Fr\'echet space. Thus, $F:L_{loc}^1\rightarrow\mathbb R$. The random field $\chi$, may be compactly expressed as (see~\cite{BDG2}),
$$
\chi(x)=\mathfrak{P}(Z\mathcal I_D(x))),
$$
where $Z$ is a real-valued integrable random variable, $\mathfrak P:\mathbb R\rightarrow\mathbb R$ is a Lipschitz continuous function (being the option's payoff) and $\mathcal I_D$ is a linear functional on $L_{loc}^1$ defined as an integral of $x$ over a compact set $D\subset\mathbb R_+$, namely $\mathcal{I}_D(x)=\int_D x(\xi)\,d\xi,\,\,x\in L^1_{loc}$.
The following lemma, which shows continuity of $F$ with respect to the natural locally convex topology of $L_{loc}^1$:
\begin{lemma}[{\cite[Lemma 6.1]{LucaUATPaperI}}]\label{lemma:L1loc}
The functional $F(x)=\mathbb E[\chi(x)]$ is locally Lipschitz on $L^1_{loc}$. 
\end{lemma}
When considering forwards and options on these, the set $D$ is typically a contractually specified week, month, quarter or year. Due to the continuity result above, we are in the context of our neural networks on a Fr\'echet space.
\hfill\\
Next, we show that under additional, realistic structural conditions the functional $F$, introduced above, can be efficiently approximated.  We do this by verifying the conditions of Corollary~\ref{cor:dynamiccase__regular_subcase}.}

\subsubsection{Efficient Approximation Rates for Pricing with Smooth Rapidly Decaying Functions}
\label{s:Examples__ss:Finance___sss:PricingL2}
{\color{black} 
We now focus on the naturality of Assumption 2.  Recall that the set of rescaled Hermite functions $\{H_k\}_{k=0}^{\infty}$ are an orthonormal basis, and thus a Schauder basis, of $L^2(\mathbb{R})$, where for each $k\in \mathbb{N}$, $H_k$ is defined by
\begin{equation}
\label{eq:hermitefunctions}
        H_k(x)
    = 
        \frac{
            (-1)^k
            e^{\frac{x^2}{2}}
        }{
            \sqrt[4]{\pi}\sqrt{2^k k!}
        }
        \,
        \frac{d^k}{dx^k}e^{-x^2}
\end{equation}
where $H_0(x) = e^{-x^2/2}/\sqrt[4]{\pi}$, where $H_1(x)= xe^{-x^2/2}/(\sqrt[4]{\pi}\sqrt{2})$, and so on.  
If we restrict Lemma~\ref{lemma:L1loc} to $L^2(\mathbb{R})$ instead of the largest set $L^1_{\operatorname{loc}}(\mathbb{R})$ then we still have Lipschitz continuity but with respect to the usual metric on $L^2(\mathbb{R})$.
Accordingly, every $f\in L^2(\mathbb{R})$ has a unique basis expansion
\begin{equation}
\label{eq:basis_expansion__HermiteFunctions}
        f 
    = 
        \sum_{k=0}^{\infty}
        \,
            \beta_k^f
            H_k
\end{equation}
where for each $k\in \mathbb{N}$, $\beta_k^f\eqdef \langle f,H_k\rangle_{L^2(\mathbb{R})}$.
}
% \begin{example}
% \label{ex:Fin_dim}
% Let $f\in L^2(\mathbb{R})$ and, for each $k\in \mathbb{N}_+$, let $\beta_k\eqdef |\langle f,\psi_k \rangle_{L^2(\mathbb{R})}|$.  
% By the Weyl asymptotics, we know that for any $s>0$ if $f$ belongs to the Sobolev space $H^s(\mathbb{R})$ then 
% \begin{equation}
% \label{eq:polynomial_decay__Weyl}
%         \beta_k 
%     \asymp
%         (1+k)^{-2s}
% .
% \end{equation}
% Consequently, if $f\in H^{\infty}(\mathbb{R})\eqdef \bigcap_{s>0}H^s(\mathbb{R})$ then~\eqref{eq:polynomial_decay__Weyl} implies that: there is some $r>0$ such that
% \begin{equation}
% \label{eq:polynomial_decay__Weyl}
%         \beta_k 
%     \lesssim 
%         e^{-r k}
% ,
% \end{equation}
% where $\lesssim$ hides a constant depending on $f$ but independent of each $k\in \mathbb{N}$.
% \end{example}
% }
\textcolor{black}{
We now explain the decay condition in~\eqref{eq:f_compatability_condition} has a very natural interpretation using the Hermite polynomials.  It can be understood as a joint tail-decay and smoothness condition.  We recall that rapid decay of the Fourier transform is a natural expression of smoothness by the Schwartz-Paley-Wiener theorem, see e.g.~\cite[Theorem 7.2.2]{SFourier}, which implies that a function is smooth only if its Fourier transform's coefficients decay super-polynomially.  Similarly, the spectral characterizations of Sobolev spaces $H^s(\mathbb{R})$ for $s>0$ as any $L^2(\mathbb{R})$ functions whose Fourier coefficients decay no slower than $(1+k)^{2s}$ by the Weyl asymptotics, see e.g.~\cite[Corollary 9.35]{BorSpec2020} or similar results.
\begin{example}[{Assumption 2 is a Decay and Smoothness Condition}]
\label{ex:SmoothessCondition}
As recently shown in~\cite[Theorem 1.1]{NHermite2025}, for any $f\in L^2(\mathbb{R})$ if $f$ and its Fourier transform $\hat{f}$ satisfy the exponential decay condition: 
there are $\lambda,c>0$ satisfying
\begin{equation}
\label{eq:condition_niceness}
        \underbrace{
            |f(x)| \le c e^{(-1/2 + \lambda)x^2}
        }_{\text{Decay Condition}}
    \mbox{ and }
        \underbrace{
            |\hat{f}(\xi)| \le c e^{(-1/2 + \lambda)\xi^2}
        }_{\text{Smoothness Condition}}
\end{equation}
for each $x,\xi \in \mathbb{R}$.  Then, there exist $C,r>0$, only depending on $c$ and on $\lambda$, such that the coefficients sequence $(\beta_k)_{k=0}^{\infty}$ in~\eqref{eq:basis_expansion__HermiteFunctions} satisfies
\begin{equation}
\label{eq:C_to_c__to_decay}
    |\beta_k^f| \le C\, e^{-r k}
.
\end{equation}
Let $\mathcal{K}\subset L^2(\mathbb{R})$ consist of all $f$ satisfying condition~\eqref{eq:condition_niceness}, for some fixed values of $c,\lambda>0$.  
Then, since the constant $C,r>0$ in~\eqref{eq:C_to_c__to_decay} only depended on the constant $c$ and on $\lambda$ in~\eqref{eq:condition_niceness} then, indeed there are constants $C,r>0$ such that: for every $f\in \mathcal{K}$\, $|\beta_k^f|\le C\, e^{-r k}$.  Whence, $\mathcal{K}$ satisfies the decay condition in~\eqref{eq:f_compatability_condition}.
\end{example}
Since $L^2$ can be Lipschitz embedded into $L^1_{loc}$, then Lemma~\ref{lemma:L1loc} clearly still holds if we restrict the functional $F$ to $L^2$ now.  Thus, $F$ is still Lipschitz on $L^2(\mathbb{R})$.
Thus, $F$ can be approximated on the compact set $\mathcal{K}$ of Example~\ref{ex:SmoothessCondition} using the efficient approximation guarantee in Corollary~\ref{cor:dynamiccase__regular_subcase}, so long as the constant $c$ in~\eqref{eq:condition_niceness} is chosen small enough so that the constant $C$ small enough, as noted in Remark~\ref{rem:_cor:dynamiccase__regular_subcase__LowRegVersion}.}

\subsection{{\textcolor{black}{Dynamic Examples: Solution Operators of Stochastic Differential Equations}}}\label{sec:applications_StochAnalMathFin}
We apply our results to show that several solution operators from stochastic analysis can be approximated by the CNO.  Our neural network model can approximate stochastic processes without assuming strong structural conditions describing their evolution.We illustrate our result's implications for obtaining numerical solutions to SDEs, and we discuss the implications for more general stochastic processes, e.g. processes with jumps, towards the end of this section.

\begin{figure}[H]
\centering
\begin{subfigure}[b]{.85\textwidth}
    \centering
   \includegraphics[width=.65\textwidth]{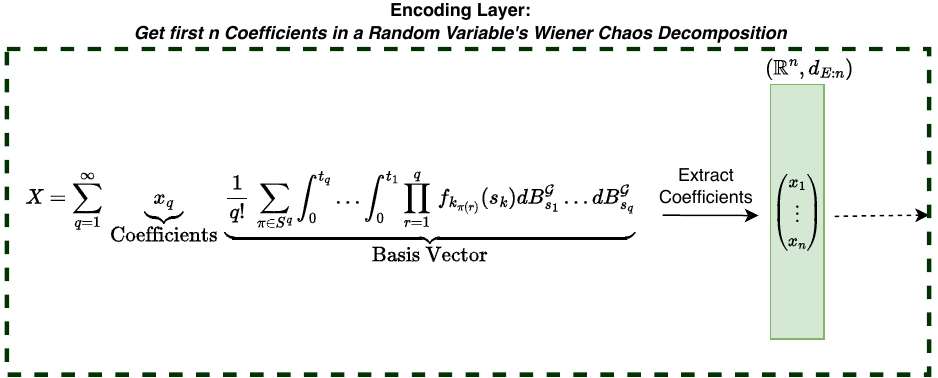}
   \caption{Encoding Layer}
   \label{fig:Neural_Filter_StochasticsApplicationsForm__Encoding}
\end{subfigure}
%%%
\begin{subfigure}[b]{.85\textwidth}
\centering
\includegraphics[width=.65\textwidth]{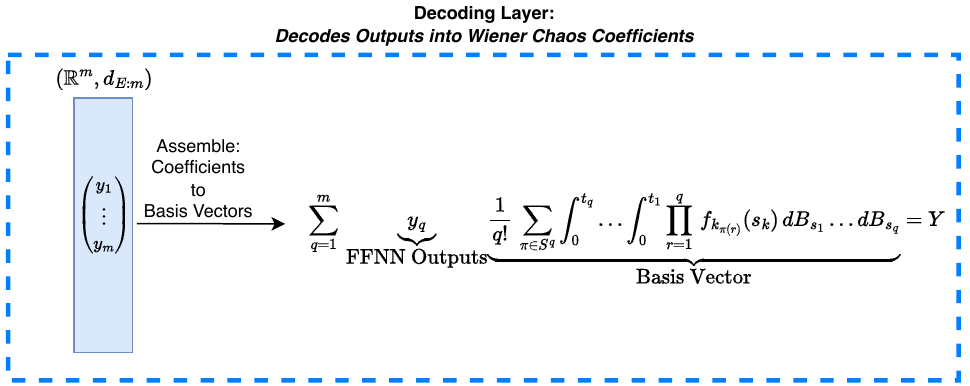}
   \caption{Decoding Layer}
   \label{fig:Neural_Filter_StochasticsApplicationsForm__Decoding}
\end{subfigure}
\caption{{Illustration of our ``static'' operator network in Definition~\ref{def:Neural_Filter} specialized to the geometry of the input space $L^2(\Omega,\mathcal{G}_T,\mathbb{P})$ and the output space $L^2(\Omega,\mathcal{F}_t,\mathbb{P})$; for $\sigma$ algebras $\mathcal{G}$ and $\mathcal{F}$ on a sample space $\Omega$.  The network is works in three phases.  {\color{green}{1)}} First inputs are encoded as finite-dimensional Euclidean data by mapping them to their truncated (Schauder) basis coefficients in the input space $E$.  {\color{violet}{2)}} Next these coefficients are transformed by a ReLU FFNN.  {\color{black}{3)}} The outputs of ReLU FFNN's output are interpreted as coefficients a Wiener Chaos expansion a truncated (Schauder) basis in the output space $F$.}}
\label{fig:Neural_Filter_StochasticsApplicationsForm}
\end{figure}

\subsection{A primer on Wiener Chaos}
\label{s:Applications__ss_Stochastics___sss_Background}

We fix a probability space $\big(\Omega,\mathcal{F},\mathbb{P}\big)$ supporting a standard one dimensional Brownian motion $(B_t)_{t\ge 0}$ and let $\mathbb{F} \eqdef (\mathcal{F}_t)_{t\ge 0}$ denote the complete and right-continuous enlargement of the filtration generated by $(B_t)_{t\ge 0}$. 
We recall that the Ito (stochastic) integral of a (deterministic) simple function $f=\sum_{i=1}^k\, \beta_i I_{[0,t_i]}$ in $L^2([0,t])$, where $0\le t_1< \dots < t_k\le t$ is the Gaussian random variable
\begin{equation}
\label{eq:stoch_integral__simplefunctions}
\int_0^t\, f(s)\, dB_s 
 \eqdef  
\sum_{i=1}^k\,
    \beta_i 
    \,
    \big(
        B_{t_i}-B_{t_{i-1}}
    \big)
.
\end{equation}
More generally, the Ito integral of any function $f\in L^2([0,t])$ is defined as the limit in $L^2(\Omega,\mathcal{F}_t,\mathbb{P})$ of a sequence $\{\int_0^t\, f_k(s)\, dB_s \}_{k=1}^{\infty}$ where the $\{f_k\}_{k=1}^{\infty}$ is any choice of simple integrands converging to $f$ in $L^2([0,t])$.  Thus, $\int_0^t\, f(s)\,dB_s$ is a centered normal random variable with variance $\int_0^t\,f^2(s)\,ds$.  We also note that such a sequence always exists and $\int_0^t\, f(s) \, dB_s$ is independent of the particular choice of the approximating sequence $\{f_k\}_{k=1}^{\infty}$.

Using tools common to (Malliavin) stochastic calculus we may exhibit an orthonormal basis of $L^2(\Omega,\mathcal{F}_t,\mathbb{P})$.  We refer the interested reader to \cite{nualart2006malliavin} for a more detailed discussion on this construction.  This construction relies on a system of orthogonal polynomials $\{h_k\}_{k=1}^{\infty}$ known as \textit{Hermite polynomials}, a rescaled-variant of the Hermite functions defined in~\eqref{eq:hermitefunctions}, $\{h_k\}_{k=0}^{\infty}$ are defined by
\[
        h_{k}(x) 
    = 
        {\color{black} (-1)^k e^{\frac{x^2}{2}} \frac{d^k}{dx^k}e^{-\frac{x^2}{2}}}
            % x h_{k}(x) - k h_{k-1}(x)}
,
\]
and, for instance, $h_0(x) \eqdef 1$, $h_1(x)\eqdef x$, and so on.  

By means of the Ito stochastic integral and the Hermite polynomials we may define the $q^{th}$ Wiener Chaos to be the subspace $\mathcal{H}_t^q$ of $L^2(\Omega,\mathcal{F}_t,\mathbb{P})$ spanned by the random variables of the form
\[
    I^q(f)  \eqdef  
            h_q\biggr(
                \int_0^t \, f(s)\, dB_s
            \biggl) 
,
\]
where $f\in L^2([0,t])$, where $q\in \mathbb{N}_+$ and $\mathcal{H}_t^0 \eqdef \rr$.  The Wiener chaos $(\mathcal{H}_t^q)_{q=0}^{\infty}$ produces an orthogonal decomposition, given in \cite[Theorem 1.1.1]{nualart2006malliavin}, of $L^2(\Omega,\mathcal{F}_t,\mathbb{P})$, meaning that for each pair of random variables $Y_q\in \mathcal{H}_t^q$ and $Y_{\tilde{q}}\in \mathcal{H}_t^{\tilde{q}}$ are orthogonal in $L^2(\Omega,\mathcal{F}_t,\mathbb{P})$ whenever $q\neq \tilde{q}$; every random variable $Y\in L^2(\Omega,\mathcal{F}_t,\mathbb{P})$ can uniquely be decomposed as
\[
    Y 
        =
    \sum_{q=0}^{\infty}\, 
    Y_q
,
\]
where $Y_q\in \mathcal{H}_t^q$ for each $q\in \mathbb{N}$ and where the sum converges in $L^2(\Omega,\mathcal{F}_t,\mathbb{P})$.

Since the Wiener Chaos is an orthogonal decomposition of $L^2(\Omega,\mathcal{F}_t,\mathbb{P})$ then the union of any set of orthogonal basises of each $\mathcal{H}_t^q$ is an orthogonal basis of $L^2(\Omega,\mathcal{F}_t,\mathbb{P})$ itself.  Therefore, we only need to exhibit an orthogonal basis of each $\mathcal{H}_t^q$ for $q\in \mathbb{N}_+$.  

We leverage the \textit{symmetrized tensor product} of elements $f_1,\dots,f_q\in L^2([0,t])$ defined by
\[
\operatorname{sym}\big(
    f_1\otimes \dots \otimes f_q
\big)
     \eqdef  \frac1{q!}\, \sum_{\pi \in S^q}\, f_{\pi(1)}\otimes \dots \otimes f_{\pi(q)}
\]
where $S^q$ is the set of permutations of the indices $\{1,\dots,q\}$.  More concretely, the Hilbert space generated by the symmetrized tensor product\footnote{See \cite[Chapter IV page 43]{Bourbaki1998Algebre1a3}.} is identified\footnote{See \cite[Lemma 8.4.2]{PecattiTaqqu_2011_ChaosDiagrams}.} with the set of symmetric functions%
\footnote{A ``function'' $f\in L^2([0,t]^q)$ is \textit{symmetric} if $f(s_1,\dots,s_q)=f(s_{\pi(1)},\dots,s_{\pi(q)})$, for all $\pi\in S^q$,
outside a set of $q$-dimensional Lebesgue measure $0$.  } in $L^2([0,t]^q)$ which we denote by $L^2_{\operatorname{sym}}([0,t]^q)$.  Since the $q$-fold symmetrized tensor product is a subspace of the (usual) $q$-fold tensor product then the identification of the $q$-fold symmetric tensor product of $L^2([0,t])$ with $L^2_{\operatorname{sym}}([0,t]^q)$ may be further simplified to
\[
    \operatorname{sym}\big(f_1\otimes \dots \otimes f_q \big)\,
        \leftrightarrow
    \frac1{q!}\, \sum_{\pi \in S^q}\, 
        \prod_{i=1}^q\, f_{\pi(i)}(s_i)
.
\]

The connection between the symmetrized tensor product and the $q^{th}$ Wiener Chaos is that the $q^{th}$ Wiener Chaos $\mathcal{H}_t^q$ is structurally identical to $L^2_{\operatorname{sym}}([0,t]^q)$ (identified with the $q$-fold symmetrized tensor product of $L^2_{\operatorname{sym}}([0,t])$ with itself). 
The map realizing this identification sends any $f\in L^2_{\operatorname{sym}}([0,t]^q)$ to its $q$-fold multiple stochastic integral 
\begin{equation}
\label{eq:HilbertSpaceIsomorphism}
    f 
        \mapsto
    \int_0^{t_q}\dots \int_0^{t_1}\,
    f(s_1,\dots,s_q)\,
    dB_{s_1}\dots dB_{s_q}.
\end{equation}
Moreover, the map~\eqref{eq:HilbertSpaceIsomorphism} is linear isometric isomorphism preserving inner products\footnote{See \cite[Proposition 8.4.6 (1)]{PecattiTaqqu_2011_ChaosDiagrams}.}.  Consequentially, any orthogonal basis of $L^2_{\operatorname{sym}}([0,t]^q)$ is sent to an orthogonal basis of $\mathcal{H}_t^q$ under this identification.  Since an orthogonal basis of $L^2_{\operatorname{sym}}([0,t]^q)$ is given by the set
\[
    \operatorname{sym}\big(
        f_1\otimes \dots \otimes f_q
    \big)
\]
where $\{f_i\}_{i=1}^{\infty}$ is an orthogonal basis\footnote{See \cite[page 153, point (iii)]{PecattiTaqqu_2011_ChaosDiagrams}.} of $L^2([0,t])$ then the identification~\eqref{eq:HilbertSpaceIsomorphism} implies that the corresponding set of random variables 
\begin{equation}
\label{eq:orthogonal_basis_symmetrized_tensorproduct}
    \int_0^{t_q}\dots \int_0^{t_1}
    \,
    \operatorname{sym}\big(
        f_1\otimes \dots \otimes f_q
    \big)(s_1,\dots,s_q)
    \,
    dB_{s_1}\dots dB_{s_q}
    ,
\end{equation}
is an orthogonal basis of the $q^{th}$ Wiener Chaos $\mathcal{H}_t^q$.  
Such an orthogonal basis of $L^2([0,t])$ is given by the Fourier basis whose elements are
\[
        f_{j,i}(x)
     \eqdef 
    \begin{cases}
    \sqrt{\frac{2}{t}}\sin\left(\frac{j\pi x}{t}\right)
        &\mbox{if } \,i\mbox{ = 0}\\
    \sqrt{\frac{2}{t}}\cos\left(\frac{(j-1)\pi x}{t}\right)
        &\mbox{if } \,i\mbox{ = 1}
    ,
    \end{cases}
\]
where $j\in \mathbb{N}_+$ and $i\in \{0,1\}$.  For convenience, with some abuse of notation, we denote an enumeration of $\{f_{i,j}\}_{i\in \mathbb{N},j\in \{0,1\}}$ by $\{f_k\}_{k=1}^{\infty}$.  
Consequentially, an orthogonal basis of $L^2(\Omega,\mathcal{F}_t,\mathbb{P})$ is given by the countable family of random variables
\[
    Z_{(k_1,\dots,k_q)}  \eqdef 
    \frac1{q!}\, \sum_{\pi \in S^q}\, 
    \int_0^{t_q}\dots \int_0^{t_1}
        % Proper Notation
        \prod_{r=1}^q\, 
        % Shorthand
        f_{k_{\pi(r)}}
        (s_k)
        \,
        dB_{s_1}\dots dB_{s_q}
,
\]
where $(k_1,\dots,k_q)$ is a multi-index belonging to $\mathcal{A} \eqdef  \bigcup_{q=0}^{\infty}\, 
\mathbb{N}^q$; we also make the convention that $Z_{\emptyset} \eqdef 1$, and we have used the linearity of the Ito (stochastic) integral in conjunction with the above considerations.

\subsection{Simultaneous Approximation of SDEs with Different Initial Conditions using CNOs}
\label{s:Applications__ss_StochProcesses}
Monte Carlo methods allow for the efficient solution to stochastic differential equations (SDEs) with a convergence rate of $\mathcal{O}(1/\sqrt{S})$\footnote{Typically in the $L^2$-sense.} to the true solution, where $S$ is the number of samples, plus a comparable discretization error when resorting to a tamed Euler scheme \cite{JenztenTamedEulre2011}.  It is known that deep learning can provide a suitable alternative to Monte Carlo schemes by learning the SDE's solution map given deterministic initial conditions, for a fixed terminal time, by approximating the solutions to their associated PDEs \cite{Beck2018} given by the Feynman-Kac Theorem.

In this section, we show how a \textit{single} CNO can be used for simultaneously solving SDEs with various noisy initial conditions across different time-horizons, by \textit{simultaneously} approximately learning solve a family of stochastic differential equations with many different stochastic initial conditions and different initial times.  

This section's application shows that the CNO can approximate causal maps with stochastic inputs on arbitrarily long time horizons.  This extends the known guarantees for recurrent neural networks, specifically reservoir computers, which can approximate time-invariant causal maps \cite{Lukas_StochInputReservoir}.

We are given a non-degenerate time grid $(t_n)_{n \in \mathbb{Z}}$ as in Assumption 1, $\beta$ and $\alpha$ in $C([0, \infty) \times \mathbb{R}, \mathbb{R})$ such that there exists $M > 0$ such that for all $t\geq 0$ and all $x_1, x_2 \in \mathbb{R}$, we have 
\begin{align}
% \label{eq:Lipschitzness_SDE}
    |\beta(t, x_1)-\beta(t, x_2)|^2 + |\alpha(t, x_1)-\alpha(t, x_2)|^2 & \leq M^{2} |x_1 - x_2|\label{eq:growth_1}
\\
% \label{eq:growth_SDE}
    |\beta(t, x_1)|^2 + |\alpha(t, x_1)|^2 & \leq M^2 ( 1 + |x_1|^2)\label{eq:growth_2}.
\end{align}
Theorem 8.7 in \cite{DaPrato2008} guarantees that for all $i\in\mathbb{N}_+$, under the growth conditions \eqref{eq:growth_1} and \eqref{eq:growth_2}, 
for $\eta \in L^{2}(\Omega, \mathcal{F}_{t_i}, \mathbb{P})$ there exists a unique $X \in C([t_{i}, t_{i+1}]; L^2(\Omega, \mathcal{F}_{t_{i+1}}, \mathbb{P} ))$ which satisfies  $\mathbb{P}$-a.s.  
\begin{equation}
\label{eq:diffusiondynamics}
    X_{t_{i+1}} = \eta + \int_{t_i}^{t_{i+1}} \alpha(s, X_s)\,ds +  \int_{t_i}^{t_{i+1}} \beta(s, X_s)\,dB_s,
\end{equation}
where we set $X_{t_{i}} = \eta$; in what follows, we will indicate the explicit dependence on $\eta$ in $X_{t_{i+1}}$, i.e. $X_{t_{i+1}}^{\eta}$. Therefore, $\forall i \in \mathbb{N}_+$ the following (non-linear) \emph{solution operator}
\begin{equation}\label{eq:solution_operator}
    \operatorname{SDE-Solve}_{t_i:t_{i+1}}\,:\,L^2(\Omega, \mathcal{F}_{t_{i}}, \mathbb{P}) \rightarrow L^2(\Omega, \mathcal{F}_{t_{i+1}}, \mathbb{P}),\quad \eta \rightarrow X_{t_{i+1}}^{\eta}
\end{equation}
is well defined\footnote{See \cite[Section 8]{DaPrato2008}.}. To see that each of the maps $\operatorname{SDE-Solve}_{t_i:t_{i+1}}$ satisfies the assumptions of our theorems, it is sufficient to note that under \eqref{eq:growth_1} and \eqref{eq:growth_2}, the operator $\operatorname{SDE-Solve}_{t_i:t_{i+1}}$ is Lipschitz and, in view of \cite[Proposition 8.15]{DaPrato2008}, it belongs to the trace-class $C_{1,\text{tr}}^{\lambda}(K, L^2(\Omega, \mathcal{F}_{t_{i+1}}, \mathbb{P}))$ for all compact subsets $K$ of $L^2(\Omega, \mathcal{F}_{t_{i}}, \mathbb{P} )$, since
\begin{equation}\label{eq:estimate}
    \begin{split}
        \|X_{t_{i+1}}^{\hat{\eta}}-X_{t_{i+1}}^{\tilde{\eta}}\|_{L^2(\Omega, \mathcal{F}_{t_{i+1}}, \mathbb{P}; \mathbb{R})} &\leq \sqrt{3} e^{\frac{3}{2}M^2(t_{i+1}-t_{i}+1)(t_{i+1}-t_{i})}
        \|\hat{\eta}-\tilde{\eta}\|_{L^2(\Omega, \mathcal{F}_{t_{i}}, \mathbb{P}; \mathbb{R})}\\
                                                        &\leq \sqrt{3} e^{\frac{3}{2}M^2(\Delta^{+}+1)\Delta^{+}}
                                                        \|\hat{\eta}-\tilde{\eta}\|_{L^2(\Omega, \mathcal{F}_{t_{i}}, \mathbb{P}; \mathbb{R})}.
    \end{split}
\end{equation}
with $\lambda \leq\sqrt{3} e^{\frac{3}{2}M^2(\Delta^{+}+1)\Delta^{+}}$ and $\Delta^{+} \eqdef \sup_{i \in \mathbb{Z}}\Delta t_{i} < \infty$ as in Assumption 1.

We consider the causal map
\begin{equation}\label{eq:solution_adapted_map}
    \begin{split}
     \operatorname{SDE-Solve}\,:\,&\left(\prod_{i \in \mathbb{Z}; t_i<0}\{0\}\right) \times\prod_{i \in \mathbb{Z}; t_i\geq 0}\,L^2(\Omega, \mathcal{F}_{t_{i} }, \mathbb{P}) \rightarrow \left(\prod_{i \in \mathbb{Z}; t_i<0}\{0\}\right)\times\prod_{i \in \mathbb{Z};t_i\geq 0} L^2(\Omega, \mathcal{F}_{t_{i}}, \mathbb{P}),\\
     &\quad\quad\quad\quad\quad\quad\,\,(\eta_{t_i})_{i \in \mathbb{Z}} \mapsto \operatorname{SDE-Solve}\left[(\eta_{t_{i}})_{i \in \mathbb{Z}}\right],
     \end{split}
\end{equation}
\begin{equation}
\label{eq:solution_adapted_map__II}
    \left(\operatorname{SDE-Solve}\left[
    (\eta_{t_i})_{i \in \mathbb{Z}}
    \right]\right)_j = 
    \begin{cases}
        0, \quad \text{if } t_j<0\\
        \operatorname{SDE-Solve}_{t_j:t_{j+1}}(\eta_{t_j})
        % {\color{red}{ = X^{\eta_{t_j}}_{t_{j+1}} }}
        , \quad\text{if } t_j\geq 0,
    \end{cases}
\end{equation}
where each $\operatorname{SDE-Solve}_{t_i:t_{i+1}}(\eta_{t_{i}})$ is defined as in Equation \eqref{eq:solution_operator}.  
The typical example which we have in mind, in the following, are input sequences which are orbits of square-integrable random variables under the an SDE's solution operator; i.e.\
\begin{equation}
\label{eq:orbits}
\eta_{t_{i+1}} = \operatorname{SDE-Solve}_{t_i:t_{i+1}}(\eta_{t_i})
\mbox{ and }
\eta_{t_0} = X,
\end{equation}
for some $X\in L^2(\Omega,\mathcal{F}_0,\mathbb{P})$.  Thus, approximating $\operatorname{SDE-Solve}$ and applying it to any compact subset of the path-space comprised of elements of the form~\eqref{eq:orbits} corresponds to simultaneously solving an SDE for several random initial conditions across arbitrarily time-intervals beginning at several initial times.

By Equation \eqref{eq:estimate}, $\operatorname{SDE-Solve}$ is a \emph{causal map} as in Definition \eqref{def:causal_maps}, since in this case we can simply take $r=0,\alpha=1$, $M=1$, $f_{t_i} =  \operatorname{SDE-Solve}_{t_i:t_{i+1}}$, and $\lambda \leq  \sqrt{3} e^{\frac{3}{2}M^2(\Delta^{+}+1)\Delta^{+}}$ holds for any $i\in \mathbb{N}_+$. Theorem \ref{thm:theorem_universality_causal} guarantees that there exists a CNO which approximates the map in Equation \eqref{eq:solution_adapted_map}, as soon as we confine ourselves on a compact path space.
Let us summarize our findings in 

\begin{corollary}[Causal Universal Approximation of SDEs with Stochastic Dynamics]
\label{cor:Applications__ss_StochProcesses}
Consider the setting of this section and fix the path space
    \[
    \mathcal{X} 
     \eqdef  
    \left(\prod_{i \in \mathbb{Z}; t_i<0}\{0\}\right) \times\prod_{i \in \mathbb{Z}; t_i\geq 0}\,\mathcal{X}_{t_i},
    \]
    where each $\mathcal{X}_{t_i}$ is a compact subset of $L^2(\Omega,\mathcal{F}_{t_i},\mathbb{P})$.  
    Then the operator $\operatorname{SDE-Solve}$ 
    \begin{equation*}
     \operatorname{SDE-Solve}\,:\,\left(\prod_{i \in \mathbb{Z}; t_i<0}\{0\}\right) \times\prod_{i \in \mathbb{Z}; t_i\geq 0}\,\mathcal{X}_{t_i}
     \rightarrow 
     \left(\prod_{i \in \mathbb{Z}; t_i<0}\{0\}\right)\times\prod_{i \in \mathbb{Z};t_i\geq 0} L^2(\Omega, \mathcal{F}_{t_{i}}, \mathbb{P})
\end{equation*}
is $(0, 1, \sqrt{3} e^{\frac{3}{2}M^2(\Delta^{+}+1)\Delta^{+}})$-H\"older.

Given $Q, \delta \in \mathbb{N}_{+}$, an ``encoding error" $\varepsilon_{D}>0$ and an ``approximation error" $\varepsilon_{A}>0$ there exist a multi-index $[d]$, a ``latent code" $z_0 \in \mathbb{R}^{P([d])+Q}$, a linear readout map $L\,:\,\mathbb{R}^{P([d])+Q}\rightarrow\mathbb{R}^{P([d])}$, and a ReLU FFNN  $\hat{h}\,:\,\mathbb{R}^{P([d])+Q}\rightarrow\mathbb{R}^{P([d])+Q}$ such that the sequence of parameters $\theta_{t_i} \in \mathbb{R}^{P([d])}$ defined recursively
\begin{equation*}
    \begin{split}
        \theta_{t_i} \eqdef L(z_{t_i})\\
        z_{t_{i+1}} \eqdef \hat{h}(z_{t_i}),
    \end{split}
\end{equation*}
with $i$ belongs to $[[I]]\cup \{0\}$ provided by the definition of causal maps \footnote{See Definition \ref{def:causal_maps}.}, satisfies to the following uniform estimates:

\begin{equation*}
    \begin{split}
    \max_{i \in [[I]]}\,
    \sup_{X_{\cdot} \in \mathcal{X}}\,\,
        \| 
        \hat{f}_{t_i}(X_{(t_{i-1}, t_i]})-
         \operatorname{SDE-Solve}(X_{\cdot})_{t_i}  \|_{L^2}
             <  \varepsilon_A + \varepsilon_D,
    \end{split}
\end{equation*}
where\footnote{We recall, Definition \ref{def:Neural_Filter}, stating that
$\hat{f}_{t_i} \eqdef I_{B_{t_i}:n_{\varepsilon_D^{out}}} \circ \varphi_{n_{\varepsilon_D^{out}}}\circ \hat{f}_{\theta_{t_i}} \circ \psi_{n_{\varepsilon_D^{out}}}\circ P_{E_{t_i}:n_{\varepsilon_D}^{in}}
$.} 
$\hat{f}_{t_i} \in \mathcal{NF}_{[n_{\varepsilon_D}]}^{(P)ReLU}$.
Moreover, for the hyperparameter $n_{\varepsilon_D}^{in}$ it holds 
\[
n_{\varepsilon_D}^{in} = \inf\left\{n\in\mathbb N_+: \max_{x\in \mathcal{X}} d_E(A_{E:n}(x),x) \leq \frac{\varepsilon_D}{2\lambda}\right\}
\]
where we have set $E\eqdef\Pi_{i\in \mathbb Z}L^2(\Omega,\mathcal{F}_{t_i},\mathbb P)$.
\end{corollary}

\subsection{Discussion - {Corollary~\ref{cor:Applications__ss_StochProcesses}}: Jumps, Path-Dependence, and Accelerated Approximation Rates Under Smoothness}
\label{sec:Discussion}
We briefly discuss some points surrounding Corollary~\ref{cor:Applications__ss_StochProcesses}.  For instance, how the result allows for stochastic discontinuity-type jumps.  We also discuss how the scope of Theorem~\ref{thm:theorem_universality_static} allows for Corollary~\ref{cor:Applications__ss_StochProcesses} to be easily generalized; but we opt not to do that in this manuscript, rather opting for a less technical illustration of our general framework.  

\paragraph{Improved Approximation Rates for SDEs Driven by Smooth Coefficients}
If, in addition to conditions~\eqref{eq:growth_2} and~\eqref{eq:growth_1}, the drift and diffusion coefficients $\alpha$ and $\beta$ are sufficiently differentiable\footnote{The precise conditions are formalized in \cite[Assumption 3.7]{RosestolatoBook2017}.}, then \cite[Theorem 3.9]{RosestolatoBook2017} implies that each of the maps $\operatorname{SDE-Solve}_{t_i:t_{i+1}}$ are $C^k$.  Whence, the operator $\operatorname{SDE-Solve}$ is a smooth causal map of finite virtual memory.  Thus, in this case, Theorem~\ref{thm:theorem_universality_causal} implies improved approximation rates by the CNO model.

\paragraph{Stochastic Discontinuities at Time-Grid Points}
We highlight that the adapted map $\operatorname{SDE-Solve}$ does accommodate jumps but only if those jumps occur on the fixed time-grid points $\{t_i\}_{i\in \mathbb{N}}$.  Such constructions have recently appeared in the rough path literature \cite{Allan2021ArXiV} and the causal/functional It\^{o} calculus literature \cite{Cont2022JMathAnnAppl}.  

In financial applications, the possibility of a stochastic process' to jump at predetermined times (called \textit{stochastic discontinuities} in that context) are an essential ingredient of accurately modeling interest rates; for example, European reference interest rates typically exhibit jumps directly after monetary policy meetings of ECB \cite{Font2020FinStoch}.  
% \end{remark}

\paragraph{Path Dependent Dynamics}
One could consider SDEs driven with path dependant random drift and diffusion coefficients, since all that is needed to apply Theorem~\ref{thm:theorem_universality_causal} is the regularity of the $\operatorname{SDE-Solve}$ operator; which is guaranteed by results such as \cite{DePrato2016AnnProb} or \cite{RosestolatoBook2017}.  However, we instead opted for a simple first presentation, explicitly illustrating the scope of our results in this easier case.

\section{The Benefit of Causal Approximation: Super-Optimal Approximation Rates for Causal Maps}
\label{sec:Comparisons}
We now illustrate the quantitative advantage of causal approximation, i.e.\ using our CNO architecture, when the target function is causal.  For illustrative purposes, we consider the simplest case where all involved spaces are finite-dimensional and Euclidean.  By considering this setting, we can juxtapose our approximation rates derived from Theorem~\ref{thm:theorem_universality_causal} against the best {\color{black} upper-bounds on the approximation} rates for ReLU networks \cite{ZHS2022JMPA} {\color{black} which apply to our class of causal maps}, {\color{black} which match the well-known lower bounds for Lipschitz maps \textit{without the additional causality constraint}~\cite{DVL1993,Grohs2021IEEETransInfoTheory}; however, there are currently no available lower bounds on this causal class.} 

Therefore, when the target function has a causal structure, ``super-optimal uniform approximation rates'' can be achieved only if one encodes that structure into the neural network model; as in the case with the CNO.  Throughout this section, we consider the integer time-grid $\{t_i\}_{i\in \mathbb{Z}}=\{t\}_{t\in \mathbb{Z}}$; which we note satisfies the non-degeneracy condition in Assumption 1.

\subsection{In the Euclidean Case, CNOs are a simple class of RNNs which are universal dynamical systems}
\label{s:Euclidean}

\begin{figure}[H]%[ht]
\centering
\includegraphics[width=0.95\textwidth]{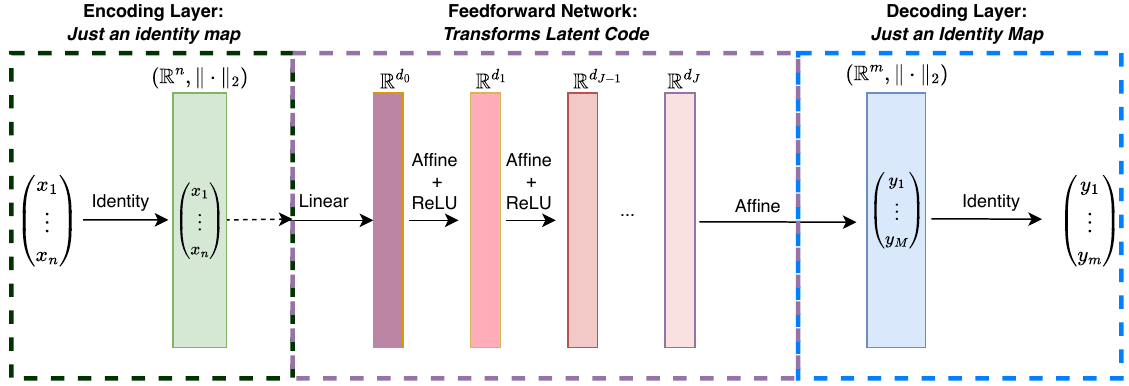}
\caption{\textbf{Neural Filters - Euclidean Spaces}: If the input and output spaces are Euclidean, then the projection and reconstruction layers in Figure~\ref{fig:Neural_Filter_GeneralForm} can be dropped; since they reduce to formal identity maps.  Thus, in this setting a neural filter is a deep ReLU FFNN.}
\label{fig:Neural_Filter_EuclideanForm}
\end{figure}
In \cite{KN1998NN}, the authors investigate the problem of approximating a dynamical system on a Euclidean space by a RNN. In their most general form, RNNs -- sometimes also called ``fully RNN", or fRNNs - are given for times $t>0$ by 
\begin{equation}\label{eq:fRNNs}
    \begin{split}
        y_t & \eqdef  \hat{f}_{\theta_t}(y_{t-1}, x_t),\\
        y_0 & \eqdef  y,
    \end{split}
\end{equation}
where $y_t$ is the state of the system, $x_t$ is an external input, $y$ the initial state, and $\hat{f}_{\theta_t}$ are (possibly deep) FFNNs with a priori no relationship among their parameters $(\theta_t)_{t \in \mathbb{N}_{+}}$. In particular, each FFNNs may have different depth and/or width. However, in practice, restrictions are put on the sequence of networks $(\hat{f}_{\theta_t})_{t \in \mathbb{N}_{+}}$; precisely, it is usually required that they all have the same \textit{complexity}, and each $\theta_{t+1}$ is recursively determined from the pair $(\theta_t, x_t)$. For instance, if it is only assumed that each FFNNs in Equation \eqref{eq:fRNNs} has the same complexity, then the classical result of \cite{SS1995JCSS} shows that one may simulate all Turing Machines by fRNNs with rational weights and biases. Although this result is promising for the expressive power of fRNNs, it is far removed from any practical model since it places absolutely no restriction on how the sequence $(\theta_t)_{t \in \mathbb{N}_{+}}$ is determined. As a consequence, the model in Equation \eqref{eq:fRNNs} is not implementable since it depends on an infinite number of parameters, as there is no relationship between $\theta_t$ and any $\theta_s$ for all past times $s<t$. On the other extreme, a very recent paper \cite{helmut2022metric} prove that a RNN with a single hidden layer and with $\theta_t = \theta_0$, for all $t\in\mathbb{N}_{+}$, can approximate linear time-invariant dynamical systems quantitatively.\\
\indent Still, surprisingly, many questions surrounding the approximation power of more sophisticated but implementable RNNs remain open. For instance, the ability of such RNNs to approximate non-linear
dynamical systems, quantitatively, and the quantitative role of the hidden state space/latent code's dimension are still open problems in the neural network literature. This subsection, addresses these open problems as a simple and direct consequence of Theorem \ref{thm:theorem_universality_causal}.\\
\indent This is because if $E=B=\mathbb{R}^{d}$, (with $\mathbb{R}^{d}$ equipped with the Euclidean distance), then our CNO model defines a very simple RNN. In order to see this, let $(e_i)_{i = 1}^{d}$ be the standard basis of $\mathbb{R}^{d}$, which is trivially a Schauder basis for the latter. Requiring that the \textit{encoding} and the \textit{decoding} dimensions of our CNO model are at least $d$, we have that the latter is given by\footnote{See Theorem \ref{thm:theorem_universality_causal} for the precise notation.}:
\begin{equation}\label{eq:CNO_1}
    \begin{cases}
        y_t \eqdef \hat{f}_{\theta_t}(x_t),\\
        \theta_t \eqdef  L(z_t),\\
        z_{t+1} \eqdef \hat{h}(z_t).
    \end{cases}
\end{equation}
\noindent Moreover, by pre-composing each $\hat{f}_{\theta_t}$ in Equation \eqref{eq:CNO_1} with the following linear projection
\begin{equation*}
    A\,:\,\mathbb{R}^{N} \times \mathbb{R}^{N} \rightarrow \mathbb{R}^{N},\quad (y, x) \rightarrow x,
\end{equation*}
and by noting that $ \hat{f}_{\theta_t} \circ A$ is a FFNN because of the invariance with respect the pre-composition by affine functions, we have that the CNO becomes 
\begin{equation}\label{eq:CNO_2}
    \begin{cases}
        y_t \eqdef \hat{f}_{\theta_t}(y_{t-1}, x_t),\quad y_0 \eqdef y\\
        \theta_t \eqdef  L(z_t),\\
        z_{t+1} \eqdef \hat{h}(z_t),
    \end{cases}
\end{equation}
\noindent 
where with a minor abuse of notation we keep using $\hat{f}_{\theta_t}$ instead of $\hat{f}_{\theta_t} \circ A$.  
By comparing Equations \eqref{eq:fRNNs} and \eqref{eq:CNO_2}, we see that the CNO model is a RNN whose weights and biases do not depend upon the input sequence $(x_t)_{t \in \mathbb{N}_{+}}$, and are determined recursively by the \textit{hypernetwork} $\hat{h}$, as in \cite{HDQICLR2017}. Therefore, our CNO is essentially the classical Elman RNN of \cite{EL1990} with $\hat{f}_{\theta_t}$ and $\hat{h}$ having several, instead of one, hidden layer.\\
\indent We now illustrate the expressive power of the CNO model in Equation \eqref{eq:CNO_2}.  For simplicity, we consider the case of dynamical system defined on a smooth compact sub-manifold $\mathcal{M}$ of $\mathbb{R}^{d}$, possibly with boundary; these types of dynamical systems arise often in physics \cite{DynamicsPhys,HamiltonianOrtega} and are actively studied in the reservoir computing literature \cite{Takens2022}.
We let $(g_{t})_{t \in \mathbb{N}}$ be a sequence of smooth functions from $\mathbb{R}^d$ to itself which fix the manifold $\mathcal{M}$, namely,\ $g_{t}(\mathcal{M}) \subseteq \mathcal{M}$ for every $n \in \mathbb{N}$.  We further require that the family $(g_{t})_{t\in\mathbb{N}}$ has uniformly bounded gradient on $\mathcal{M}$; meaning that for some $\lambda \ge 0$ it holds 
\[
    \sup_{t \in \mathbb{N}}\,\max_{x\in \mathcal{M}}\,
        \|\nabla g_{t}(x)\|
    \le 
    \lambda
    .
\]
NB, this is of-course satisfied by any autonomous dynamical system; namely when $g_{t}=g_0$ for all integers $t$, with $g_0$ smooth.

Then the restriction of each $g_{t}$ to $\mathcal{M}$ defines a dynamical system and we can express the causal structure in the orbit of any initial state $x_0 \in \mathcal{M}$ evolving under $g$ as a smooth causal map\footnote{See Definitions \ref{def:causal_maps}.}. To see this, consider the path space $\mathcal{X}$ whose elements are sequences $x_{\cdot} \in \mathcal{M}^{\mathbb{Z}}$ of the following form 
\begin{equation*}
    x^{(t)} \eqdef 
    \begin{cases}
        g_{t} \circ \ldots \circ g_{0}(x_0) &\mbox{if } t>0\\ 
        x_0 &\mbox{if } t\leq 0.\\
    \end{cases}
\end{equation*}
Now, let $\mathcal{Y}  \eqdef (\mathbb{R}^{d})^{\mathbb{Z}}$.  Then, by construction, we immediately deduce that the operator $f\,:\,\mathcal{X}\rightarrow\mathcal{Y}$ defined as
\begin{equation}
\label{eq:dynamical_system_formulation}
    f(x_{\cdot})_{t} \eqdef 
    \begin{cases}
        g_{t}(x^{(t)}) & \mbox{if } t>0\\
        x_{0} & \mbox{if } t\leq 0,\\
    \end{cases}
\end{equation}
defines a $(0, \infty, \lambda)$-smooth causal map. 

\paragraph{The Quantitative Advantage of the Hypernetwork for Approximating Causal Maps}
\hfill\\
We fix a positive integer $T$ and a $1$-Lipschitz function $G:\mathbb{R}^2\rightarrow [0,1]$.  For any input sequence $(z_t)_{t=1}^T\in [0,1]^T$ define the output sequence $(z^{(t)})_{t=1}^T\in [0,1]^T$ by
\begin{equation}
\label{eq:recursive_structure}
\begin{aligned}
    z^{(t)} & \eqdef  G(z_{t},z^{(t-1)}), &\qquad t=1,\dots,T
,
\end{aligned}
\end{equation}
where we set $z^{(0)}  \eqdef  0$.  We define the map $f:[0,1]^T\rightarrow \mathbb{R}$ as follows
\[
    f(z_1,\dots,z_T) \eqdef z^{(T)} = G(z_{T},z^{(T-1)})
.
\]
Evidently, $f$ is causal, whence, it can be approximated both by the CNO model or by a neural filter (which in this setting reduces to a deep ReLU FFNN). Comparing the approximation rates in either case in Tables~\ref{tab:Model_Complexity_DynamicCase} and~\ref{tab:Model_Complexity} we see that an approximation by a deep ReLU network (i.e.\ a neural filter in this case) requires a depth of $\tilde{O}(\varepsilon_A^{-T/2})$ and a width of $\tilde{O}(\varepsilon_A^{-T/2})$ to approximate $f$ uniformly on $[0,1]^T$ to a maximal error of $\varepsilon_A$.  In contrast, a CNO model only requires a latent state dimension $P([d])+Q = \tilde{O}( \varepsilon_A^{-6} -\log_{1/2}(T-1) )$ with hypernetwork $\hat{h}$ of depth $\tilde{O}(T^{3/2})$ and width $\tilde{O}(\varepsilon_A^{-6} - \log_{1/2}(T-1)T)$ in order to achieve the same uniform approximation of $f$ on $[0,1]^T$ with a maximal error of $\varepsilon_A$.  
\hfill\\
As shown in \cite[Theorem 2.4]{ZHS2022JMPA}, the ReLU feedforward networks achieve the optimal approximation rates when approximating arbitrary Lipschitz functions, then, our rates in Theorem~\ref{thm:theorem_universality_causal} imply that the CNO achieves super-optimal rates when approximating generic Lipschitz functions of the form in~\eqref{eq:recursive_structure}.  Moreover, a direct examination of the above rates shows that the CNO is not cursed by dimensionality when measured in the number of time steps one wishes the uniform approximation to hold for, while deep ReLU FFNNs are.  Consequently, this shows that CNOs are highly advantageous for (causal) sequential learning tasks from the approximation theoretic perspective.

\section{Conclusion}
\label{s:Conclusion}
We presented a first universal approximation theorem which is both causal, quantitative, compatible with infinite-dimensional operator learning, and which is not restricted to ``function spaces'' but is compatible with general ``good'' infinite-dimensional linear metric spaces.  Our main contributions, Theorem~\ref{thm:theorem_universality_static} and Theorem~\ref{thm:theorem_universality_causal}, provided approximation guarantees for any smooth or H\"{o}lder (non-linear) operator between Fr\'{e}chet spaces in the ``static'' or ``causal'' case, where temporal structure is or is not present in the approximation problem, respectively.  

We showed how the CNO model can approximate a variety of solution operators, and infinite dimensional dynamical systems, arising in stochastic analysis.  Moreover, in the Euclidean case, we showed that our neural filter's approximation rates are optimal.  We then showed that, when the target operator being approximated is a dynamical system, then the CNO's approximation rates are super-optimal.  Optimality is quantified in terms of the number of parameters required to approximate any arbitrary map belonging to some broad class as in constructive approximation theory of \cite{DVL1993}.  

We believe the observations made in this work open up avenues for future literature.  As a prime example, we would like to further optimize our CNO for the stochastic filtering problem assuming additional structural conditions.  As future work, we aim to build on these results in the context of robust finance.

\section*{Acknowledgments}
The authors would like to thank Alessio Spagnoletti for his helpful feedback. This research was funded by the NSERC Discovery grant (RGPIN-2023-04482) and was partially supported by the Research Council of Norway via the Toppforsk project Waves and Nonlinear Phenomena (250070).

\appendix

\section{Background material for proofs}
\label{sec:Background_for_proofs}
In an effort to keep the paper as self-contained as possible, this appendix contains any background material required in the derivations of our main results but not required for their formulation.  We cover various properties of deep ReLU neural networks, covering and packing results, and we overview some properties of finite-dimensional ``linear dimension reduction'' techniques in well-behaved Fr\'{e}chet spaces.  
We also include a list of some useful properties of generalized inverses.  

\subsection{Neural Network Regressors}
\label{subsec:DNN_background}
This section contains auxiliary results on neural network approximation, parallelization, and memorization.  
\subsubsection{DNN Approximation for Smooth and H\"{o}lder Functions}\label{subsubsec:DNN_approx_smooth_functions}
\noindent Theorem 1.1 in \cite{JZHS2021SIAM} proves that ReLU FFNNs with width $\mathcal{O}(N\,\log(N))$ and depth $\mathcal{O}(L\,\log(L) + d)$ can approximate a function $f \in C^{s}([0, 1]^{d})$ with a nearly optimal approximation error 
$\mathcal{O}(\|f\|_{C^{s}([0, 1]^{d})}\,N^{- 2s/d}\,L^{-2s/d})$, where the norm $\|\cdot\|_{C^{s}([0,1]^{d})}$ is defined as:
\begin{equation}\label{eq:norm_in_C_s}
    \|f\|_{C^{s}([0, 1]^{d})} \eqdef \max\{\|\partial^{\boldsymbol{\alpha}} f\|_{L^{\infty}([0,1]^{d})}\,:\,|\boldsymbol{\alpha}| \leq s,\,\boldsymbol{\alpha} \in \mathbb{N}^{d}\}, \quad
    f \in C^{s}([0,1]^{d}).
\end{equation}

More precisely, they state and prove the following
\begin{theorem}[{\cite{JZHS2021SIAM}}]\label{thm:approximation_in_C_s}
Given a function $f \in C^{s}([0,1]^{d},\mathbb{R})$ with $s \in \mathbb{N}_{+}$, for any $N, L \in \mathbb{N}_{+}$, there exists a function $\phi$ implemented by a ReLU FFNN with width $C_1\,(N+2)\,\log_2(8N)$ and depth $C_2\,(L+2)\,\log_{2}(4L) + 2 d$ such that
\begin{equation}\label{eq:approximation_in_C_s}
    \|\phi - f\|_{L^{\infty}([0,1]^{d})} \leq C_3\,\|f\|_{C^{s}([0,1]^{d})}\,N^{-2s/d}\,L^{-2s/d},
\end{equation}
where $C_1 = 17 s^{d+1} 3^{d} d$, $C_2 = 18 s^2$ and $C_3 = 85 (s+1)^{d} 8^{s}$.
\end{theorem}
\noindent In particular, note that the previous result does not privilege the width to the depth and vice versa because the exponent for \textit{both} $N$ and $L$ on the right-hand side of Equation \eqref{eq:approximation_in_C_s} is $-2s/d$.\\
\indent     On the other hand, \cite{ZHS2022JMPA}, as a consequence of their main theorem for explicit error characterization, state and prove the following.
\begin{theorem}[{\cite{ZHS2022JMPA}}]\label{thrm:approximation_in_holder}
    Given a H\"older continuous function on $[0,1]^{d}$ of order $\alpha \in (0, 1]$ with H\"older constant $\lambda>0$, i.e., $f \in C^{\lambda}_{\alpha}([0,1]^{d},\mathbb{R})$, then for any $N \in \mathbb{N}_{+}$, $L \in \mathbb{N}_{+}$ and $p \in [1, \infty]$, there exists a function $\varphi$ implemented by a ReLU network with width $C_1\max\{d \lfloor N^{1/d} \rfloor, N+2 \}$ and depth $11 L + C_2$ such that \begin{equation}\label{eq:approximation_in_L_p_holder}
        \|f - \varphi\|_{L^{p}([0,1]^{d})} \leq 131 \lambda \sqrt{d} (N^2 L^2 \log_3(N+2))^{-\alpha/d},
    \end{equation}
    where $C_1 = 16$ and $C_2 = 18$ if $p \in [1, \infty)$; $C_1 = 3^{d+3}$ and $C_2 = 18+2 d$ if $p=\infty$.
\end{theorem}

\subsubsection{Efficient parallelization of ReLU neural networks}
\label{subsubsec:Eff_approx_high_dim_functions_with_nn}
\cite{CJR2020IEEE} propose an efficient parallelization of neural networks with different depths for a special class of activation functions, namely the ones that have the so-called $c$-identity requirements. Before giving a formal definition of such activation functions, we remind  some quantities introduced in \cite{CJR2020IEEE}. More precisely, $\mathcal{N}$ denotes the set of neural network skeletons, i.e.,
\begin{equation}\label{eq:neural_network_skeleton}
    \mathcal{N} = \bigcup_{D \in \mathbb{N}}\; \bigcup_{(l_0, \ldots, l_D) \in \mathbb{N}^{D+1}} \;\prod_{k=1}^{D}(\mathbb{R}^{l_k \times l_{k-1}} \times \mathbb{R}^{l_k}),
\end{equation}
where we follow the convention that the empty Cartesian product is the empty set.  
For $\phi \in \mathcal{N}$, the quantity $\mathcal{D}(\phi) = D$ indicates the depth of $\phi$, $l_k^{\phi} = l_k$ the number of neurons in the $k$th layer, $k \in \{0, \ldots, D\}$, and $\mathcal{P}(\phi) = \sum_{k = 1}^{D} l_k (l_{k-1}+1)$ the number of network parameters.\\
\noindent If $\phi \in \mathcal{N}$ is given by $\phi = [(V_1, b_1), \ldots, (V_D, b_D)]$, $\mathcal{A}_{k}^{\phi} \in C(\mathbb{R}^{l_{k-1}}, \mathbb{R}^{l_{k}})$, $k \in \{1, \ldots, D\}$, denotes the affine function $x \rightarrow V_{k} x + b_k$. In addition, $a : \mathbb{R} \rightarrow \mathbb{R}$ indicates a continuous activation function which can be naturally extended to a function from $\mathbb{R}^{d}$ to $\mathbb{R}^{d}$, $d \in \mathbb{N}_{+}$ applying $\alpha$ component-wise. Finally, the $a$-realization of $\phi \in \mathcal{N}$ is the function $\mathcal{R}_{a}^{\phi} \in C(\mathbb{R}^{l_0}, \mathbb{R}^{l_D})$ given by:
\begin{equation}\label{eq:a_realization}
    \mathcal{R}_a^{\phi} = \mathcal{A}_{D}^{\phi} \circ a  \circ  \mathcal{A}_{D-1}^{\phi} \circ \cdots a \circ \mathcal{A}_{1}^{\phi}.
\end{equation}
We give now the following definition (cfr. \cite{CJR2020IEEE}, Definition 4):
\begin{definition}\label{def:c_identity}
A function $a \in C(\mathbb{R}, \mathbb{R})$ fulfills the $c$-identity requirement for a number $c \geq 2$ if there exists $I \in \mathcal{N}$ such that $\mathcal{D}(I) = 2$, $l^{I}_{1} \leq c$, and $\mathcal{R}_{a}^{I} = \text{id}_{\mathbb{R}}$. 
\end{definition}
\noindent For our scopes, we note that the ReLU activation fulfills the 2-identity requirement with $I = [([1\,-1]^{T}, [0\,\,\,0]^{T}), ([1\,-1], 0)]$. In addition, the following proposition hold (cfr. \cite{CJR2020IEEE}, Proposition 5):
\begin{proposition}\label{prop:number_parameters_c_identity}
    Assume that $a \in C(\mathbb{R}, \mathbb{R})$ fulfills the $c$-identity requirement for a number $c \geq 2$ with $I \in \mathcal{N}$. Then, the parallelization $p_I : \bigcup_{n \in \mathbb{N}} \mathcal{N}^{n} \rightarrow \mathcal{N}$ satisfies:
    \begin{equation}\label{eq:bound_parameters}
        \mathcal{P}(p_{I}(\phi_1, \ldots, \phi_n)) \leq \left(\frac{11}{16}\,c^2\,l^2\,n^2-1\right)\sum_{j = 1}^{n} \mathcal{P}(\phi_j)
    \end{equation}
    for all $n \in \mathbb{N}$ and $\phi_1, \ldots, \phi_n \in \mathcal{N}$, where $l = \max_{j \in \{1,\ldots,n\}} \max\{l_0^{\phi_j}, l_{\mathcal{D}(\phi_j)}^{\phi_j}\}$. In particular, $p_{I}(\phi_1, \ldots, \phi_n)$ denotes the parallelization of $\phi_1, \ldots, \phi_n$.
\end{proposition}

\subsubsection{Memory Capacity of Deep ReLU regressor}
\label{subsec:embedding}
\noindent   We here report a very recent lemma\footnote{\cite[Lemma 20]{KDDWP2022}.} appearing in the deep metric embedding paper of \cite{KDDWP2022}; see Lemma 20 in the just cited reference. 

\noindent  For the sake of completeness, we remind that the \emph{aspect-ratio} of the finite metric space $(\mathcal{X}_{N}, \|\cdot\|_2)$ is defined as the ratio of the maximum distance between any two points therein over the minimum separation between any two distinct points, i.e.:
\begin{equation}\label{eq:aspect_ratio_finite}
    \operatorname{aspect}(\mathcal{X}_N, \|\cdot\|_2) \eqdef \frac{\max_{x_i, x_j \in \mathcal{X}_N}\|x_i-x_j\|_2}{\min_{x_i, x_j \in \mathcal{X}_N\\ x_i \neq x_j}\|x_i-x_j\|_2}.
\end{equation}
\noindent We notice that \cite{KLMN2005GFA} introduce the notion of an aspect ratio of a measure space as the ratio of total mass over the minimum mass at any point. The relevance of the aspect ratio to our analysis is that it quantifies the difficulty to memorize a dataset.  This is because finite subset of a Euclidean space with large aspect ratio are logarithmically (in the aspect ratio) more difficult to memorize than subsets with a small aspect ratio.

\begin{lemma}\label{lem:embedding}
Let $n, d, N \in \mathbb{N}_{+}$, let $f\,:\,\mathbb{R}^{n}\rightarrow\mathbb{R}^{d}$ be a function, and consider pair-wise distinct $x_1, \ldots, x_{N} \in \mathbb{R}^{n}$. There exists a deep ReLU networks $\mathcal{NN}\,:\,\mathbb{R}^{n}\rightarrow\mathbb{R}^{d}$ satisfying
$$\mathcal{NN}(x_i) = f(x_i),$$
for every $i = 1, \ldots, N$. Furthermore, the following quantitative ``model complexity estimates" hold
\begin{itemize}
    \item[(\,i\,)] \textbf{Width\,:\,} $\mathcal{NN}$ has width $n(N-1) + \max\{d, 12\}$,
    \item[(\,ii\,)] \textbf{Depth\,:\,} $\mathcal{NN}$ has depth of the order of
    \begin{equation*}
    \hspace{-1.0cm}
        \mathcal{O}\left(N\left( 1+\sqrt{N\log(N)} \left(1 + \frac{\log(2)}{\log(N)} 
        \Bigl[  C_d + \frac{\log\left(N^2 \operatorname{aspect}(\mathcal{X}_N, \|\cdot\|_2)\right)}{\log(2)}\Bigr]%_{+}
        \right)\right)\right),
    \end{equation*}
    where $\mathcal{X}_{N} \eqdef \{x_1, \ldots, x_N\}$.
    \item[(iii)] \textbf{Number of non-zero parameters\,:\,} The number of non-zero parameters in $\mathcal{NN}$ is at most
    \begin{equation*}
    \begin{split}
        \mathcal{O}&\Bigl(N\left(\frac{11}{4}\max\{n,d\}N^2 -1\right)  
        \Bigl(d + \sqrt{N \log(N)} \Bigl( 1+\frac{\log(2)}{\log(N)}  \Bigl[ C_d 
        % \\
        % & 
        + \frac{\log\left(N^2 \operatorname{aspect}(\mathcal{X}_N, \|\cdot\|_2)\right)}{\log(2)}\Bigr]%_{+}
        \Bigr)\left(\max\{d,12\}\left(\max\{d,12\}+1\right)\right)\Bigr)\Bigr).
    \end{split}
    \end{equation*}
    The  ``dimensional constant" $C_d$ is defined by $$C_d  \eqdef  \frac{2 \log(5 \sqrt{2\pi}) + 3 \log(d)-\log(d+1)}{2 \log(2)}.$$
\end{itemize}
\end{lemma}

\subsection{Covering and packing numbers}\label{subsec:covering_and_packing}
% We remind here the concept of covering and packing; we refer to{\color{black}~\cite[page 147]{VanderVaart2023v2}}.   
In what follows, $\Theta$ will always have at least two points.  We recall the basic definitions of these objects here, and we refer the reader to, e.g.~\cite[Section 2.2.2]{VanderVaart2023v2}, for more details and relations between them.
% \luca{No! We are gonna use Dudley's book page 8, as serious people would do}
\begin{definition}[$\varepsilon$-covering]
    Let $(V, \|\cdot\|)$ be a normed space, and $\Theta \subset V$. A subset $F\subset V$ is an $\varepsilon$-covering (or $\varepsilon$-net) of $\Theta$ if for any $\theta \in \Theta$ there exists $f\in F$ such that $\|\theta - f\| \leq \varepsilon$.
    \end{definition}
    \begin{definition}[$\varepsilon$-packing]
    Let $(V, \|\cdot\|)$ be a normed space, and $\Theta \subset V$ a subset. $F\subset \Theta$ is an $\varepsilon$-packing of $\Theta$ if $\inf_{x,y\in F,\,x \neq y}\|x - y \| > \varepsilon$ (notice the inequality is strict).
\end{definition}
\noindent Both of these definitions define the notion of packing and covering.
\begin{definition}[Covering number] $N(\Theta, \|\cdot\|, \varepsilon) \eqdef \inf\{ \#(F)\,:\, F \text{ is an-}\varepsilon \text{ covering for }\,\Theta \}.$
\end{definition}
\begin{definition}[Packing number] 
$
            M(\Theta, \|\cdot\|, \varepsilon) 
    \eqdef 
        \sup\{ \#(F)\,:\, F\text{ is an-}\varepsilon \text{ packing for }\,\Theta\}.
$  
\end{definition}
{\color{black} We note that, $N(\Theta, \|\cdot\|, \varepsilon) \le M(\Theta, \|\cdot\|, \varepsilon) \le N(\Theta, \|\cdot\|, \varepsilon/2)$; see e.g.~\cite[page 147]{VanderVaart2023v2}.}

\subsection{Bounded Approximation Property in Fr\'{e}chet spaces with Schauder basises}\label{subsec:BAP}
\indent We now remind the following important definition (cfr.~\cite{BONET2020WP} Definition 1.6) and proposition (cfr.~\cite{BONET2020WP} Proposition 1.16 (2)).
\begin{definition}[Bounded Approximation property]\label{def:bounded_approximation_property}
    A locally convex space $E$ has the bounded approximation property (BAP, henceforth) if there exists an equi-continuous net $(A_{j})_{j \in I} \subset L(E)$, with $\text{dim}(A_{j}(E)) < \infty$ for every $j \in E$ and $\lim_{j \in I} A_{j}(x) = x$ for every $x \in E$. In other words, the net $(A_{j})_{j \in I}$ converges to the identity for the topology of point-wise or simple convergence. In all the previous expressions, $I$ denotes a generic directed indexing set.
\end{definition}

\begin{proposition}\label{prop:proposition_BONET_116}
    If $F$ is a barreled locally convex space with a Schauder basis, then $F$ has the BAP.
\end{proposition}
\noindent Since every Fr\'echet space $F$ is barreled\footnote{See \cite[Theorem 4.5]{O2014}.}, then $F$ will enjoy the BAP as soon as it admits a Schauder basis. We also have the following:\footnote{All the authors warmly thank Prof.~Jos\'e Bonet for providing us a precise reference on the following fact.} if $(A_{j})_{j \in \mathbb{N}}$ is a sequence of continuous linear operators from $E$ onto itself such that $A_{0}(x) \eqdef \lim_{n \rightarrow \infty} A_j(x)$ exists for every $x \in E$, then $(A_{j})_{j \in \mathbb{N}}$ is equicontinuous by the Banach-Steinhaus\footnote{See, e.g., \cite{K1983}, Result 39.1 Page 141).} theorem for Fr\'echet spaces, $A_0$ is a continuous linear operator, and the sequence $(A_{j})_{j \in \mathbb{N}}$ converges to $A_0$ uniformly on the compact subsets of $E$.\\
\indent Also, we have the following proposition regarding finite-dimensional topological vector spaces:
\begin{proposition}\label{prop: unique vst}
    A finite-dimensional vector space $F$ can have just one vector space topology up to homeomorphism.
\end{proposition}

\begin{remark}\label{rmk: eli eli lama sabachthani}
We observe the following characterization for an equi-continuous family $H\subset L(E,F)$, with $E,F$ Fr\'echet spaces.
\begin{itemize}
    \item $H\subset L(E,F)$ is an equi-continuous family if and only if
    \item for any $V\subset F$ open neighborhood of the origin, $\cap_{T\in H}T^{-1}(V)$ is an open neighborhood of the origin
    (\cite{BONET2020WP} page 1), if and only if
    \item for any $V\subset F$ open neighborhood of the origin, there exists $U\subset E$ open neighborhood of the origin such that $\cup_{T\in H}T(U)\subset V$.
\end{itemize}
In this last case, we call the family $H$ uniformly equi-continuous (see \cite{K1983}, page 169).

\end{remark}

\section{Proofs}\label{app:proofs}
\subsection{Proof of Lemma \ref{lem:lemma_Ck_stability}}\label{app:proof_of_lemma_Ck_stability}
\begin{proof}
By assumption, $f:E \rightarrow B$ is $C^{k}$-$\textrm{Dir}$. This means that
\begin{equation*}
    D^{k} f : E \times E^{k} \rightarrow B,\quad (x, h_1, \ldots, h_k) \rightarrow D^{k} f(x)\{h_1, \ldots, h_k\}
\end{equation*}
is continuous, jointly as a function on the product space. Moreover, an arbitrary linear and continuous operator $T: E \rightarrow B$ between two Fr\'echet spaces is trivially $C^{k}$-$\textrm{Dir}$, for any $k$. By implication, $\tilde{I}$ and $\tilde{P}$ are $C^{k}$-$\textrm{Dir}$. By Theorem 3.6.4 in \cite{H1982AMS} (chain rule), $\tilde{P} \circ f \circ \tilde{I}$ is $C^{k}$-$\textrm{Dir}$. In other words,
\begin{equation*}
    D^k (\tilde{P} \circ f \circ \tilde{I}): \mathbb R^n \times (\mathbb R^n)^k \to \mathbb R^m,\quad (x,h_1,\dots ,h_k)\mapsto D^k(\tilde{P} \circ f \circ \tilde{I})(x)\{h_1,\dots, h_k\}
\end{equation*}
is jointly continuous in the product space. To conclude the proof, it is sufficient to choose as directions $\{h_1, \ldots, h_k\}$ in the previous expression the following ones: $h_1 = e_{j_1}, \ldots, h_k = e_{j_k}$, being $\{e_1, \ldots, e_n\}$ the canonical basis of $\mathbb{R}^{n}$. In this case, we obtain:
\begin{equation*}
    D^{k}(\tilde{P} \circ f \circ \tilde{I})(x)\{h_1, \ldots, h_k\} = \partial_{j_1, \ldots, j_k} (\tilde{P} \circ f \circ \tilde{I})(x),
\end{equation*}
which is, as a function of $x$ only, continuous. Thus, we see that all the partial derivatives of order $k$ of $(\tilde{P} \circ f \circ \tilde{I})$ are continuous on $\mathbb{R}^n$, and so $(\tilde{P} \circ f \circ \tilde{I})$ is $C^{k}$ in the usual sense. Namely, $f$ is $C^{k}$ stable.
\end{proof}

\noindent Before proceeding, we state and prove the following Lemma.
\begin{lemma}\label{lem:aux_lemma_2}
Let $(X, d)$ and $(Y, \varrho)$ be two metric spaces and let $\mathcal{F} \subset C(X, Y)$ be a family of maps from $X$ to $Y$ such that $\forall \varepsilon > 0$ $\exists \delta>0\,:\,d(x, x') \leq \delta$, then $\varrho(f(x), f(x')) \leq \varepsilon$, $f \in \mathcal{F}$. Then, the family $\mathcal{F}$ has a common modulus of continuity.
\end{lemma}
\begin{proof}
    Le $\omega\,:\,[0, \infty) \rightarrow [0, \infty]$ be defined as:
    \begin{equation*}
        \omega(\delta) \eqdef \sup\{\varrho(f(x), f(x'))\,:\,d(x, x') \leq \delta,\,f \in \mathcal{F}\}.   
    \end{equation*}
\noindent   It holds that: (\,i\,) $\omega(0) = 0$; (\,ii\,) $\omega(\delta) \in [0, +\infty]$, $\delta>0$, but $\omega(\delta) < \infty$ in a neighborhood of $0$; (\,iii\,) $\omega$ is non decreasing; (\,iv\,) continuity at $0$\,:\, it holds that $\lim_{\delta \rightarrow 0^{+}} \omega(\delta) = \inf_{\delta > 0} \omega(\delta)  \eqdef  \ell \in [0, +\infty)$. In order to prove the statement, we have to prove that $\ell=0$. Assume by contradiction that $\ell>0$ and let $(\delta_n)_{n \in \mathbb{N}}$ a decreasing sequence to zero such that $\omega(\delta_n)$ converges toward $\ell$ from above. By definition of $\sup$, $\exists\,x_n, x_n' \in X\,:\,d(x_n, x_n') \leq \delta_n$ and $f_n \in \mathcal{F}\,:\,\varrho(f_n(x), f_n(x_n'))>\ell/2$, $n \in \mathbb{N}$. Now, set $\varepsilon = \ell/4$ in the definition of uniform continuity and choose $\delta > 0$ accordingly, i.e.,
    \begin{equation*}
        d(x, x') \leq \delta \Rightarrow \varrho(f(x), f(x')) \leq \ell/4,\quad f \in \mathcal{F}.
    \end{equation*}
Now, pick a $\delta_{n_0} < \delta$. Because $d(x_{n_0}, x_{n_0}') \leq \delta_{n_0} < \delta$, we have that the following inequality holds  $\varrho(f_{n_0}(x_{n_0}), f_{n_0}(x_{n_0}')) \leq \ell/4$, which is a contradiction. Finally, given $z, z' \in X$, $z \neq z'$, by definition it holds that:
    \begin{equation*}
        \varrho(f(x), f(x')) \leq \omega(d(z, z')),\,\,\text{for any}\,\,x, x'\,:\,d(x, x') \leq d(z, z'),\,\,f \in \mathcal{F}.
    \end{equation*}
In particular it holds for $x = z$ and $x' = z'$, i.e. $\varrho(f(z), f(z')) \leq \omega(d(z, z')),\,f\in\mathcal{F}$. Notice that if $z = z'$, than the statement is trivial. 
\end{proof}

\begin{remark}\label{rmk: eli eli lama sabachthani 2}
We observe that, in view of Remark \ref{rmk: eli eli lama sabachthani} and the fact that the metric of a Fr\'echet space is translation-invariant, an equi-continuous family $H\subset L(E,F)$, with $E,F$ Fr\'echet spaces, satisfies the assumption of Lemma \ref{lem:aux_lemma_2}.

\end{remark}

\subsection{Proof of Theorem~\ref{thm:theorem_universality_static}}\label{app:proof_of_theorem_universality_static}

{The proof of Theorem~\ref{thm:theorem_universality_static} proceeds in three main steps. First, the target nonlinear operator is replaced by a finite-dimensional surrogate that preserves its regularity properties—precisely, uniform continuity and a prescribed degree of smoothness. This finite-dimensional surrogate is then approximated using a (P)ReLU MLP. Finally, an infinite-dimensional approximator—our neural filter—is constructed by projecting any infinite-dimensional input onto a finite-dimensional subspace, passing the result through the (P)ReLU MLP, and interpreting the MLP’s outputs as coefficients in a Schauder basis, which are then reassembled into an infinite-dimensional prediction. Tracking and controlling the approximation errors introduced at each step completes the proof.}

\begin{proof}
\noindent In order to outline the ideas behind Theorem~\ref{thm:theorem_universality_static}, we draw the diagram chase in Figure~\ref{fig:tikz_static_proof_summary}. Moreover, in order not to burden the notations, we will use the following abbreviations for any ``encoding error" $\varepsilon_{D}$: $n^{in} \eqdef n^{in}_{\varepsilon_D}$ and $n^{out} \eqdef n^{out}_{\varepsilon_D}$. In what follows, we detail the proof for the case that\footnote{See Definition \ref{def:Ck_traceclass}.} $f \in C^{k, \lambda}_{\operatorname{tr}}(K,B)$. The case where $f$ belongs to $C_{\alpha,\text{tr}}^{\lambda}(K, B)$ will be treated at the end of the \emph{Proof} for the sake of clarity, and we will highlight the main differences with respect to the $C^{k, \lambda}_{\operatorname{tr}}(K,B)$ case.

\begin{figure}[!h]%[H]
    \centering
\begin{tikzcd}
&K \arrow{d}{\iota} \arrow[swap]{dl}{I_{E:n^{in}}\circ P_{E:n^{in}}} \arrow{dr}{f} \\
I_{E:n^{in}}\circ P_{E:n^{in}}(K)
\arrow[hookrightarrow]{r}{\iota}
\arrow{d}{P_{E:n^{in}}}
& E \arrow[swap]{d}{P_{E:n^{in}}} \arrow{r}{F} & B \\%
P_{E:n^{in}}(K) \arrow{r}{\iota} \arrow{d}{\psi} &
\big(R^{n^{in}}, d_{E:n^{in}}\big)
\arrow[swap]{d}{\psi}
\arrow{r}{\phi \circ \hat{f}_{\theta}\circ \psi}
& 
\big(R^{n^{out}}, d_{B:n^{out}}\big)
\arrow{u}{I_{B:n^{out}}}
\\
P_{E:n^{out}}(K) \arrow{r}{\iota} &
\big(R^{n^{out}}, \|\cdot\|_2\big)
\arrow{r}{F_{\varepsilon_D}}
& 
\big(R^{n^{out}}, \|\cdot\|_2\big)
\arrow{u}{\phi}
\end{tikzcd}
\caption{Outline of Theorem~\ref{thm:theorem_universality_static}'s proof: The diagram chase.}
\label{fig:tikz_static_proof_summary}
\end{figure}

By assumption, $f:K \rightarrow B$ belongs to the trace-class $C^{k, \lambda}_{\operatorname{tr}}(K,B)$. Therefore, there exists a $\lambda$-Lipschitz $C^{k}$-stable (non-linear) operator $F:E \rightarrow B$ such that $F(x) = f(x)$ for every $x \in K$. Whence, it is sufficient to approximate $F$, and then restrict $F$ to $K$ to deduce an estimate on $f$. Without loss of generality, we can assume that the function $f$ is not constant.\\
\indent 
To shorten the notation, we now set for $n \in \mathbb{N}$ the map $A_{E:n}$ in the following way $A_{E:n} \eqdef I_{E:n} \circ P_{E:n}\,:\,(E, d_{E}) \longrightarrow (E, d_{E})$. In particular, for every $x \in E$ it holds that $A_{E:n}(x)  =  \sum_{h=1}^{n} \langle \beta_h^{E}, x\rangle e_h$, where, we remind, $(\langle \beta_h^{E},x\rangle)_{h = 1}^{\infty}$ is the \textit{unique} real sequence satisfying the following equality $x = \sum_{h=1}^{\infty} \langle \beta_h^{E},x\rangle e_h$. It is manifest that these maps $A_{E:n}$ are linear, continuous, with finite dimensional range,
and converging to the identity of $E$ as $n\to\infty$, i.e. they are equi-continuous. %By Banach-Steinhaus's theorem for Frechet spaces\footnote{See, e.g., \cite{K1983}, Result 39.1 Page 141.}, they are (uniformly) equicontinuous. We then see that they satisfy Definition  \ref{def:bounded_approximation_property}.    

Let define $\omega_{A, E}\,:\,[0,\infty) \rightarrow [0,\infty)$ the modulus of continuity of the family $(A_{E:n})_{n \in \mathbb{N}}$, which we get from Lemma \ref{lem:aux_lemma_2} and Remark \ref{rmk: eli eli lama sabachthani 2}. {\color{black} We note that $\omega_{A,E}$ is non-decreasing.}
% Since $\omega_{A,E}$ might be not non-decreasing, with a slight abuse of notation we re-define it as $\frac{1}{t} \int_t^{2t} \sup_{s\leq x}\, \omega_{A, E}(s)\,dx$, obtaining now the sought non-decreasing property. 
Moreover, let $\omega_{A, E}^{\dagger}$ be the generalized inverse of $\omega_{A, E}$; see Subsection \ref{subsec:generalized_inverse}. A similar reasoning done into the Fr\'echet space $B$ with $A_{B:n}$ defined similarly to $A_{E:n}$ leads to the existence of a continuous non-decreasing modulus of continuity $\omega_{A, B}:[0,\infty)\to [0,\infty)$, whose generalized inverse will be denoted as $\omega_{A, B}^\dagger$ this time. 

Because of the equi-continuity of $(A_{E:n})_{n \in \mathbb{N}}$, for any ``encoding error" $\varepsilon_{D}$ there exists $n' \in \mathbb{N}_{+}$ such that, if $n \geq n'$, then the following estimation holds: $\max_{x \in K} d_{E}(A_{E:n}(x), x) < \frac{1}{\lambda}\omega^{\dagger}_{A, B}\left(\frac{\varepsilon_{D}}{2}\right)$; see the argument below Proposition \ref{prop:proposition_BONET_116} for a precise reference of the previous fact. 

Moreover, analogously as above, we derive the following inequality, because $F(K)$ is compact: $\max_{x \in F(K)} d_{B}(A_{B:n}(x), x) < \frac{\varepsilon_{D}}{2}$. Thus, the following positive integers
\begin{equation}\label{eq:dimensionality_approximation_theorem}
    \begin{split}
        n^{in}
            &  
            \eqdef  
        \inf\Big\{n \in \mathbb{N}_{+}\,:\,\max_{x \in K} d_{E}(A_{E:n}(x), x) \le \frac{1}{\lambda}\omega^{\dagger}_{A,B}\left(\frac{\varepsilon_{D}}{2}\right)\Big\},
        \\
        n^{out} &  \eqdef  \inf\Big\{n \in \mathbb{N}_{+}\,:\,\max_{y \in F(K)} d_{B}( A_{B:n}(y), y) \le \frac{\varepsilon_{D}}{2}\Big\},
    \end{split}
\end{equation}
are finite. At this point, we remind that $\psi$ and $\varphi$ are the following two set-theoretic identity maps 
\begin{equation}\label{eq:set_theoretic_identity}
    \psi\,:\,(\mathbb{R}^{n^{in}}, d_{E:n^{in}})\longrightarrow (\mathbb{R}^{n^{in}},\|\cdot\|_2),\quad \phi\,:\,(\mathbb{R}^{n^{out}},\|\cdot\|_2) \longrightarrow (\mathbb{R}^{n^{out}}, d_{B:n^{out}}),
\end{equation}
and we define the following map $\bar{F}:\,(\mathbb{R}^{n^{in}}, \|\cdot\|_2) \longrightarrow (\mathbb{R}^{n^{out}},\|\cdot\|_2)$ by $\bar{F} \eqdef  \phi^{-1} \circ P_{B:n^{out}} \circ F \circ I_{E:n^{in}} \circ \psi^{-1}$. Notice that since $\phi \circ P_{B:n^{out}}$ and $I_{E:n^{in}} \circ \psi^{-1}$ are continuous linear maps and $F$ is $C^{k,\lambda}$-stable by assumption, then $\bar{F} \in C^{k, \lambda}(\mathbb{R}^{n^{in}}, \mathbb{R}^{n^{out}})$.\\
\indent     Now, let $\hat{f}_{\theta} \in \mathcal{NN}_{[d]}^{\text{ReLU}}$ a deep ReLU neural network having \emph{complexity} $[d]  \eqdef  (d_0, \ldots, d_{J})$ for a multi-index $[d]$ and a $J \in \mathbb{N}_{+}$ such that $d_{0} = n^{in}$ and $d_{J} = n^{out}$.  Moreover, in order not to burden the notation, we set for $k \in \{E, B\}$ and $\ell \in \{in, out\}$, $I_{k}  \eqdef  I_{k:n^{\ell}}$, $P_{k}  \eqdef  P_{k:n^{\ell}}$ and, as before, $A_k \eqdef  I_k\circ P_k$. Then, the following estimate holds: 
\begin{align}
\label{eq:Proof_Static_Term_to_Bound}
\vspace{-0.5cm}
    &   \max_{x\in K}\,d_{B}\big(I_B \circ \phi \circ \hat{f}_{\theta}\circ \psi\circ P_E(x),f(x)\big)\\
=   & \max_{x\in K}\,d_{B}\big(I_B \circ \phi \circ \hat{f}_{\theta}\circ \psi\circ P_E(x),F(x)\big)\label{eq:extension_to_F}\\ 
\leq & 
    \max_{x\in K}\,
        d_{B}\big(
            I_B\circ\phi\circ\hat{f}_{\theta}\circ\psi\circ P_E(x),
            I_B\circ \phi\circ \phi^{-1}  \circ P_{B} \circ F \circ I_{E} \circ \psi^{-1} \circ \psi\circ P_E(x)\big)\label{eq:triangular_inequality}\\
  +& 
    \max_{x\in K}\,
        d_{B}\big(
            I_B\circ \phi\circ \phi^{-1} \circ P_{B} \circ F \circ I_{E} \circ \psi^{-1} \circ \psi\circ P_E(x), I_B\circ \phi\circ\phi^{-1} \circ P_B\circ F(x)\big)\nonumber\\
  +& 
    \max_{x\in K}\,
        d_{B}\big(I_B\circ \phi\circ\phi^{-1}\circ P_B\circ F(x), F(x)\big)\nonumber\\
%% Identities
= & 
\label{eq:Proof_Static_NeuralApprox_Term}
    \max_{x\in K}\,
        d_{B}\big(I_B\circ \phi  \circ \hat{f}_{\theta}\circ \psi\circ P_E(x),
            I_B \circ \phi \circ \bar{F} \circ \psi\circ P_E(x)\big)
\\
\label{eq:Proof_Static_E_DimRed_Estimate_}
+& 
    \max_{x\in K}\,
        d_{B}\big(I_B\circ P_B \circ F\circ I_E \circ P_E(x),
        I_B\circ P_B\circ F(x)\big)
\\
\label{eq:Proof_Static_B_DimRed_Estimate}
 +& 
    \max_{y\in f(K)}\,
        d_{B}\big(I_B\circ P_B(y),y\big),
\end{align}
where the equality in Equation \eqref{eq:extension_to_F} follows from the fact that on the compact $K$ the maps $f$ and $F$ coincides, the inequality in Equation \eqref{eq:triangular_inequality} follows from the triangular inequality by using the diagram chase in Figure \ref{fig:tikz_static_proof_summary}, and the equality in Equation \eqref{eq:Proof_Static_NeuralApprox_Term} from the definition of $\bar{F}$. We now bound each of the above terms \eqref{eq:Proof_Static_NeuralApprox_Term}, \eqref{eq:Proof_Static_E_DimRed_Estimate_} and \eqref{eq:Proof_Static_B_DimRed_Estimate}. We start from the last one: it is controlled, by using the definition of $n^{out}$ as:
\begin{equation}\label{eq:estimation_first}
    \max_{y \in f(K)} d_{B}(I_{B} \circ P_{B}(y), y) < \frac{\varepsilon_{D}}{2}.
\end{equation}
\indent We now bound the second term, i.e., the term $ \max_{x\in K}\,d_{B}\big(I_B\circ P_B \circ F\circ I_E \circ P_E(x), I_B\circ P_B\circ F(x)\big)$. Recall that $F$ is $\lambda$-Lipschitz. By using the definition of $n^{in}$ in \eqref{eq:dimensionality_approximation_theorem}, we have for $x\in K$: 
\begin{equation}
\label{eq:Proof_Static_E_DimRed_Estimate_1}
\begin{aligned}
&d_{B}\big(I_B\circ P_B \circ F\circ I_E \circ P_E(x), I_B\circ P_B\circ F(x) \big)\\
& \leq \omega_{A,B}\left[ d_B(F\circ I_E\circ P_E(x),F(x)) \right] \\
&\leq \omega_{A,B}\left[\lambda \,d_E(I_E\circ P_E(x),x) \right]
\\
&\leq  \omega_{A,B}\Big(\lambda \max_{x\in K}\,d_E\big(I_E \circ P_E(x),x\big)\Big) \leq 
    \omega_{A,B}\Big(
        \lambda \frac{1}{\lambda}\omega^{\dagger}_{A,B}\Big(\frac{\varepsilon_D}{2}\Big)\Big) = \frac{\varepsilon_{D}}{2},
\end{aligned}
\end{equation}
and hence $\max_{x\in K}d_{B}\big(I_B\circ P_B \circ F\circ I_E \circ P_E(x), I_B\circ P_B\circ F(x) \big)\leq \varepsilon_D/2$.

We now control the term \eqref{eq:Proof_Static_NeuralApprox_Term}. In order to do so, we make the following observations: (\,1\,) $(R^{n^{in}}, d_{E:n^{in}})$ is a topological vector space in which the topology coincides with the standard one; see Lemma \ref{lem:aux_lemma}; (\,2\,) therefore, the identity map and its inverse are continuous. (\,3\,) Being linear, it is also uniform continuous; see \cite{S1971}, Page 74. These observations allow us to define $\omega_{\varphi}\,:\,[0, +\infty) \rightarrow [0, +\infty)$ the modulus of continuity of the map $\varphi$ which we may assume to be, without loss of generality\footnote{See the argument done above for $\omega_{A, E}$.}, continuous and strictly monotone; $\omega_\varphi^\dagger$ will denote, as usual, its generalized inverse. This allows us to compute: 
\begin{equation}
    \label{eq:Proof_Static_NeuralApprox_Term__BOUNDED}
\begin{aligned}
&   \max_{x\in K}\,
        d_{B}\big(I_B\circ \phi \circ \hat{f}_{\theta}\circ \psi\circ P_E(x),
            I_B \circ \phi\circ \bar{F} \circ \psi\circ P_E(x)\big)
\\
\le &
    \max_{x\in K}\, d_{B:n^{out}}
       \Big(\phi \circ \hat{f}_{\theta}\circ \psi\circ P_E(x),
            \phi\circ \bar{F} \circ \psi\circ P_E(x)\Big)
\\
\le &
    \max_{x\in K}\, \omega_\phi
       \Big( \| \hat{f}_{\theta}
        \circ \psi\circ P_E(x)-\bar{F} \circ \psi\circ P_E(x)\|_2\Big)
\\
\le &
    \omega_{\phi}\Big(\max_{x\in K}\,
    \| \hat{f}_{\theta}
        \circ \psi\circ P_E(x)-\bar{F} \circ \psi\circ P_E(x)\|_2 \Big)
\\
= & \omega_{\phi}\Big(
        \max_{u\in \psi\circ P_E(K)}\,
        \|\hat{f}_{\theta}(u)-\bar{F}(u)\|_2\Big)
,
\end{aligned}
\end{equation}
where the second line of~\eqref{eq:Proof_Static_NeuralApprox_Term__BOUNDED} holds since $I_B$ is an isometric embedding, and thus in particular $\operatorname{Lip}(I_B)=1$.

We now remind that $\bar{F} \in C^{k, \lambda}(\mathbb{R}^{n^{in}}, \mathbb{R}^{n^{out}})$; by Theorem \ref{thm:approximation_in_C_s}, we can pick the above-mentioned ReLU neural network $\hat{f}_{\theta}$ in such a way that 
\begin{equation}
\label{eq:Proof_StaticResult_ReLUFFNN_Estimate}
      \max_{u\in \psi\circ P_E(K)}\,
        \| 
            \hat{f}_{\theta}(u)
        -
            \bar{F}(u)
        \|_2
    \leq 
        \omega_{\phi}^{\dagger}(\varepsilon_{A})
    =: 
        \delta,
\end{equation}
where $\varepsilon_{A}$ is the ``approximation error" as in the statement of the theorem; we will prove later on the existence of such $\hat{f}_{\theta}$. Meanwhile, we note that the bound in Equation \eqref{eq:Proof_Static_NeuralApprox_Term__BOUNDED} becomes:
\begin{equation*}
\max_{x\in K}\,
        d_{B}\big(I_B\circ \phi  \circ \hat{f}_{\theta}\circ \psi\circ P_E(x),
            I_B \circ \phi \circ \bar{F} \circ \psi\circ P_E(x)\big)
\leq  \omega_{\varphi} \Big(\omega^{\dagger}_{\varphi}\Big( \varepsilon_{A} \Big)\Big) \leq \varepsilon_{A}.
\end{equation*}
\noindent Putting together the previous equation with the estimates in Equations \eqref{eq:estimation_first} and \eqref{eq:Proof_Static_E_DimRed_Estimate_1}, we have that:
\begin{equation*}
\max_{x\in K}\,d_{B}\big(I_B\circ \phi\circ \hat{f}_{\theta}\circ \psi\circ P_E(x),f(x)\big)\leq 
        \varepsilon_D 
    + 
        \varepsilon_A
\end{equation*}

\indent Finally, we demonstrate the existence of a map $\hat{f}_{\theta}$, which ``\textit{depends upon some parameters}'' and that satisfies the estimates in Equation~\eqref{eq:Proof_StaticResult_ReLUFFNN_Estimate}. Before proceeding, we make the following considerations: (1) $\bar{F} \in C^{k,\lambda}(\mathbb{R}^{n^{in}}, \mathbb{R}^{n^{out}})$, where $\mathbb{R}^{n^{in}}$ and $\mathbb{R}^{n^{out}}$ are endowed with the Euclidean topology. (2) We can define, by using a reasoning similar to the one used for $\omega_{\varphi}$, $\omega_{\psi}\,:\,[0,+\infty) \rightarrow [0,+\infty)$ the modulus of continuity of the map $\psi$ which we may assume to be continuous and strictly monotone; $\omega_{\psi}^{\dagger}$ will denote its generalized inverse. (3) Moreover, the following estimates hold true:
\begin{equation*}
    \begin{split}
        d_{E:n^{in}}(P_{E}(x), P_{E}(y)) &= d_{E}\left(\sum_{h=1}^{n^{in}}\langle \beta_h^{E}, x\rangle e_h, \sum_{h=1}^{n^{in}}\langle \beta_h^{E}, y\rangle e_h\right)\\ 
                                                    &= d_{E}(A_{E}(x), A_{E}(y)) \leq \omega_{A, E}(d_{E}(x, y))\quad \forall x, y \in E.
    \end{split}
\end{equation*}
\noindent Now, let $\text{diam}_{E}(\cdot)$, $\text{diam}_{2}(\cdot)$ and $\text{diam}_{E:n^{in}}(\cdot)$ denote the \emph{diameter} computed with respect to the metric $d_{E}$, the Euclidean distance and the distance $d_{E:n^{in}}$ respectively. It holds that:
\begin{equation*}
    d_{E:n^{in}}(P_{E}(x), P_{E}(y)) \leq \omega_{A, E}(d_{E}(x, y)) \leq \omega_{A, E}(\text{diam}_E(K)),\quad \forall x, y \in K.
\end{equation*}
\noindent Moreover, it follows that:
\begin{equation*}
    \|\psi \circ P_{E}(x) - \psi \circ P_{E}(y) \|_2 \leq \omega_{\psi}(d_{E:n^{in}}(P_{E}(x), P_{E}(y))) \leq \omega_{\psi}(\omega_{A, E}(\text{diam}_{E}(K))), \quad \forall x, y \in K.
\end{equation*}
In particular, it holds that:
\begin{equation}\label{estimate:diameter}
    \text{diam}_2(\psi \circ P_{E}(K)) \leq \omega_{\psi}(\omega_{A, E}(\text{diam}_{E}(K))).
\end{equation}
\noindent We now identify a hypercube ``nestling" $\psi \circ P_{E:n^{in}}(K)$, and we explicit the dependence on $n^{in}$. To this end, let 
\begin{equation*}
    r_{K}  \eqdef   \omega_{\psi}(\omega_{A,E}(\text{diam}_{E}(K)))\sqrt{\frac{n^{in}}{2 (n^{in}+1)}}.
\end{equation*}
By Jung's Theorem\footnote{See \cite{J1901JFDRUAM}.}, there exists $x_0 \in \mathbb{R}^{n^{in}}$ such that the closed Euclidean ball $\overline{\text{Ball}_{(\mathbb{R}^{in},\|\cdot\|_2)}\left(x_0, r_{K} \right)}$ contains $\psi \circ P_{E:n^{in}}(K)$.  Now set, for rotational convenience, $\bar{1}  \eqdef  (1, \ldots, 1) \in \mathbb{R}^{n^{in}}$, and define the the following  affine function  $ W\,:\,(\mathbb{R}^{n^{in}}, \|\cdot\|_2)\rightarrow (\mathbb{R}^{n^{in}}, \|\cdot\|_2)$:
\begin{equation*}
    W\,:\,(\mathbb{R}^{n^{in}}, \|\cdot\|_2)\rightarrow (\mathbb{R}^{n^{in}}, \|\cdot\|_2)  \quad x \rightarrow W(x) \eqdef  (2 r_{K})^{-1} (x - x_0) + \frac{1}{2}\bar{1},
\end{equation*}
\noindent which is well-defined and invertible, and maps $\psi \circ P_{E:n^{in}}(K)$ to $[0, 1]^{n^{in}}$. In particular, the map 
\begin{equation}\label{eq:fbarcomposewminusone}
\bar{F} \circ W^{-1}\,:\,(\mathbb{R}^{n^{in}},\|\cdot\|_2)\rightarrow (\mathbb{R}^{n^{out}},\|\cdot\|_2)
\end{equation}
is of class $C^{k, \lambda}$: indeed,we already know that $\bar{F}$ is $C^{k, \lambda}$; pre-composing $\bar{F}$ with the smooth map $W^{-1}$ clearly produces an object of class $C^{k, \lambda}$.  As a consequence, if we denote by $(\bar{e}_i)_{i = 1}^{n^{out}}$ the standard orthonormal basis of $(\mathbb{R}^{n^{out}},\| \cdot \|_2)$, then the maps $\bar{f}_i \eqdef \langle \bar{F} \circ W^{-1}, \bar{e}_i \rangle$, $i \in [[ n^{out}]]$, are of class $C^{k, \lambda}$; where here, $\langle\cdot ,\cdot\rangle$ is the standard Euclidean scalar product.  Moreover, by construction, for each $x\in \rr^{n^{in}}$ it holds that 
\begin{equation}\label{identity:translation_to_cube}
    \begin{aligned}
        \sum_{i=1}^{n^{out}} \bar{f}_i(x)\bar{e}_i = \bar{F} \circ W^{-1}(x).
    \end{aligned}
\end{equation}
\noindent Therefore, we may apply Theorem~\ref{thm:approximation_in_C_s} to $\bar{F}\circ W^{-1}$ (restricted to the unit cube) $n^{out}$ times to deduce that there are $n^{out}$ ReLU FFNN $\hat{f}_{\theta}^{(i)}\,:\,\mathbb{R}^{n^{in}}\rightarrow\mathbb{R}$, $i \in [[ n^{out}]]$, satisfying to the following estimate
\begin{equation}\label{estimate:univariateFFNNestimate_on_fdim_cube}
    \begin{aligned}
        \max_{i = 1, \ldots, n^{out}}\sup_{x\in [0,1]^{n^{in}}}\, |\bar{f}_i(x)-\hat{f}_{\theta}^{(i)}(x)|\leq \frac{\delta}{\sqrt{ n^{out}}}.
    \end{aligned}
\end{equation}

\noindent 
In the notation of Theorem~\ref{thm:approximation_in_C_s}, if we set, $C_3\eqdef\,\max_{i=1,\ldots, n^{out}} \,\|\bar{f}_i\|_{C^{k}([0,1]^{n^{in}})}\, N^{-2k/n^{in}}\, L^{-2k/n^{in}} = \delta/(n^{out})^{1/2}
$ and we also set $N=L$ then, the same result implies that the width and the depth of each $\hat{f}_{\theta}^{(i)}$ is provided in the same reference and, 
upon recalling the definition of $\delta$ in~\eqref{eq:Proof_StaticResult_ReLUFFNN_Estimate} we find that it is given by:
\begin{itemize}
\item[(i)] \textbf{\emph{Width}\,:\,}
    \begin{equation}\label{eq:width_approximation_theorem_proof}
    \hspace{-1.5cm}
        \begin{split}
        & 
                C_1\biggl(
                        % N - BEGIN
                        \big\lceil
                            (C_3C_{\bar{f}})^{n^{in}/4k}
                                \,
                            (n^{in})^{n^{in}/8k}
                                \,
                            [\omega_{\phi}^{\dagger}(\varepsilon_{A})]^{-2k/n^{in}}
                        \big\rceil
                        % N - END
                    +
                        2
                \biggr)
            \cdot 
                \log_2\biggl(
                    8
                    \,
                    % N - BEGIN
                    \big\lceil
                        (C_3C_{\bar{f}})^{n^{in}/4k}
                            \,
                        (n^{in})^{n^{in}/8k}
                            \,
                        [\omega_{\phi}^{\dagger}(\varepsilon_{A})]^{-2k/n^{in}}
                    \big\rceil
                    % N - END
                \biggr)
\end{split}
\end{equation}
\item[(ii)] \textbf{\emph{Depth}\,:\,}
\begin{equation}\label{eq:depth_approximation_theorem_proof}
    \hspace{-1.5cm}
        \begin{split}
            &
            C_2\,
                \biggl(
                    %L = N
                    % N - BEGIN
                    \big\lceil
                        (C_3C_{\bar{f}})^{n^{in}/4k}
                            \,
                        (n^{in})^{n^{in}/8k}
                            \,
                        [\omega_{\phi}^{\dagger}(\varepsilon_{A})]^{-2k/n^{in}}
                    \big\rceil
                    % N - END
                +
                    2
                \biggr)
                \,
                \log_2\biggl(
                    %L = N 
                    % N - BEGIN
                    \big\lceil
                        (C_3C_{\bar{f}})^{n^{in}/4k}
                            \,
                        (n^{in})^{n^{in}/8k}
                            \,
                        [\omega_{\phi}^{\dagger}(\varepsilon_{A})]^{-2k/n^{in}}
                    \big\rceil
                    % N - END
                \biggr)
                +
                2
                n^{in}
        \end{split}
    \end{equation}
    where $C_{1} \eqdef  17 k^{n^{in}+1} 3^{n^{in}} n^{in}$, $C_2 = 18 k^2$, $C_3 = 85 (k+1)^{n^{in}} 8^{k}$ and $C_{\bar{f}} \eqdef \max_{i=1,\ldots, n^{out}} \|\bar{f}_i\|_{C^{k}([0,1]^{n^{in}})}$.
\end{itemize}

Since the ReLU has the 2-Identity Property\footnote{See Definition \ref{def:c_identity}.}, we can apply Proposition \ref{prop:number_parameters_c_identity} to conclude that there exists an ``efficient parallelization" $\tilde{f}\,:\,\mathbb{R}^{n^{in}} \rightarrow \mathbb{R}^{n^{out}}$ of $x \rightarrow (\hat{f}_{\theta}^{(i)}(x),\ldots, \hat{f}_{\theta}^{(n^{out})}(x))$. This is equivalent to say that for every $x \in \mathbb{R}^{n^{in}}$ the following identity holds true $\tilde{f}(x) \eqdef (\hat{f}_{\theta}^{(1)}(x), \ldots, \hat{f}_{\theta}^{(n^{out})}(x))$. The width and the depth of $\tilde{f}$, denoted by $\emph{Width}(\tilde{f})$ and $\emph{Depth}(\tilde{f})$ are given by:
\begin{itemize}
    \item[(\,2\,)] \textbf{\emph{Width}}\,:\,
    \begin{equation}\label{eq:width_approximation_theorem_proof_1}
    \hspace{-.5cm}
        \begin{split}
                \text{\emph{Width}}(\tilde{f}) 
                &
                = n^{in}(n^{out}-1) + \text{\emph{Width}}(\hat{f}^{(1)}_{\theta})
        \end{split}
    % $} % END RESIZE BOX
    \end{equation}
    where $\text{\emph{Width}}(\hat{f}^{(1)}_{\theta})$ denotes the width of $\hat{f}_{\theta}^{(1)}$, and where we have used the fact that $\text{\emph{Width}}(\hat{f}^{(1)}_{\theta}) = \text{\emph{Width}}(\hat{f}^{(i)}_{\theta})$ for every $i = 1, \ldots, n^{in}$.
 \item[(\,3\,)] \textbf{\emph{Depth}}\,:\,
    \begin{equation}\label{eq:depth_approximation_theorem_proof_1}
    \hspace{-.5cm}
        \begin{split}
                \text{\emph{Depth}}(\tilde{f})
                &= 
                n^{out}(1+\text{\emph{Depth}}(\hat{f}_{\theta}^{(1)}))
                ,
        \end{split}
    \end{equation}
     where $\text{\emph{Depth}}(\hat{f}^{(1)}_{\theta})$ denotes the width of $\hat{f}_{\theta}^{(1)}$, and where we have used the fact that $\text{\emph{Depth}}(\hat{f}^{(1)}_{\theta}) = \text{\emph{Depth}}(\hat{f}^{(i)}_{\theta})$ for every $i = 1, \ldots, n^{out}$.
\end{itemize}

\noindent Finally, define $\hat{f}_{\theta}  \eqdef  \tilde{f} \circ W$ and note that the space $\mathcal{NN}^{\operatorname{ReLU}}_{[d]}$ introduced in Subsection \ref{subsec:feed_forward} is invariant to pre-composition by affine maps. Therefore, $\hat{f}_{\theta}$ has the same depth and width of $\tilde{f}$. Whence, we have:
\[
    \begin{aligned}
        \max_{u\in \psi\circ P_{E:n^{in}}(K)}\,\|\hat{f}_{\theta}(u) - \bar{F}(u)\|_2 &= \max_{u\in \psi\circ P_{E:n^{in}}(K)}\,\|      \tilde{f}\circ W(u) - \bar{F}(u) \|_2\\
            &= \max_{z\in W[\psi\circ P_{E:n^{in}}(K)]}\,\|\tilde{f}(z) -   \bar{F}\circ W^{-1}(z)\|_2\\
            &\leq \max_{z\in [0,1]^{n^{in}}}\,\|\tilde{f}(z)- \bar{F}\circ W^{-1}(z)\|_2\\
            &\leq \sqrt{n^{out}} \max_{i = 1, \ldots, n^{out}}\,\max_{z\in [0,1]^{n^{in}}}\,\| \hat{f}_{\theta}^{(i)} - \bar{f}_i(z)\|_2\\
            &\leq \sqrt{n^{out}} \frac{\delta}{\sqrt{n^{out}}} = \delta.
    \end{aligned}
\]
which is nothing but \eqref{eq:Proof_StaticResult_ReLUFFNN_Estimate}. The Theorem is whence proved for $f \in C_{\operatorname{tr}}^{k,\lambda}(K, B)$.\\
%%%%%%%%%%%%%%%%%%%%%%%%%%
\paragraph{The $C_{\alpha,\text{tr}}^{\lambda}(K, B)$ Case:} 
%%%%%%%%%%%%%%%%%%%%%%%%%%
We report to the reader the main changes of the proof.
\begin{itemize}
    \item[(\,i\,)] The quantity $n^{in}$ in Equation \eqref{eq:dimensionality_approximation_theorem} is instead given by:
    \begin{equation*}
        n^{in} 
            \eqdef  
        \inf\left\{n \in \mathbb{N}_{+}\,:\,\max_{x \in K} d_{E}(A_{E:n}(x), x) \le \left(\frac{1}{\lambda}\omega^{\dagger}_{A,B}\,\left(\frac{\varepsilon_{D}}{2}\right)\right)^{1/\alpha}\right\}.
    \end{equation*}
    In this way, the estimate in Equation \eqref{eq:Proof_Static_E_DimRed_Estimate_1} continues to hold with $F \in C_{\alpha,\text{tr}}^{\lambda}(K, B)$.
    \item[(\,ii\,)] The inequality in Equation \eqref{eq:Proof_StaticResult_ReLUFFNN_Estimate} is now guaranteed by Theorem~\ref{thrm:approximation_in_holder}, instead of by Theorem \ref{thm:approximation_in_C_s}. 
    Note, that the pre/post-composition of an $\alpha$-H\"older function with a Lipschitz function is again an $\alpha$-H\"older function.
    \item[(\,iii\,)] The function $\bar{F}\circ W^{-1}$ in Equation \eqref{eq:fbarcomposewminusone} is $C_{\alpha,\text{tr}}^{\lambda}(K, B)$, and so, we may apply Theorem~\ref{thrm:approximation_in_holder} to deduce that there are $n^{in}$ ReLU FFNN satisfying to the estimates in Equation \eqref{estimate:univariateFFNNestimate_on_fdim_cube}.
    \item[(\,iv\,)] The width and the depth of each $\hat{f}_{\theta}^{i}$ are thus provided by Theorem~\ref{thrm:approximation_in_holder}.
    Setting $N=L$ in that result yields
    \begin{itemize}
    \item[(i)] \textbf{\emph{Width}\,:\,}
        \begin{equation}\label{eq:width_approximation_theorem_holder_proof}
        \hspace{-0.5cm}
        \resizebox{0.90\textwidth}{!}{$ 
            \begin{aligned}
            C_1
            \,
            \max\biggl\{
                n^{in}
                \,
                \biggl\lfloor
                \biggl(
                % N - BEGIN
                    [\omega_{\phi}^{\dagger}(\varepsilon_{A})]^{-n^{in}/\alpha}
                    \,
                    V\big(
                        (131 \,\lambda )^{n^{in}/\alpha}
                        \,
                        (n^{in} n^{out})^{n^{in}/\alpha}
                    \big)
                % N - END
                \biggr)^{1/n^{in}}
                \biggr\rfloor
            ,
                    \biggl\lceil
                        % N - BEGIN
                        [\omega_{\phi}^{\dagger}(\varepsilon_{A})]^{-n^{in}/\alpha}
                        \,
                        V\big(
                            (131 \,\lambda )^{n^{in}/\alpha}
                            \,
                            (n^{in} n^{out})^{n^{in}/\alpha}
                        \big)
                        % N - END
                    \biggr\rceil
                +
                    2
            \biggr\}
            \end{aligned}
        $}
        \end{equation}
        with $C_1 = 3^{n^{in}}+3$.
    \item[(ii)] \textbf{\emph{Depth}\,:\,}
        \begin{equation}\label{eq:depth_approximation_theorem_holder_proof}
        \hspace{-.5cm}
            \begin{split}
            11\,
            % L = N
            \biggl\lceil
                % N - BEGIN
                [\omega_{\phi}^{\dagger}(\varepsilon_{A})]^{-n^{in}/\alpha}
                \,
                V\big(
                    (131 \,\lambda )^{n^{in}/\alpha}
                    \,
                    (n^{in} n^{out})^{n^{in}/\alpha}
                \big)
                % N - END
            \biggr\rceil
            \,
            +
            C_2
            \end{split}
        % $}
        \end{equation}
        with $C_2 = 18 + 2\,n^{in}$.
    \end{itemize}
\item[(\,vi\,)] The considerations on the existence of an ``efficient parallelization" continue to hold with the width and depth appropriately defined by using (\,v\,).  
\end{itemize}

\end{proof}

\subsection{Proof of \texorpdfstring{Corollary~\ref{cor:dynamiccase__regular_subcase}}{Efficient Approximation Case}}
\label{s:Efficient_and_CleanRates}

\noindent Before proving Corollary~\ref{cor:dynamiccase__regular_subcase}, we recall the definition of a linear $i$-width of a subset $A$ of a infinite-dimensional normed linear space $(X,\|\cdot\|_X)$, see e.g.~\cite[Definition I.1.2]{Pin1985Book}: for every $i\in \mathbb{N}_+$ set
    \[
            \delta_i(A,X)
        \eqdef 
            \inf_{T}\,\sup_{a\in A}\,\|a-Ta\|_X
    \]
where the infimum is taken over all continuous linear operators $T:X\rightarrow X$ whose rank is at most $i$.  It will also be convenient, for the proof of Corollary~\ref{cor:dynamiccase__regular_subcase}, to recall the definition of the $i$-width in the sense of Kolmogorov, see e.g.~\cite[Defintiion I.1.1]{Pin1985Book}.  For any $i\in\mathbb{N}$ and any subset $A$ of an infinite-dimensional Banach space $X$, the \textit{Kolmogorov i width of $A$ in $X$} is defined to be
\[
        d_i(A,X)
    \eqdef 
        \inf_{X_i}
        \,
        \sup_{a\in A}\,
        \inf_{u\in X_i}
        \,
        \|a-u\|_X
\]
where the outer infimum is taken over all $i$-dimensional linear subspaces $X_i$ of $X$.  Both of these notions of ``width'', i.e.\ linear complexity, of a subset coincide when the space $X$ is a Hilbert space, see e.g.~\cite[Proposition II.5.2]{Pin1985Book}; however, we introduce both notions since some results are formulate for general Banach spaces using one width rather than the other, in most parts of the literature.

\begin{proof}[{Proof of Corollary~\ref{cor:dynamiccase__regular_subcase}}]
% By \cite[Corollary IV.2.6]{Pin1985Book}, for each $i\in \mathbb{N}$ 
For each $n\in \mathbb{Z}$, define the set 
% \textcolor{red}{$Z_{t_n}$?} 
$Z_n\subseteq E_{t_n}$ by 
% \textcolor{red}{maybe stupid question: Are there always elements with this characteristic?}
% {\textcolor{violet}{ for sure, we construct them by 
% $z= \sum_{i} z_i\, e_i$ where $|z_i|\lesssim e^{-\rho i}$ and note that $\sum z_i^2 <\infty$.}}
\[
Z_{n}\eqdef 
\big\{z\in E_{t_n}:\, 
% (\sum_{i=0}^{\infty}\,C^2e^{\rho\,i}\,\langle z,e_{n}\rangle_{E_{t_n}}^2 )^{1/2} 
{\color{black}(\forall i \in \mathbb{N})\, \langle z,e_{i}\rangle_{E_{t_n}}^2\le C^{-2 \rho i}}
% \le 1
\big\}
.
\]
{\color{black}Now, for each ``re-scaling parameter'' $0<r<1$ let $Z_n^r\eqdef r\cdot Z_n\subset Z_n$ where for any 
% \textcolor{red}{Should this $K$ depend on $n$?} \textcolor{violet}{nein, just an arbitrary set} 
$K\subset E_{t_n}$ we define $r\cdot K\eqdef \{rx:\, x\in  K\}$.
Note that, if $f:Z_{n}\mapsto \mathbb{R}$ is $\lambda$-Lipschitz then $f_r\eqdef f(r \cdot ): Z_{n} \ni x \to f(rx)\in \mathbb{R}$ is at-most $r\lambda$-Lipschitz.
% \luca{This is not correct: first of all, we need to ensure that $rx\in Z_n$ where $f$ is defined.  Since $x=ry$ for some $y\in Z_n$, this amounts to impose $0<r<1$, (which implies $Z_n^r\subset Z_n$). Then the constant is at most $r^2\lambda \le r\lambda$.  }
and satisfies
\begin{equation}
\label{eq:renormalization}
    f(x) = f_r\circ S_{1/r}(x)
\end{equation}
% \luca{also this is not correct: one must assume that $\frac{1}{r}\cdot x\in Z^r_n$}
for all $x\in r\cdot Z_n$; where $S_{1/r}:E_{t_n}\ni x\mapsto \frac1{r}\cdot x$.  Note that 
\begin{equation}
\label{eq:simple_key_point}
    S_{1/r}(Z_n^r) =Z_n
\end{equation}
for each $n\in \mathbb{Z}$.  We thus, approximate $f_r$ on each $Z_n$.}

Consider the Kolmogorov $i$-width $\delta_i(Z_{n,i},E_{t_n})$ 
is optimized by the linear subspace spanned by $\{e_{n,j}\}_{j=0}^{i-1}$ and satisfies
\begin{equation}
\label{eq:linear_width_LorenzIthink}
        \delta_i(Z_{n},E_{t_n})=d_i(Z_{n},E_{t_n})
    =
        \sup_{z\in Z_{n,i}}
        \,
        \inf_{u\in \operatorname{span}\{e_{n,j}\}_{j=0}^{i-1}}
        \,
            \big\|
                    z
                -
                    u
            \big\|_{E_{t_n}}
    % =
    %     C\,
    %     e^{-i\rho/2}
    {\color{black}
    \le 
        \sqrt{
            C\sum_{j=i}^{\infty}\,
            e^{-2 \rho j}
        }
    =
        \tilde{C}_1\,e^{-\rho i}
    }
% ,
\end{equation}
where the outer infimum is taken over all at-most $i$-dimensional subspaces $Z_{n,i}$ of $E_{t_n}$
{\color{black}and where $\tilde{C}_1\eqdef \sqrt{Ce^{\rho}/(1 - e^{-\rho})}>0$.}
Condition~\eqref{eq:f_compatability_condition} implies that, for each $n\in \mathbb{Z}$ and each $i\in \mathbb{N}$, we have the inclusion $\mathcal{X}_{t_n}\subseteq Z_{n,i}$; therefore, \cite[Theorem I.1.1 (v)]{Pin1985Book} implies
\[
        d_i(\mathcal{X}_{t_n},E_{t_n})
    \le 
        d_i(Z_{n},E_{t_n})
    .
\]
Consequentially,~\eqref{eq:linear_width_LorenzIthink} implies that, for each $n\in \mathbb{Z}$ and each $i\in \mathbb{N}$, the following holds
\begin{equation}
\label{eq:Kolmogorov_width_functorialitylinear_width_LorenzIthink}
        \delta_i(\mathcal{X}_{t_n},E_{t_n})
    =
        d_i(\mathcal{X}_{t_n},E_{t_n})
    \le 
        {\color{black}
            \tilde{C}_1
            e^{-i\rho}
        }
\end{equation}
Moreover, since $\{e_{n,j}\}_{j=0}^{i-1}$ is an orthonormal set then the orthogonal projection operator $A_{E_{t_n},i}:E_{t_n}\mapsto \operatorname{span}\{e_{n,j}\}_{j=0}^{i-1}$, given by $x\mapsto \sum_{j=0}^{i-1}\, \langle x,e_{n,j}\rangle_{E_{t_n}}\,e_{n,j}$ is optimal; whence,
\begin{equation}
\label{eq:eigenfunction_optimization}
        \delta_i(\mathcal{X}_{t_n},E_{t_n})
    =
        \sup_{z\in \mathcal{X}_{t_n}}
        \,
            \Big\|
                    z
                -
                \sum_{j=0}^{i-1}\,
                    \langle z,e_{n,j}\rangle_{E_{t_n}}\,e_{n,j}
            \Big\|_{E_{t_n}}
    =
        \sup_{z\in \mathcal{X}_{t_n}}
        \,
            \|
                z
                -
                A_{E_{t_n},i}(z)
            \|_{E_{t_n}}
.
\end{equation}
Since orthonormal basises of Hilbert spaces are trivially Schauder basises, then, for each $n\in \mathbb{Z}$ and every $i\in \mathbb{N}$, $A_{E_{t_n},i}$ is as in Table~\ref{tab:Model_Complexity} and together~\eqref{eq:Kolmogorov_width_functorialitylinear_width_LorenzIthink} and~\eqref{eq:eigenfunction_optimization} imply that
\begin{equation}
\label{eq:putting_it_all_together}
        \sup_{z\in \mathcal{X}_{t_n}}
        \,
            \|
                z
                -
                A_{E_{t_n},i}(z)
            \|_{E_{t_n}}
    \le 
        % C
        % \,
        % e^{-\rho\,i/2}
        {\color{black}\tilde{C}_1\,e^{-\rho i}}
.
\end{equation}
Note that in a separable Hilbert space, we have the $1$-BAP property (also called the \textit{metric approximation property}), and thus $\omega_{A,B}(t)=t$. In particular, 
$\frac{1}{\lambda{\color{black}r}}\omega^{\dagger}_{A,B}\left(\frac{\varepsilon_{D}}{2}\right)=\frac{\varepsilon_D}{2\lambda{\color{black}r}}$.
As recorded in Table~\ref{tab:Model_Complexity}, an encoding dimension $i$ of at-least $
        \inf\Big\{i \in \mathbb{N}_{+}\,:\,\max_{z \in \mathcal{X}_i} 
        \|A_{E_{t_n},i}(z)- z\|_{E_{t_n}} \le \frac{1}{\lambda{\color{black}r}}\omega^{\dagger}_{A,B}\left(\frac{\varepsilon_{D}}{2}\right)\Big\}
    $ is necessary to guarantee that $\max_{z \in \mathcal{X}_i} 
        \|A_{E_{t_n},i}(z)- z\|_{E_{t_n}} \le\frac{\varepsilon_D}{2\lambda{\color{black}r}}$.  
        Therefore, setting
\begin{equation}
\label{eq:constant_dimension_setup}
n_{\varepsilon_D}^{\text{in}} \le  i^{in} 
=
% \lceil 2\ln(\varepsilon_D^{-1}\,2\lambda C)/\rho \rceil
{\color{black}
    \big\lceil
        \ln(
            c
            % \lambda^{1/\rho}\,
            \,
            \underbrace{
                (r\varepsilon_D^{-1})^{1/\rho} 
            }
        )
    \big\rceil
}
\end{equation}
{\color{black}where $c\eqdef (2\tilde{C}_1 \lambda)^{1/\rho}$}, implies that
\begin{equation}
\label{eq:n_in_dimension}
        \sup_{z\in Z_n}
        \,
            \|
                z
                -
                A_{E_{t_n},i^{in}}(z)
            \|_{E_{t_n}}
    \le 
        % C
        % \,
        % e^{-\rho\,i^{in}/2}
        {\color{black}
        \tilde{C}_1\,e^{-\rho i^{in}}
        }
    \le
        \frac{\varepsilon_D}{2 \lambda }
.
\end{equation}
{\color{black}By~\eqref{eq:constant_dimension_setup}, setting $r=\varepsilon_D$ implies that~\eqref{eq:n_in_dimension} holds while
\begin{equation}
\label{eq:constant_dimension_setup}
n_{\varepsilon_D}^{\text{in}} \le  i^{in} 
=
{\color{black}
    \big\lceil
        \ln(
            c
            % \lambda^{1/\rho}\,
        )
    \big\rceil
\in \mathcal{O}(1)
.
}
\end{equation}
}
{\color{black}Since the target space is one dimensional then $n_{\varepsilon_D}^{out}=1$ for all $\varepsilon_D>0$.  Thus, when approximating $f$ on $r\cdot K$, for $r=\varepsilon_D$, both $n_{\varepsilon_D}^{\text{in}}$ and $n_{\varepsilon_D}^{\text{out}}$ are constants.}
% For each $n\in \mathbb{Z}$ and $i\in\mathbb{N}_+$, the Assumption~\eqref{eq:f_compatability_condition}, and the fact that $B_{t_n}$ is a separable infinite-dimensional Hilbert spaces implies the following, following a similar computation we deduce that a decoding dimension $n_{\varepsilon_D}^{out}$ of at-least
% \[
%         n_{\varepsilon_D}^{\text{out}}
%     \le
%     %     \ln\big(
%     %         \varepsilon_D^{- \rho/2}
%     %     \big)
%     % +
%     %     \ln(2^{\rho/2}C^{\rho/2})
%         {
%             \big\lceil
%             \ln(
%                 \varepsilon_D^{-1/\rho} \, 2\tilde{C}_1)^{1/\rho}
%             )
%             \big\rceil
%         \le
%             \big\lceil
%                 \ln(
%                     \varepsilon_D^{-1/\rho} \, A)^{1/\rho}
%                 )
%                 \big\rceil
%         }
% \]
% is no smaller than the optimal decoding dimension $\inf\left\{n \in \mathbb{N}_{+}\,:\,\underset{y \in F(K)}{\max}\,d_{B}( A_{B:n}(y), y)\le\frac{\varepsilon_{D}}{2}\right\}$.  

{\color{black}Fix $\varepsilon>0$, set $\varepsilon_A=\varepsilon_D=\varepsilon$.}
Since $f$ is $(r,\infty,\lambda)$-smooth then, it is $(r,\lceil n_{\varepsilon_D}^{\text{in}}/8\rceil,\lambda)$.  Therefore, for all $I\in \mathbb{N}_+$, Theorem~\ref{thm:theorem_universality_causal} implies that there is a CNO such that
\begin{equation}
\label{eq:funny_trick}
    \begin{split}
    \max_{i \in 
        [[I]]
    }\,
    \sup_{x \in K_n}\,\,
        d_{B_{t_i}}\big(
                \hat{f}_{t_i}(x_{(t_{i-M}, t_i]})
                    , 
                f(rx)_{t_i}
            \big) <  \varepsilon_A + \varepsilon_D,
    \end{split}
\end{equation}
where $\hat{f}_{t_i} \in \mathcal{NF}_{[n_{\varepsilon_D}]}^{(P)ReLU}$, $
\hat{f}_{t_i} = I_{B_{t_i}:n_{\varepsilon_D^{out}}} \circ \varphi_{n_{\varepsilon_D^{out}}}\circ \hat{f}_{\theta_{t_i}} \circ \psi_{n_{\varepsilon_D^{out}}}\circ P_{E_{(t_{i-M},t_i]}:n_{\varepsilon_D}^{in}}
$ where each $\hat{f}_{\theta_{t_i}}$.  Now, we use~\eqref{eq:simple_key_point} and the fact that $S_{1/r}$ is nothing but rescaling by a constant factor of $1/r$, which commutes with each $\psi_{n_{\varepsilon_D^{out}}}\circ P_{E_{(t_{i-M},t_i]}:n_{\varepsilon_D}^{in}}$ by linearity, and which can be absorbed into the first affine layer in each $\hat{f}_{\theta_{t_i}}$; denote these MLPs with re-scaled first affine layer by $\hat{f}_{\tilde{\theta}_{t_i}}$.  Note that the depth and width of each $\hat{f}_{\theta_{t_i}}$ and $\hat{f}_{\tilde{\theta}_{t_i}}$ are identical.
For each $i\in I$, re-define
\[
\tilde{f}_{t_i} = I_{B_{t_i}:n_{\varepsilon_D^{out}}} \circ \varphi_{n_{\varepsilon_D^{out}}}\circ \hat{f}_{\tilde{\theta}_{t_i}} \circ \psi_{n_{\varepsilon_D^{out}}}\circ P_{E_{(t_{i-M},t_i]}:n_{\varepsilon_D}^{in}}.
\]
Consequently,~\eqref{eq:funny_trick} implies that
\begin{equation}
\label{eq:funny_trick__worked}
    \begin{split}
    \max_{i \in 
        [[I]]
    }\,
    \sup_{x \in K_n^r}\,\,
        d_{B_{t_i}}\big(
                \tilde{f}_{t_i}(x_{(t_{i-M}, t_i]})
                    , 
                f(x)_{t_i}
            \big) 
=
    \max_{i \in 
        [[I]]
    }\,
    \sup_{x \in K_n}\,\,
        d_{B_{t_i}}\big(
                \hat{f}_{t_i}(x_{(t_{i-M}, t_i]})
                    , 
                f(rx)_{t_i}
            \big) 
    <  
        \varepsilon_A + \varepsilon_D,
    \end{split}
\end{equation}
Finally, Table~\ref{tab:Model_Complexity} implies that the neural filters defining the CNO have width at most
\allowdisplaybreaks
{\color{black}
\begin{align}
\label{eq:exact_width}
 C_1\left(\left\lceil C_3C_{\bar{f}}\, \sqrt{\lceil \ln(c) \rceil}\, \varepsilon^{-1/2} \right\rceil+2\right)\cdot \log_2\!\left(8\,\left\lceil C_3C_{\bar{f}}\, \sqrt{\lceil \ln(c) \rceil}\, \varepsilon^{-1/2} \right\rceil\right).
 \in
 \mathcal{O}(\varepsilon^{-1/2})
\end{align}
and its depth is 
\begin{align}
\label{eq:exact_depth}
 1 +
C_2\,\Biggl(
\Bigl\lceil C_3C_{\bar{f}}\,\sqrt{\lceil \ln(c) \rceil}\,\varepsilon^{-1/2} \Bigr\rceil + 2
\Biggr)
\,\log_2\!\Bigl(\Bigl\lceil C_3C_{\bar{f}}\,\sqrt{\lceil \ln(c) \rceil}\,\varepsilon^{-1/2} \Bigr\rceil\Bigr)
+ 2\,\lceil \ln(c) \rceil
 \in
 \mathcal{O}(\varepsilon^{-1/2}\log(1/\varepsilon))
\end{align}
Consequently, the number of non-zero (trainable) parameters is almost the width squared times the depth; whence $\mathcal{O}(\varepsilon^{-3/2}\log(1/\varepsilon)^3)$.}

{\color{black}Fix a time-horizon $I\in \mathbb{N}_+$.
Lastly, since the depth, with, and especially, a number of parameters of the hypernetwork only depends on $P([d])$ and on the time-horizon $I$; then, they are as in Table~\ref{tab:Model_Complexity_DynamicCase}.  Specifically, the number of trainable parameters defining the hypernetwork are at-most
\begin{align}
\mathcal{O}\Biggl(
I^3\,\Bigl(\varepsilon^{-3/2}\,\log(1/\varepsilon)^3+Q\Bigr)^2\
\Biggl(1+\Bigl(\varepsilon^{-3/2}\,\log(1/\varepsilon)^3+Q\Bigr)
    \sqrt{I\,\log(I)}\,
\Bigl(1+\frac{\log(2)}{\log(I)}\,\Bigl[C_d+\frac{2\log(I)+\tfrac{1}{2}\log(2)-\log(\delta)}{\log(2)}\Bigr]_+\Bigr)
\Biggr)
\Biggr)
\end{align}
which implies to $\mathcal{O}\big(\varepsilon^{-9/2} I^{1/2} \log(I)^{3/2}\log(1/\varepsilon)^9\big) \in \tilde{\mathcal{O}}(\sqrt{\varepsilon^{-9} I})$.
}

\end{proof}

\subsection{The Dynamic Weaving Lemma}\label{subsec:proof_waving_lemma}
We now present our main technical tool for ``weaving together'' several neural filters approximating a causal map on distinct time windows.  The key technical insight here is that each neural filter is approximated while the hypernetwork ``weaving together'' these neural filter memorizes, and memorization requires exponentially fewer parameters than approximation.  {\color{black} The reason for this is that memorizing \( N \) points requires between \( N \) and \( N^2 \) trainable (non-zero) parameters, as demonstrated in sources like \cite{VershinynMemorization} and \cite{Ruiyang}. Notably, only \( \mathcal{O}(1) \) neurons are necessary to memorize a function's value at a single point. In contrast, approximating a function's value on each sub-cube of \([0, 1]^d\) with side length \( \delta \) requires \( \mathcal{O}(1) \) neurons for each sub-cube, with a total of \( \Theta(\delta^{-d}) \) such sub-cubes. As a result, any uniform approximator needs an exponential number of neurons to uniformly approximate a function over any hypercube, whereas a memorizer of \( N \) points does not have that same requirement.}

\begin{lemma}[Dynamic Weaving Lemma]\label{lemma:deterministic_weaving}
Let $[d]=(d_0,\ldots, d_{J})$, $J \in \mathbb{N}_{+}$, be a multi-index such that $P([d]) = \sum_{j=0}^{J-1} d_{j} (d_{j+1} + 2) + d_{J} \geq 1$, and let $(\hat{f}_{\theta_t})_{t\in \N}$ a sequence in $\mathcal{N}\mathcal{N}^{\operatorname{(P)ReLU}}_{[d]}$. Then, for every ``latent code dimension'' $Q\in \mathbb{N}_+$ with $Q+P([d])\geq 12$ and every ``coding complexity parameter'' $\delta>0$, there is a ReLU FFNN $\hat{h}:\mathbb{R}^{P([d])+Q} \rightarrow \mathbb{R}^{P([d])+Q} $, an ``initial latent code'' $z_0\in \mathbb{R}^{P([d])+Q}$, and a linear map $L:\mathbb{R}^{P([d])+Q}\rightarrow \mathbb{R}^{P([d])}$ satisfying
\begin{equation*}
    \begin{split}
     \hat{f}_{L(z_t)} &= \hat{f}_{\theta_t},\\
     \quad z_{t+1}    &= \hat{h}(z_t),
    \end{split}
\end{equation*}
for every ``time'' $t=0,\dots, \big\lfloor \delta^{-Q} \big\rfloor =: T_{\delta,Q}-1$. 
Moreover, the ``model complexity'' of $\hat{h}$ is specified by 
\begin{enumerate}
	\item[(i)]  \textbf{Width:} $\mathcal{NN}$ has width at-most $(P([d])+Q)T + 12$;
	\item[(ii)] \textbf{Depth:} $\mathcal{NN}$ has depth at-most of the order of
    \begin{equation*}
        \mathcal{O}\Biggr(
        T\biggl(
        1+\sqrt{T\log(T)}\,\Big(1+\frac{\log(2)}{\log(T)}\,
        \biggl[
            C + \frac{\Big(\log\big(T^2\,2^{1/2}\big)-\log(\delta)\Big)}{\log(2)}
        \biggr]_+\Big)
        \biggr)
        \Biggr);
    \end{equation*}
	\item[(iii)] \textbf{Number of non-zero parameters:} The number of non-zero parameters in $\mathcal{NN}$ is at-most
    \[
    \begin{aligned}
        \mathcal{O}\Bigl(
        T^3(P([d])+Q)^2\,\left(1+(P([d])+Q)
        \sqrt{T\log(T)}\,\left(1+\frac{\log(2)}{\log(T)}\,
            \left[
            C_d+
                \frac{\Big(\log\big(T^2\,2^{1/2}\big)-\log(\delta)\Big)}{\log(2)}
            \right]_+\right)\right)
    \Bigr),
    \end{aligned}
    \]
    where the constant $C_d>0$ is defined by 
\[
    C_d  \eqdef   \frac{2\log(5 \sqrt{2\pi})+ 3\log(P([d])+Q)-\log(P([d])+Q+1)}{2\log(2)}
    .
\]
	\end{enumerate}
\noindent   
In the previous expressions $(i)$, $(ii)$ and $(iii)$ we set, for simplicity of notation, $T \eqdef T_{\delta, Q}-1$.
\end{lemma}

{The proof of Lemma~\ref{lemma:deterministic_weaving} proceeds in two stages. First, we construct a $\delta$-packing $\{\tilde{z}_i\}_{i=0}^T$ of high-dimensional balls with minimal radius $\delta$, for a suitably chosen $T\in \mathbb{N}_+$. We then augment each parameter vector $\theta_0,\dots,\theta_T$ with the corresponding separated vector, ensuring that even if some parameter vectors—say, $\theta_1$ and $\theta_2$—are identical, their augmented versions, for example, $z_1\eqdef (\theta_1,\tilde{z}_1)$ and $z_2\eqdef (\theta_2,\tilde{z}_2)$, remain distinct and are separated by a fixed positive distance.
\hfill\\
With this list of $T$ distinct augmented parameter vectors, we can construct a ReLU MLP memorizer using Lemma~\ref{lem:embedding}, which solves the recursion problem by mapping any $z_t$ (input) to $z_{t+1}$ (output). 
\hfill\\
Two technical points to highlight are: 1) since we pack a sphere of radius $R > 0$, the maximal distance between any two vectors $\tilde{z}_i$ and $\tilde{z}_j$ is uniformly bounded above by $2R$, and 2) by considering a high-dimensional sphere instead of a one-dimensional line segment, we can separate more parameter vectors while keeping the distance between the augmented parts, i.e., the $\tilde{z}_i$ and $\tilde{z}_j$, low.}

\begin{proof}
\noindent   Set $P \eqdef P([d])$, and let $Q \in \mathbb{N}_{+}$ such that $P + Q \geq 12$. Moreover, let $R>0$ such that $0 < \delta < R$; the precise value of $R$ will be derived below. Now, let $(\theta_t)_{t \in \mathbb{N}_{+}}$ be a sequence in $\mathbb{R}^{P}$ ($P$ defined at the beginning of the proof), and let, for every $T \in \mathbb{N}_{+}$, $M_T$ be the constant defined as:
\begin{equation}\label{eq:constant_M_T}
    M_T  \eqdef  \max\{1, \max_{s, t = 0, \ldots, T}\|\theta_t-\theta_s\|_2\}
\end{equation}
\noindent Now, let $\overline{\text{Ball}_{(\mathbb{R}^{Q}, \|\cdot\|_2)}(0, R)} \subset \mathbb{R}^{Q}$ be the closed Euclidean ball centered in zero and with radius $R$. Because $\delta < R$ and because of the geometry of the Euclidean ball, there exists an integer $T_{R, \delta, Q}>1$ such that $\{\tilde{z}_0, \ldots, \tilde{z}_{T_{R, \delta, Q}-1}\}$ is an $\delta$-packing of $\overline{\text{Ball}_{(\mathbb{R}^{Q}, \|\cdot\|_2)}(0, R)}$ meaning that $\min_{i, j = 0, \ldots, T_{R, \delta, Q}-1; i \neq j} \|\tilde{z}_i - \tilde{z}_j\|_2 > \delta$.
It holds that:
\begin{equation*}\label{eq:bound_packing_number}
  \left(\frac{R}{\delta}\right)^{Q} \leq T_{R,\delta,Q}.
\end{equation*}

\noindent   At this point, we define the sequence  $(z_t)_{t \in \mathbb{N}} \in \mathbb{R}^{P+Q}$ in the following way:
\begin{equation}
    z_t \eqdef 
    \begin{cases}
        \left(\frac{1}{M_{T}}\theta_t, \tilde{z}_t\right)\quad\quad\,\,\,\,:\,t < T_{R,\delta,Q}\\
        \left(\theta_{T_{R, \delta, Q}}, \mathbf{0}_{Q}\right)\,\,\,\,\,\quad\,:\,t \geq T_{R,\delta,Q},
    \end{cases}
\end{equation}
where $\mathbf{0}_{Q} \eqdef (0, \ldots, 0) \in \mathbb{R}^{Q}$.\\
\noindent At this point, we use the (multi-dimensional) Pythagorean theorem and by construction of the  sequence  $(z_t)_{t \in \mathbb{N}} \in \mathbb{R}^{P+Q}$ each $z_0, \ldots, z_{T_{R, \delta, Q}-1}$ is distinct from each other and the aspect ratio, see Equation \eqref{eq:aspect_ratio_finite}, of the finite metric space $(\mathcal{Z}_{T_{R, \delta, Q}}, \|\cdot\|_2)$, where $\mathcal{Z}_{T_{R, \delta, Q}} \eqdef \{z_0, \ldots, z_{T_{R, \delta, Q}-1}\}$, is bounded above by: 
\begin{equation}
\label{eq:dynamicweaving_ControlAspectRatio}
\begin{split}
    \operatorname{aspect}(\mathcal{Z}_{T_{R, \delta, Q}},\|\cdot\|_2) = &
    \frac{\max_{t,s=0,\dots,T_{R,\delta,Q}-1}\, \|z_t-z_s\|_2}{\min_{i,j=0,\dots,T_{R,\delta,Q}-1;\,i\neq j}\,\|z_i-z_j\|_2}\\
    \leq &\frac{\Big(\max_{t,s=0,\dots,T_{R,\delta,Q}-1}\,
    \frac1{M_T}\,
    \|\theta_t-\theta_s\|_2^2+\max_{k,l=0,\dots,T_{R,\delta,Q}-1}\,\|\tilde{z}_k-\tilde{z}_l\|_2^2\Big)^{1/2} }{ \min_{i,j=0,\dots,T_{R,\delta,Q}-1;\,i\neq j}\,
            \|\tilde{z}_i-\tilde{z}_j\|_2}\\
        \leq & \frac{\Big(1 + 4\,R^2\Big)^{1/2}}{\delta}.
\end{split}
\end{equation}
\noindent Therefore, we can apply Lemma \ref{lem:embedding} to say that there exists a deep ReLU networks $\tilde{h}\,:\,\mathbb{R}^{P+Q}\rightarrow\mathbb{R}^{P+Q}$ satisfying
$$z_{t+1} = \tilde{h}(z_t),$$
for every $t = 0, \ldots, T_{R, \delta, Q}-1$. Furthermore, the following quantitative ``model complexity estimates" hold
\begin{itemize}
    \item[(\,i\,)] \textbf{\emph{Width}\,:\,} $\tilde{h}$ has width $(P+Q)T_{R, \delta, Q} + 12$,
    \item[(\,ii\,)] \textbf{\emph{Depth}\,:\,} $\tilde{h}$ has depth of the order of
    \begin{equation*}
        \mathcal{O}\left(T_{R, \delta, Q}\left( 1+\sqrt{T_{R, \delta, Q}\log(T_{R, \delta, Q})} \left(1 + \frac{\log(2)}{\log(T_{R, \delta, Q})} \Bigl[ C_d + \frac{\log\left(T_{R, \delta, Q}^2 (1+4 R^2)^{1/2}-\log(\delta)\right)}{\log(2)}\Bigr]_{+} \right)\right)\right)
    \end{equation*}
    \item[(iii)] \textbf{\emph{Number of non-zero parameters: The number of non-zero parameters in $\mathcal{NN}$ is at most}}
    \[
    \begin{aligned}
        \mathcal{O}\Bigl( T_{R, \delta, Q} (P+Q)^2 \left(1+
        (P+Q) \sqrt{T_{R, \delta, Q} \log(T_{R, \delta, Q})} \left(1+\frac{\log(2)}{\log(T_{R, \delta, Q})}\Bigl[ C_d + \frac{\log\left(T_{R, \delta, Q}^2 (1+4R^2)^{1/2}-\log(\delta)\right)}{\log(2)}\Bigr]_{+}\right)\right)\Bigr).
    \end{aligned}
    \]
    The  ``dimensional constant" $C_d>0$ is defined by $$C_d  \eqdef  \frac{2 \log(5 \sqrt{2\pi}) + 3 \log(P+Q)-\log(P+Q+1)}{2 \log(2)}$$.
\end{itemize}
At this point, define the map $\hat{h}:\mathbb{R}^{P+Q}\rightarrow \mathbb{R}^{P+Q}$ by 
\begin{equation*}
    \hat{h}  \eqdef   \tilde{h} \circ L_2
\end{equation*}
where $L_2:\mathbb{R}^{P+Q}\rightarrow \mathbb{R}^{P+Q}$ maps any $(\vartheta,z)\in \mathbb{R}^{P+Q}$ to $(\frac1{M_{T_{\delta,R,Q}}} \vartheta, z)$.  Since every linear map is affine and the composition of affine maps are again affine then $\hat{h}$ is itself a deep ReLU network with depth, width, and number of non-zero parameters equal to that of $\tilde{h}$, respectively. Define the linear map $L_1:\mathbb{R}^{P+Q}\rightarrow \mathbb{R}^P$ as sending any $(\vartheta,z)\in \mathbb{R}^{P}\times \mathbb{R}^{Q}$ to  $M_{\delta,R,Q}\vartheta$.  By construction we have that: for every $t=0,\dots,T_{R,\delta,Q}-1$
\begin{equation*}
    \theta_{t+1} = L_1\circ \hat{h}(z_t),
\end{equation*}
for every $t=0,\dots,T_{R,\delta,Q}$.  Setting $R \eqdef 1$ and $T \eqdef T_{R,\delta,Q}$ we conclude.

\end{proof}

\subsection{Proof of Theorem~\ref{thm:theorem_universality_causal}}\label{app:proof_of_theorem_universality_causal}

{The proof of Theorem~\ref{thm:theorem_universality_causal} proceeds as follows. We first independently apply Theorem~\ref{thm:theorem_universality_static} $T+1$ times—once for each time point $t = 0, \dots, T$—to obtain a sequence of neural filters, for a suitable time horizon $T \in \mathbb{N}_+$.
\hfill\\
Each of these neural filters is determined by a corresponding sequence of parameter vectors $\theta_0, \dots, \theta_T$, which we aim to link recursively via a hypernetwork. To this end, we apply Lemma~\ref{lemma:deterministic_weaving} to the augmented parameter vectors $z_0 = (\theta_0, \tilde{z}_0), \dots, z_T = (\theta_T, \tilde{z}_T)$, where $\{\tilde{z}_t\}_{t=0}^T$ is a $\delta$-packing of a high-dimensional sphere as described before Lemma~\ref{lemma:deterministic_weaving}.
As a result, we obtain a ReLU MLP memorizer which, at any time point $t$, takes $z_t$ as input and returns $z_{t+1} = (\theta_{t+1}, \tilde{z}_{t+1})$. Given $z_{t+1}$, we project off the auxiliary component $\tilde{z}_{t+1}$—which is used solely to ensure separation—and use the updated parameter vector $\theta_{t+1}$ (output by the memorizing ReLU MLP) in our neural filter model to predict at time $t+1$. Controlling the resulting errors completes the proof.}

\noindent   {We now prove Theorem~\ref{thm:theorem_universality_causal}.  First, we} introduce the following ``zero-padding" notation, where $A \oplus B$ denotes the direct sum between two matrices $A$ and $B$. For any $k, s \in \mathbb{N}_{+}$, we denote by $0_{k, s}$ the $k \times s$ zero-matrix and by $0_{k}$ the column zero-vector in $\mathbb{R}^{k}$. Instead, for any non-positive integers $k, s$ we define $A \oplus 0_{k, s} \eqdef A$, for any matrix $A$, and $b \oplus 0_k \eqdef b $, for any vector column vector $b$. As in Theorem \ref{thm:theorem_universality_static}, we will detail the proof for the case that $f$ is $(r, k, \lambda)$-smooth; the case in which $f$ is $(r, \alpha, \lambda)$-H\"{o}lder is analogous.\\   
\noindent Let $\varepsilon_A>0$ be a given ``approximation error" and a ``time horizon'' $I\in \mathbb{N}_+$ satisfying $I\le \lfloor \delta^{-Q}\rfloor$. By assumption, $f\,:\,\mathcal{X}\rightarrow\mathcal{Y}$ is $(r, k, \lambda)$-smooth, $\mathcal{X}$ is compact and $\mathcal{Y}$ is linear\footnote{See Definition \ref{def:causal_maps}.}. Therefore, there exists $M$ 
% and $I \in \mathbb{N}_{+}$ 
such that for every $i \in [[I]]$ there is a $f_{t_i} \in C_{\text{tr}}^{k, \lambda}(\mathcal{X}_{(t_{i-M, t_i]}}, B_{t_i})$ which satisfies the following inequality: 
\begin{equation}\label{eq:bound_due_to_the_definition}
    \max_{i \in [[I]]}\sup_{x \in \mathcal{X}}d_{B_{t_i}}(f_{t_i}(x_{(t_{i-M}, t_i]}), f(x)_{t_i}) < \frac{\varepsilon_A}{2},
\end{equation}
where $M=M(\varepsilon_A,I) = O(\varepsilon_A^{-r})$. Now, for every $i \in [[I]]$, for a fixed ``encoding error" $\varepsilon_{D}>0$ (and ``approximation error" $\varepsilon_{A}$), Theorem \ref{thm:theorem_universality_static} ensures the existence of a neural filter\footnote{See Definition~\ref{def:Neural_Filter}.} $\hat{f}_{t_i} \in \mathcal{NF}_{[n_{\varepsilon_D}]}^{\text{(P)ReLU}}$ satisfying to the following uniform estimates 
\[
\max_{i \in [[I]]}\sup_{u \in \mathcal{X}_{(t_{i-M}, t_{i}]}} d_{B_{t_i}}(f_{t_i}(u), \hat{f}_{t_i}(u)) < \varepsilon_{D} + \frac{\varepsilon_{A}}{2}.
\]
and hence
\begin{equation}\label{eq:thmuniversalityboundone}
\max_{i \in [[I]]}\sup_{x \in \mathcal{X}} d_{B_{t_i}}(f_{t_i}(x_{(t_{i-M}, t_i]}), \hat{f}_{t_i}(x_{(t_{i-M}, t_i]})) < \varepsilon_{D} + \frac{\varepsilon_{A}}{2}.
\end{equation}

\noindent Moreover, the ``model complexity" of each $\hat{f}_{\theta_{t_i}}$\footnote{Refer to equation \eqref{eq:neuralfilter}} is reported in Table \ref{tab:Model_Complexity}. In particular, for $i \in [[I]]$, let $[d^{(i)}] \eqdef (d_{0}^{(i)}, \ldots, d_{J_i}^{(i)})$ be the complexity of  $\hat{f}_{\theta_{t_i}}$, and let $J^{\star, I}$ be the maximum depth of the networks $\{\hat{f}_{\theta_{t_i}}\}_{i=1}^{I}$, i.e. $J^{\star, I} \eqdef \max_{i \in [[I]]}J_{i}$. 
In addition, for each $j \in [[J^{\star, I}]]$, set 
\[
[[I]]_j \eqdef \{
i\in[[I]];\; d^{(i)}_j 
\mbox{ and } j\le J_i
\}
\]
and let $d_{j}^{\star}$ be the maximum width among the $j^{th}$ layers, i.e. $d_{j}^{\star} \eqdef \max_{i \in [[I]]_j} d_{j}^{(i)}$.

Define $A\oplus 0_0\eqdef A$ for any matrix $A$.  
Finally, let $[d^{\star}] \eqdef (d_{0}^{\star}, \ldots, d_{J^{\star, I}}^{\star})$. Now, for each $i \in [[I]]$ and $j \in [[d^{\star}_{J^{\star,I}}]]$ we define: 
\begin{equation*}
    \begin{aligned}
\tilde{A}_j^{(i)}
 \eqdef  &
\begin{cases}
A_j^{(i)} \oplus 0_{
(d_{j+1}^{\star}-d_{j+1}^{(i)})
\times
(d_{j}^{\star}-d_j^{(i)})}
& : \mbox{if } j\le J_{(i)}
        \\
    I_{d_j^{\star}\times d_{j}^{\star}} 
    \oplus 
    0_{(d_{j+1}^{\star}-d_{j}^{\star})\times d_{j}^{\star}} 
        & : \mbox{if } J_{(i)}<j \le J^{\star,I}
    ,
    \end{cases}
    \\
    \tilde{b}_j^{(i)}
         \eqdef  &
    \begin{cases}
    b_j^{(i)} \oplus 
    0_{(d_{j+1}^{\star}-d_{j+1}^{(i)})}
    &\quad\quad\quad\quad\,\,: \mbox{if } j\le J_{(i)}\\
    0_{d_{j+1}^{\star}} &\quad\quad\quad\quad\,\,:\mbox{if } J_{(i)}<j \le J^{\star,I}
    \end{cases}\\
    \alpha_j^{(i)}  \eqdef  &
    \begin{cases}
    0 &\quad\quad\quad\quad\quad\quad\quad\quad\quad\quad\quad\,\,\,\,:\mbox{if}\,\,j \le J_{(i)}\\
    1 &\quad\quad\quad\quad\quad\quad\quad\quad\quad\quad\quad\,\,\,\,:\mbox{if}\,\,J_{(i)} < j \le J^{\star,I}.
    \end{cases}
    \end{aligned}
\end{equation*}
\noindent   In particular, with the previous definition we ensure that each matrix $\tilde{A}_j^{(i)}$ is $d_{j+1}^{\star} \times d_{j}^{\star}$-dimensional, instead of being $d_{j+1}^{(i)} \times d_{j}^{(i)}$-dimensional. Now, for every $i \in [[I]]$ we define $\theta^{\star}_{t_i}$ by $\theta^{\star}_{t_i} \eqdef (\tilde{A}^{(i)}_{j}, \tilde{b}^{(i)}_{j}, \alpha^{(i)}_{j})_{j=0}^{J^{\star,I}}$. Instead, for every $i>I$ we set $\theta^{\star}_{t_i} \eqdef \theta^{\star}_{t_I}$. Notice that by construction
\begin{equation}\label{eq:waving_lemma_applied_one}
(\hat{f}_{\theta^{\star}_{t_i}})_{i \in \mathbb{N}_+} 
= (\hat{f}_{\theta_{t_i}})_{i \in  \mathbb{N}_+} 
\end{equation}
is a sequence in $\mathcal{NN}_{[d^{\star}]}^{\text{ReLU}}$. We therefore apply Lemma \ref{lemma:deterministic_weaving}. In particular, for every there is a (P)ReLU FFNN $\hat{h}\,:\,\mathbb{R}^{P([d^{\star}])+Q} \rightarrow \mathbb{R}^{P([d^{\star}])+Q}$, with $P([d^{\star}]) \eqdef \sum_{j=0}^{J^{\star, I}-1}d_{j}^{\star}(d_{j+1}^{\star}+2)+d_{J^{\star,I}} \geq 1$, an ``initial latent code" $z \in \mathbb{R}^{P([d^{\star}])+Q}$, and a linear map $L\,:\,\mathbb{R}^{P([d^{\star}])+Q} \rightarrow \mathbb{R}^{P([d^{\star}])}$ satisfying
\begin{equation}\label{eq:waving_lemma_applied_two}
    \begin{split}
        \hat{f}_{L(z_{t_i})} &= \hat{f}_{\theta^{\star}_{t_i}}\\
        z_{t_{i+1}} &= \hat{h}(z_{t_i})
    \end{split}
\end{equation}
\noindent for every ``time" $i = 1, \ldots, I_{\delta,Q} -1$, where $I_{\delta,Q}\eqdef \lfloor \delta^{-Q}\rfloor$.

The depth and the width of the network are provided by the same lemma with $T_{\delta, Q} \eqdef I_{\delta, Q}$. Equations \eqref{eq:waving_lemma_applied_one} and \eqref{eq:waving_lemma_applied_two} imply that 
\begin{equation}\label{eq:waving_lemma_applied_three}
    \begin{split}
        \hat{f}_{L(z_{t_i})} &= \hat{f}_{\theta_{t_i}}\\
        z_{t_{i+1}} &= \hat{h}(z_{t_i})
    \end{split}
\end{equation}
for every $i \in [[I]]$. At this point, combining Equations \eqref{eq:bound_due_to_the_definition} and \eqref{eq:thmuniversalityboundone}, we have:
\allowdisplaybreaks
\begin{align*}
    \max_{i \in [[I]]}\sup_{x \in \mathcal{X} } d_{B_{t_i}}( \hat{f}_{t_i}(x_{(t_{i-M}, t_i]}), f(x)_{t_i}) 
\le & 
    \,
    \max_{i \in [[I]]}\sup_{x \in \mathcal{X} } d_{B_{t_i}}(f_{t_i}(x_{(t_{i-M}, t_i]}), f(x)_{t_i}) 
\\
& +
    \max_{i \in [[I]]}\sup_{x \in \mathcal{X} } d_{B_{t_i}}(f_{t_i}(x_{(t_{i-M}, t_i]}), \hat{f}_{t_i}(x_{(t_{i-M}, t_i]}) ) 
\\
<& \,
    \frac{\varepsilon_A}{2} + \varepsilon_{D} + \frac{\varepsilon_A}{2}  
\\
= & \,
    \varepsilon_A + \varepsilon_D,
\end{align*}
which concludes the proof.

\section{Technical Lemmata}\label{app:technical_lemmata}
\begin{lemma}\label{lemma: product of Schauder basis}
Let $(E, (p_\ell)_{\ell=1}^\infty, (e_k)_{k=1}^\infty)$ (respectively $(F, (q_m)_{m=1}^\infty, (f_k)_{k=1}^\infty)$) be a Fr\'echet space with seminorms $(p_\ell)_\ell$ (respectively $(q_m)_m$) and Schauder basis $(e_k)_k$ (respectively $(f_k)_k$). Then the Cartesian product $$G=E\times F$$ endowed with the product topology is still a Fr\'echet space carrying a Schauder basis: a choice for this one is provided by $(b_t)_{t=1}^\infty\subset G$, where
\[
\begin{cases}
    b_{2t-1} \eqdef  (e_t,0),\quad t=1,2,\dots\\
    b_{2t} \eqdef  (0,f_t),\quad t=1,2,\dots
\end{cases}
\]

\end{lemma}
\begin{proof}
From elementary results from functional analysis and topology, it is clear that $G$ endowed with the product topology is a topological vector space. This topology can be induced also by a metric, e.g.
\[
d:G\times G\to [0,\infty) 
\]
\[
d((e,f),(e',f')) \eqdef  d_E(e,e') + d_F(f,f'),\quad (e,f),(e',f')\in G,
\]
where $d_E$ (respectively $d_F$) is a compatible metric for $E$ (respectively $F$). Evidently, $(G,d)$ is also complete. This topology is locally convex because it can be induced by the following countable collection of seminorms
\[
\gamma_{\ell,m}(e,f)  \eqdef  p_\ell(e) + q_m(f),\quad \ell,m \in \mathbb{N}_+,\,e\in E, f\in F.
\]
Define the following elements of $G$:
\[
\begin{cases}
    b_{2t-1} \eqdef  (e_t,0),\quad t=1,2,\dots\\
    b_{2t} \eqdef  (0,f_t),\quad t=1,2,\dots
\end{cases}
\]
We claim that $(b_t)_{t=1}^\infty$ is a Schauder basis for $G$. Indeed, let $x=(e,f)$, with 
\[
e=\sum_{k=1}^\infty \beta^E_k(e)e_k,\quad f=\sum_{k=1}^\infty \beta^F_k(f)f_k.
\]
Let $\varepsilon>0$ be arbitrary. Since $(e_k)_k$ and $(f_k)_k$ are Schauder basis, it follows that there exists $N_\varepsilon$ such that for all $N\geq N_\varepsilon$
\[
d_E\left(\sum_{k=1}^N \beta^E_k(e)e_k,e
\right)<\varepsilon/2,
\]
\[
d_F\left(\sum_{k=1}^N \beta^F_k(f)f_k,f
\right)<\varepsilon/2.
\]
Set $T_\varepsilon= 2N_\varepsilon$ and consider $T\in\N_+$ with $T\geq T_\varepsilon$. Set 
\[
x^T  \eqdef  \beta^E_1(e)b_1 + \beta^F_1(f)b_2 + \beta^E_2(e)b_3 + \beta^F_2(f)b_4+\cdots
+ u b_T\in G
\]
whereas
\[
u=
\begin{cases}
    \beta^F_{T/2}(f),\;\text{ if } T \text{ even}\\
    \beta^E_{(T+1)/2}(e),\;\text{ if } T \text{ odd}.
\end{cases}
\]
Thus, for $T$ odd, we have
\begin{equation*}
    \begin{split}
        d(x^T,x) &= 
    d_E(\beta^E_1(e)e_1+\cdots  \beta^E_{(T+1)/2}(e)e_{(T+1)/2},e)\\
    &+
    d_F(\beta^F_1(f)f_1+\cdots \beta^F_{(T-1)/2}f_{(T-1)/2},f)
    \end{split}
\end{equation*}
and, for $T$ even, 
\begin{equation*}
    \begin{split}
        d(x^T,x) &= 
    d_E(\beta^E_1(e)e_1+\cdots  \beta^E_{T/2}(e)e_{T/2},e)\\
    &+
    d_F(\beta^F_1(f)f_1+\cdots \beta^F_{T/2}f_{T/2},f).
    \end{split}
\end{equation*}
In both cases, we deduce by construction that
\[
d(x^T,x) < \varepsilon/2 + \varepsilon/2=\varepsilon, \quad T\geq T_\varepsilon,
\]
namely $x^T\to x$ as $T\to\infty$. This proves that any $x\in G$ can be written as
\begin{equation}\label{eq: expansion}
    x = \sum_{t=1}^\infty x_t b_t
\end{equation}
with  
\begin{equation}\label{eq: coefficients Schauder basis}
x_t =
\begin{cases}
    \beta^F_{t/2}(f),\;\text{ if } t \text{ even}\\
    \beta^E_{(t+1)/2}(e),\;\text{ if } t \text{ odd}.
\end{cases}
\end{equation}
In order to prove that such decomposition is unique, suppose that there exists $x\in G$ such that
\[
\sum_{t=1}^\infty x_t b_t = x = \sum_{t=1}^\infty \bar{x}_t b_t
\]
with $x_t$ defined as in \eqref{eq: coefficients Schauder basis} and with $\bar{x}_t\neq x_t$ for some $t$. Let $t_0$ be one of these coefficients, and suppose wlog that $t_0=2j$: the odd-case is similar and it will not be treated. By projecting on the factor $F$ we obtain ($\Pi_F=$canonical projection)
\[
\begin{split}
    \Pi_F \sum_{t=1}^\infty x_t b_t & = \Pi_F\sum_{t=1}^\infty \bar{x}_t b_t\\
    \sum_{t=1}^\infty x_t\Pi_F b_t & = \sum_{t=1}^\infty \bar{x}_t \Pi_F b_t\\
    \sum_{t=1}^\infty x_{2t}f_t & = \sum_{t=1}^\infty \bar{x}_{2t} f_t\\
\end{split}
\]
and $x_{2j}\neq \bar{x}_{2j}$, contradicting the fact that $(f_t)_t$ is a Schauder basis. Therefore, the expansion \eqref{eq: expansion} is unique, and this concludes the proof.
\end{proof}

\section{Additional Background Material}
\label{A:Additional_Backgound}
{\color{black} In an effort to keep our manuscript as self-contained as possible, we collects some additional background results on generalized inverses and on Fr\'{e}chet spaces.}

\subsection{Further Results on Fre\'chet Spaces}\label{s:Additional_Background__Frechet}
\indent     We now state and prove the following auxiliary lemma.

\begin{lemma}\label{lem:aux_lemma}
Let F be a separable Fr\'echet space admitting a Schauder basis $(f_k)_{k \in \mathbb{N}_+}$ and $d_F$ a metric on F compatible with the pre-existing topology (see Equation~\eqref{eq:distance_Frechet}). Fix $n \in \mathbb{N_+}$ and define on $\mathbb{R}^{n}$ the following metric: 
\begin{equation}\label{eq:metric_on_F}
    d_{F:n}(x, y) \eqdef  d_{F}\biggl(\sum_{k = 1}^{n} x_k f_k, \sum_{k = 1}^{n} y_k f_k\biggr), \quad x,y\in\rr^n.
\end{equation}
Then, the topology induced on $\mathbb{R}^{n}$ by this metric is the standard one.
\end{lemma}
\begin{proof}
First, notice that $d_{F:n}$ is a metric on $F$. This follows directly from the fact that $d_{F}$ is a metric\footnote{The only non trivial thing to prove is the identity of indiscernibles, i.e. that $d_{F:n}\left(x, y\right) = 0 \iff x = y$. But this fact follows directly from the fact that $d_{F}$ is a metric and from the definition of Schauder basis $(f_{k})_{k}$; see Subsection \ref{subsec:Frechet}.}. Now, let $x^{(J)}  \eqdef  (x_1^{(J)}, \ldots, x_n^{(J)}),\, J\in \mathbb{N}$ and $x \eqdef  (x_1, \ldots, x_n)$ such that
\begin{equation*}
   x^{(J)} \underset{J \rightarrow \infty}{\overset{d_{F:n}}{\longrightarrow}} x.
\end{equation*}
This means in particular that 
\begin{equation*}
d_{F}\left(\sum_{k=1}^{n} x_k^{(J)} f_k, \sum_{k=1}^{n} x_k f_k\right) \underset{J \rightarrow \infty}{\longrightarrow} 0,\,\,\,i.e.,\,\,\,\sum_{k=1}^{n} x_k^{(J)} f_k \underset{J \rightarrow \infty}{\overset{F}{\longrightarrow}} \sum_{k = 1}^{n} x_k f_k.
\end{equation*}
Now, let $(\beta_k^{F})_{k \leq n}$ be the unique sequence in the topological dual of $F$, say $F'$, such that each $f \in F$ has the following representation $f = \sum_{k = 1}^{\infty} \langle \beta_k^{F}, f \rangle f_k$. Because $(\beta_k^{F})_{k \leq n}$ are continuous and linear, we clearly get that $x_k^{(J)}\underset{J \rightarrow \infty}{\longrightarrow} x_k$ for each $k \in [[n]]$. This implies that
\begin{equation*}
    \left[\sum_{k=1}^{n} |x_k^{(J)} - x_k|^2\right]^{1/2}
    \underset{J \rightarrow \infty}{\longrightarrow} 0,\,\,\,i.e.\,\,\,x^{(J)}  \underset{J \rightarrow \infty}{\overset{\|\,\cdot\,\|_2}{\longrightarrow}} x.
\end{equation*}
\noindent Vice-versa, let $x^{(J)}  \eqdef  (x_1^{(J)}, \ldots, x_n^{(J)})$ and $x \eqdef  (x_1, \ldots, x_n)$ such that $x^{(J)}  \underset{J \rightarrow \infty}{\overset{\|\cdot\|_2}{\longrightarrow}} x$. This implies that $\sum_{k=1}^{n} |x_k^{(J)} - x_k|  \underset{J \rightarrow \infty}{\longrightarrow} 0$. 
We pick an arbitrary continuous seminorm $p\in\mathcal{P}$. It holds for all $(t_1, \ldots, t_n) \in \mathbb{R}^{n}$ that
\begin{equation*}
    \begin{split}
        p\left(\sum_{k=1}^{n} t_k f_k\right) \leq \sum_{k=1}^{n} |t_k| p(f_k) \leq \max_{k = 1, \ldots, n} p(f_k) \sum_{k=1}^{n} |t_k|.
    \end{split}
\end{equation*}
This shows that
\begin{equation*}
    p\left(\sum_{k=1}^{n} x_k^{(J)} f_k - \sum_{k=1}^{n} x_k f_k\right)  \underset{J \rightarrow \infty}{\longrightarrow} 0
\end{equation*}
for all $p\in\mathcal{P}$. This means in particular that 
\begin{equation*}
d_{F}\left(\sum_{k=1}^{n} x_k^{(J)} f_k, \sum_{k=1}^{n} x_k f_k\right) \underset{J \rightarrow \infty}{\longrightarrow} 0,\,\,\,i.e.,\,\,\,d_{F:n}(x^{(J)}, x) \underset{J \rightarrow \infty}{\rightarrow} 0.
\end{equation*}
Since the metric spaces $(\mathbb{R}^{n}, d_{F:n})$ and  $(\mathbb{R}^{n},\|\cdot\|_2)$ enjoy the same converging sequences, the topology must be the same.
\end{proof}

\subsection{Generalized inverses}\label{subsec:generalized_inverse}
\cite{EH2013MMOR} wrote a thorough paper about generalized inverses and their properties. Analogously to \cite{EH2013MMOR}, we understand \textit{increasing} in the weak sense, that is, $T\,:\,\mathbb{R}\rightarrow\mathbb{R}$ is \textit{increasing} if $T(x) \leq T(y)$ for all $x < y$. Also, we remind the notion of an inverse for such functions.
\begin{definition}[Generalized Inverse]
\label{def:generalized_inverse}
    For an increasing function $T\,:\,\mathbb{R}\rightarrow\mathbb{R}$ with $T(-\infty) \eqdef \lim_{x \downarrow -\infty} T(x)$ and $T(\infty) \eqdef \lim_{x \uparrow \infty}T(x)$, the generalized inverse $T^{-}\,:\,\mathbb{R}\rightarrow\bar{\mathbb{R}}=[-\infty, \infty]$ of $T$ is defined by
\begin{equation*}
    T^{-}(y)  \eqdef  \inf\{x \in \mathbb{R}\,:\,T(x) \geq y\},\quad y \in \mathbb{R}, 
\end{equation*}
with the convention that $\inf\,\emptyset = \infty$. 
\end{definition}

To keep our manuscript self-contained, we list some properties of generalized inverses which can be found in (\cite{EH2013MMOR}, cfr. Proposition 1).  We denote the \textit{range} of a map $T:\mathbb{R}\rightarrow \mathbb{R}$ by $\operatorname{ran}\,T \eqdef  \{T(x):\,x\in \mathbb{R}\}$.
\begin{proposition}[Properties of Generalized Inverses]
\label{prop:generalized_inverse_properties}
    Let $T$ be as in Definition~\ref{def:generalized_inverse} and let $x, y \in \mathbb{R}$. Then,
    \begin{enumerate}
        \item[(\,1\,)] $T^{-}(y) = -\infty$ if and only if $T(x) \geq y$ for all $x \in \mathbb{R}$. Similarly, $T^{-}(y) = \infty$ if and only if $T(x) < y$ for all $x \in \mathbb{R}$.
        \item[(\,2\,)]  $T^{-}$ is increasing. If $T^{-}(y) \in (-\infty, \infty)$, $T^{-}$ is left-continuous at $y$ and admits a limit from the right at $y$.
        \item[(\,3\,)]  $T^{-}(T(x)) \leq x$. If $T$ is strictly increasing, $T^{-}(T(x)) = x$.
        \item[(\,4\,)]  Let $T$ be right-continuous. Then $T^{-}(y) < \infty$ implies $T(T^{-}(y)) \geq y$. Furthermore, $y \in \operatorname{ran}\,T \bigcup \{\inf\,\operatorname{ran}\,T, \sup\,\operatorname{ran}\,T\}$ implies $T(T^{-}(y))=y$. Moreover, if $y < \inf \operatorname{ran}\,T$ then $T(T^{-}(y)) > y$ and if $y > \sup \operatorname{ran}\,T$ then $T(T^{-}(y)) < y$.
    \end{enumerate}
\end{proposition}
%%%

%%% Bibliography
\bibliographystyle{plain}

\end{document}